\documentclass[11 pt]{amsart}
\usepackage{latexsym,amscd,amssymb, graphicx, amsthm, mathtools}  
\usepackage{blkarray}
\usepackage{tikz-cd, xcolor}

\usepackage[margin=1in]{geometry}

\numberwithin{equation}{section}

\newtheorem{theorem}{Theorem}[section]
\newtheorem{proposition}[theorem]{Proposition}
\newtheorem{corollary}[theorem]{Corollary}
\newtheorem{lemma}[theorem]{Lemma}

\newtheorem{observation}[theorem]{Observation}

\newtheorem{problem}[theorem]{Problem}
\newtheorem{example}[theorem]{Example}
\newtheorem{remark}[theorem]{Remark}

\definecolor{2purple}{RGB}{204,102,255}
\definecolor{3green}{RGB}{0,204,0}

\newtheorem{defn}[theorem]{Definition}
\theoremstyle{definition}

\newcommand{\sign}{{\mathrm {sign}}}

\newcommand{\GL}{{\mathrm {GL}}}
\newcommand{\odd}{{\mathrm {odd}}}
\newcommand{\Stir}{{\mathrm {Stir}}}
\newcommand{\Hilb}{{\mathrm {Hilb}}}

\newcommand{\Frob}{{\mathrm {Frob}}}

\newcommand{\ANC}{{\mathrm {ANC}}}

\newcommand{\symm}{{\mathfrak{S}}}

\newcommand{\CC}{{\mathbb {C}}}
\newcommand{\QQ}{{\mathbb {Q}}}
\newcommand{\ZZ}{{\mathbb {Z}}}

\newcommand{\BBB}{{\mathcal{B}}}

\newcommand{\xx}{{\mathbf {x}}}
\newcommand{\Nar}{{\mathrm {Nar}}}
\newcommand{\Cat}{{\mathrm {Cat}}}
\newcommand{\NC}{{\mathrm {NC}}}

\begin{document}

\title[Set partitions, fermions, and skein relations]
{Set partitions, fermions, and skein relations}

\author{Jesse Kim}
\author{Brendon Rhoades}
\address
{Department of Mathematics \newline \indent
University of California, San Diego \newline \indent
La Jolla, CA, 92093-0112, USA}
\email{(jvkim, bprhoades)@math.ucsd.edu}

\begin{abstract}
Let $\Theta_n = (\theta_1, \dots, \theta_n)$ and $\Xi_n = (\xi_1, \dots, \xi_n)$ be two lists of $n$
 variables and consider the diagonal action of $\symm_n$ on the exterior algebra 
$\wedge \{ \Theta_n, \Xi_n \}$ generated by these variables. Jongwon Kim and 
the second author defined
and studied
the {\em fermionic diagonal coinvariant ring}  $FDR_n$ obtained from 
$\wedge \{ \Theta_n, \Xi_n \}$ by modding out by the $\symm_n$-invariants with vanishing constant term.
On the other hand, the second author described an action of $\symm_n$ on 
the vector space with basis given by noncrossing
set partitions of $\{1,\dots,n\}$ using a 
novel family of skein relations which resolve crossings in set partitions.
We give an isomorphism between a natural Catalan-dimensional submodule of $FDR_n$
and the skein representation.
To do this, we show that  set partition skein relations arise naturally in the context of exterior algebras.
Our approach yields an $\symm_n$-equivariant way to resolve crossings in set partitions.
We use fermions to clarify, sharpen, and extend the theory of set partition crossing resolution.
\end{abstract}

\maketitle

\section{Introduction}
\label{Introduction}

This paper concerns two modules over the symmetric group $\symm_n$.
The first is combinatorial, involving skein relations which resolve crossings in set partitions of 
$[n] := \{1, \dots, n \}$. The second is algebraic, arising from the ring of fermionic diagonal coinvariants.
We describe the combinatorial module first.

A set partition $\pi$ of $[n]$ is {\em noncrossing} if whenever $1 \leq a < b < c < d \leq n$
are four indices such that $a \sim c$ and $b \sim d$ in $\pi$, we have $a \sim b \sim c \sim d$ in $\pi$.
Drawing the indices $1, 2, \dots, n$ around a circle, this means that the convex hulls of the blocks 
of $\pi$ do not intersect. A noncrossing and `crossing' partition when $n = 6$ are shown below.
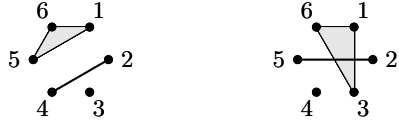
\begin{center}
\begin{tikzpicture}[scale = 0.5]

   \newcommand{\squa}
        {
        \foreach \i in {1,...,6}
                {
                \coordinate (p\i) at (-60*\i + 120:\r);
                \filldraw(p\i) circle (\pointradius pt);
                \node[] at (-60*\i + 120:1.5) {\footnotesize $\i$};
                }
        }
            
    \def\r{1}           % pentagon radius
    \def\pointradius{3} % points radius

       \foreach \a/\b/\c/\d/\e/\f [count=\shft from 0] in   
       {1/2/3/4/5/6}
        {
        \begin{scope}[xshift=0+100*\shft,yshift=0]
            \squa
            \foreach \i in {1,...,6}
                  \draw[thick] (p\a) -- (p\f) -- (p\e) -- (p\a) -- cycle;
		 \fill[black!10] (p\a) -- (p\f) -- (p\e) -- (p\a) -- cycle;
                \draw[thick] (p\b) -- (p\d);
            \squa
        \end{scope}}
        
             \foreach \a/\b/\c/\d/\e/\f [count=\shft from 0] in   
       {1/2/3/4/5/6}
        {
        \begin{scope}[xshift=200+100*\shft,yshift=0]
            \squa
            \foreach \i in {1,...,6}
                  \draw[thick] (p\a) -- (p\c) -- (p\f) -- (p\a) -- cycle;
		 \fill[black!10] (p\a) -- (p\c) -- (p\f) -- (p\a) -- cycle;
                \draw[thick] (p\e) -- (p\b);
            \squa
        \end{scope}
        }

\end{tikzpicture}
\end{center}

We write $\NC(n)$ for the family of noncrossing partitions of $[n]$ and $\NC(n,k)$ for the subfamily of 
noncrossing partitions of $[n]$ with $k$ blocks.
These sets are counted by the {\em Catalan} and {\em Narayana numbers}
\begin{equation}
|\NC(n)| = \Cat(n) = \frac{1}{n+1}{2n \choose n} \quad \quad  \quad
|\NC(n,k)| = \Nar(n,k) = \frac{1}{n}{n \choose k}{n \choose k-1}
\end{equation}

The family $\Pi(n)$ of all set partitions of $[n]$ (noncrossing or otherwise) carries a natural action of 
the symmetric group $\symm_n$.
Given $w \in \symm_n$ and $\pi \in \Pi(n)$, let $w(\pi) \in \Pi(n)$ be the set partition whose blocks are
$w(B) = \{ w(i) \,:\, i \in B \}$ where $B$ is a block of  $\pi$. 
Although the subset $\Pi(n,k) \subseteq \Pi(n)$ of $k$-block set partitions of $[n]$ is stable under this action
of $\symm_n$ for any $k \leq n$, 
this action of $\symm_n$ {\bf does not} preserve the noncrossing property: the sets 
$\NC(n)$ and $\NC(n,k)$ are not closed under this action.
Despite this, Rhoades introduced \cite{Rhoades} an action of $\symm_n$ on the linearized versions
$\CC[\NC(n)]$ and $\CC[\NC(n,k)]$ of these sets\footnote{We work over $\CC$ for convenience,
but all of the results in this paper hold over $\QQ$.}. 
We use a modified version of this 
action sketched as follows
(for a precise formulation see Definition~\ref{skein-action-definition}).

For $1 \leq i \leq n-1$, let $s_i := (i,i+1) \in \symm_n$ be the corresponding adjacent transposition.
If $\pi \in \NC(n)$ is a noncrossing partition, the partition $s_i(\pi) \in \Pi(n)$ may or may not be noncrossing.
If $s_i(\pi)$ is noncrossing, we set $s_i \cdot \pi := - s_i(\pi)$.
If $s_i(\pi)$ is not noncrossing, we resolve the local crossing at $i, i+1$ using the skein relations
shown in Figure~\ref{fig:skein}.
These relations come in three flavors, depending on whether
 of the blocks of $\pi$ being crossed at $i$ and
$i+1$ have exactly two or more than two elements.
The top skein relation is the famous transformation
\begin{center}

\begin{tikzpicture}[scale=0.5]

    \newcommand{\squa}
        {
        \foreach \i in {1,...,4}
                {
                \coordinate (p\i) at (90*\i + 45:\r);
                \filldraw(p\i) circle (\pointradius pt);
                }
        }
            
    \def\r{1}           % pentagon radius
    \def\pointradius{3} % points radius

    \foreach \a/\b/\c/\d [count=\shft from 0] in   
       {1/3/2/4,
        1/2/3/4,
        1/4/2/3}
        {
        \begin{scope}[xshift=100*\shft,yshift=0]
            \squa
            \foreach \i in {1,...,4}
                \draw[thick] (p\a) -- (p\b);
                \draw[thick] (p\c) -- (p\d);
            \squa
        \end{scope}
        }

	\node at (1.8,0) {$\mapsto$};
	\node at (5.2,0) {$+$};

    \end{tikzpicture}
    
    \end{center}
    which appears in invariant theory, knot theory, and elsewhere.
    The lower two skein relations are less classical; to the knowledge of the authors they were not 
    studied prior to \cite{Rhoades}.
    The 2-term and 3-term skein relations are `degenerations' of the 4-term
    skein relation in which one omits terms involving singleton blocks.
    We will make this more precise by means of certain `block operators'; see the proof of 
    Theorem~\ref{operator-theorem}.

   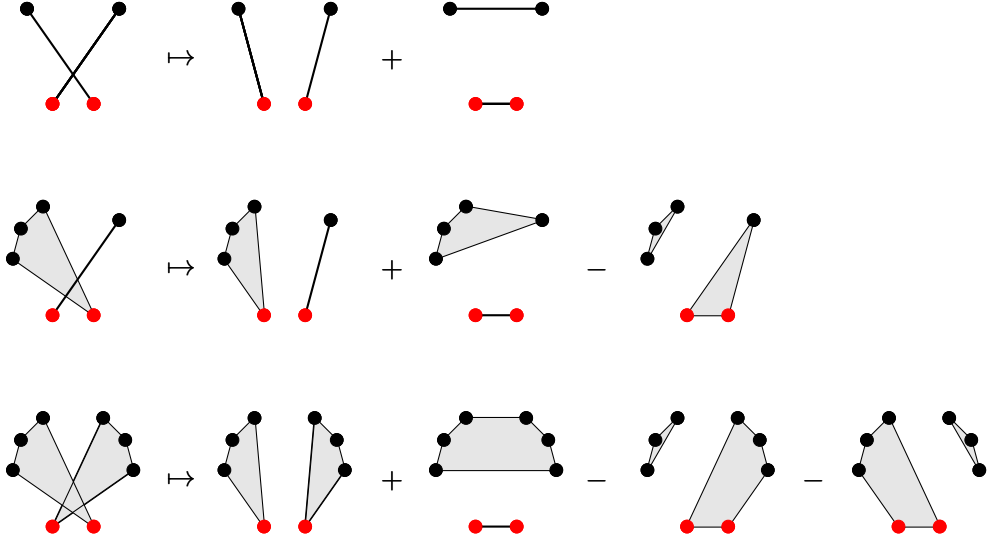
\begin{figure}

\begin{tikzpicture}[scale=0.8]

    \newcommand{\oneplusone}
        {
        \foreach \i in {1,2}
                {
                \coordinate (p\i) at (40*\i + 210:\r);
                \filldraw(p\i)[red] circle (\pointradius pt);
                }
         \foreach \j in {3,4}
                {
                \coordinate (p\j) at (100*\j + 100 :\r);
                \filldraw(p\j) circle (\pointradius pt);
                }
        }
            
    \def\r{1}           % pentagon radius
    \def\pointradius{3} % points radius

        \foreach \a/\b/\c/\d [count=\shft from 0] in   
       {1/3/2/4,
        1/4/2/3,
        1/2/3/4}
        {
        \begin{scope}[xshift=100*\shft,yshift=0]
            \oneplusone
            \foreach \i in {1,...,4}
                \draw[thick] (p\a) -- (p\b);
                \draw[thick] (p\c) -- (p\d);
                
                \oneplusone
            
        \end{scope}
        }

    \oneplusone

    \newcommand{\oneplussome}
        {
        \foreach \i in {1,2}
                {
                \coordinate (p\i) at (40*\i + 210:\r);
                \filldraw(p\i)[red] circle (\pointradius pt);
                }
         \foreach \j in {3}
                {
                \coordinate (p\j) at (100*\j + 100 :\r);
                \filldraw(p\j) circle (\pointradius pt);
                }
          \foreach \k in {4,5,6}
                {
                \coordinate (p\k) at (30*\k + 360 :\r);
                \filldraw(p\k) circle (\pointradius pt);
                }
        }
            
    \def\r{1}           % pentagon radius
    \def\pointradius{3} % points radius
    
     \foreach \a/\b/\c/\d/\e/\f [count=\shft from 0] in   
       {1/2/3/4/5/6,
        2/1/3/4/5/6,
        1/3/2/4/5/6}
        {
        \begin{scope}[xshift=100*\shft,yshift=-100]
           	 \oneplussome
	 	
               	 \draw[thick] (p\b) -- (p\d) -- (p\e) -- (p\f) -- cycle;
		 \fill[black!10] (p\b) -- (p\d) -- (p\e) -- (p\f) -- cycle;

                \draw[thick] (p\a) -- (p\c);
                
                \oneplussome
        \end{scope}
        }

         \foreach \a/\b/\c/\d/\e/\f [count=\shft from 0] in   
       {1/2/3/4/5/6}
        {
        \begin{scope}[xshift=100*\shft + 300,yshift=-100]
           	 \oneplussome
	 	
               	 \draw[thick] (p\a) -- (p\b) -- (p\c) -- cycle;
	 
	  \fill[black!10] (p\a) -- (p\b) -- (p\c)  -- cycle;
       
                \draw[thick] (p\d) -- (p\e) -- (p\f) -- cycle;
                
                 \fill[black!10] (p\d) -- (p\e) -- (p\f) -- cycle;
                 
                 \oneplussome
        \end{scope}
        }

    \tikzset{twopurple/.style={2purple,line width=8pt,rounded corners=2pt,cap=round}}
    \tikzset{threegreen/.style={3green,line width=8pt,rounded corners=2pt,cap=round,fill}}

    \newcommand{\someplussome}
        {
        \foreach \i in {1,2}
                {
                \coordinate (p\i) at (40*\i + 210:\r);
                \filldraw(p\i)[red] circle (\pointradius pt);
                }
         \foreach \j in {3,4,5}
                {
                \coordinate (p\j) at (30*\j + 270 :\r);
                \filldraw(p\j) circle (\pointradius pt);
                }
          \foreach \k in {6,7,8}
                {
                \coordinate (p\k) at (30*\k + 300 :\r);
                \filldraw(p\k) circle (\pointradius pt);
                }
        }
            
    \def\r{1}           % pentagon radius
    \def\pointradius{3} % points radius

         \foreach \a/\b/\c/\d/\e/\f/\g/\h [count=\shft from 0] in   
       {1/2/3/4/5/6/7/8,
       2/1/3/4/5/6/7/8}
        {
        \begin{scope}[xshift=100*\shft,yshift=-200]
           	 \someplussome
	 	
               	 \draw[thick] (p\a) -- (p\c) -- (p\d) -- (p\e) -- cycle;
		 \fill[black!10] (p\a) -- (p\c) -- (p\d) -- (p\e) -- cycle;
		 
		 \draw[thick] (p\b) -- (p\f) -- (p\g) -- (p\h) -- cycle;
		 \fill[black!10] (p\b) -- (p\f) -- (p\g) -- (p\h) -- cycle;
		 
		 \draw[thin] (p\a) -- (p\c);
		  \draw[thin] (p\a) -- (p\e);
		 
		 \someplussome

        \end{scope}
        }
        
                 \foreach \a/\b/\c/\d/\e/\f/\g/\h [count=\shft from 0] in   
       {1/2/3/4/5/6/7/8}
        {
        \begin{scope}[xshift=100*\shft + 200,yshift=-200]
           	 \someplussome
	 	
               	 \draw[thick]  (p\c) -- (p\d) -- (p\e) -- (p\f) -- (p\g) -- (p\h) -- cycle;
		 \fill[black!10] (p\c) -- (p\d) -- (p\e) -- (p\f) -- (p\g) -- (p\h) -- cycle;
		 
		 \draw[thick] (p\a) -- (p\b);

		 \someplussome

        \end{scope}
        }
        
                 \foreach \a/\b/\c/\d/\e/\f/\g/\h [count=\shft from 0] in   
       {1/2/3/4/5/6/7/8,
       1/2/6/7/8/3/4/5}
        {
        \begin{scope}[xshift=100*\shft + 300,yshift=-200]
           	 \someplussome
	 	
               	 \draw[thick] (p\a) -- (p\b) -- (p\c) -- (p\d) -- (p\e) -- cycle;
		 \fill[black!10] (p\a) -- (p\b) -- (p\c) -- (p\d) -- (p\e) -- cycle;
		 
		 \draw[thick]  (p\f) -- (p\g) -- (p\h) -- cycle;
		 \fill[black!10] (p\f) -- (p\g) -- (p\h) -- cycle;
		 
		 \someplussome

        \end{scope}
        
        \node at (1.8,-0.2) {$\mapsto$};
          \node at (5.3,-0.2) {$+$};

        \node at (1.8,-3.7) {$\mapsto$};
        \node at (5.3,-3.7) {$+$};
        \node at (8.7,-3.7) {$-$};

         \node at (1.8,-7.2) {$\mapsto$};
         \node at (5.3,-7.2) {$+$};
         \node at (8.7,-7.2) {$-$};
         \node at (12.3,-7.2) {$-$};
        
        }

    \end{tikzpicture}
    
      \caption{The three skein relations defining the action of $\symm_n$ on $\CC[\NC(n)]$. The red vertices
      are adjacent indices $i, i+1$
      and the shaded blocks have at least three elements. The symmetric 3-term relation obtained
      by reflecting the middle relation across the $y$-axis is not shown.}
     \label{fig:skein}
    
 \end{figure}

    The action $s_i \cdot \pi$ described above extends to an action of $\symm_n$ on the vector space
    $\CC[\NC(n)]$.  Since the skein relations in Figure~\ref{fig:skein} preserve the total number of blocks
    in a set partition, the subspace $\CC[\NC(n,k)]$ is a submodule for this action. 
    We have further submodules $\CC[\NC(n,k,m)]$, where $\NC(n,k,m)$ is the family of $k$-block
    noncrossing set partitions of $[n]$ with $m$ singletons.
    We refer to these modules collectively as {\em skein actions} of $\symm_n$, and their canonical
    bases $\NC(n), \NC(n,k)$, and $\NC(n,k,m)$ as {\em skein bases}.
    
    The skein action was introduced to give representation theoretic proofs of cyclic sieving 
    results of Reiner-Stanton-White \cite{RSW} and Pechenik \cite{Pechenik} 
    involving the rotational action of $\ZZ_n$ on various sets of noncrossing partitions of $[n]$.
    Skein bases generalize the Kazhdan-Lusztig cellular and  $\mathfrak{sl}_2$-web bases 
   (see \cite{KungRota, PPR, RhoadesPoly, RT}) of symmetric group irreducibles labeled by 2-row rectangles.

   The skein action has nice combinatorial properties. 
   Permutations $w \in \symm_n$ have representing matrices in the skein basis with entries in
   $\ZZ$.
 `Local symmetries' of 
   noncrossing partitions are preserved:
   if $w \in \symm_n$ and $\pi \in \NC(n)$ are such that the set partition $w(\pi)$
   is noncrossing, then $w \cdot \pi = \pm w(\pi)$ (Corollary~\ref{local-symmetries}).
   If we endow $\CC[\Pi(n)]$ with a sign-twisted version of the permutation action of $\symm_n$,
   there is a $\symm_n$-equivariant projection
   \begin{equation}
   p: \CC[\Pi(n)] \twoheadrightarrow \CC[\NC(n)]
   \end{equation}
   in which $p(\pi)$ is a $\ZZ$-linear combination of noncrossing partitions for any set partition $\pi$
   (Definition~\ref{projection-map-definition}, Theorem~\ref{projection-equivariance}).
   We regard $p(\pi)$ as a `resolution of crossings' in the set partition $\pi$; this generalizes the classical 
   resolution of crossings in perfect matchings/chord diagrams.
   Before proceeding further, we issue a
   
   \begin{quote}
   \noindent
   {\bf Warning.}
   {\em The skein action used in this paper differs from that in \cite{Rhoades}.
    The fundamental relations in Figure~\ref{fig:skein} are unchanged, but the sign convention 
   for applying $s_i$ to a noncrossing set partition $\pi$ for which $s_i(\pi)$ is also noncrossing differs.
   Our conventions yield sharper results, cleaner proofs, and give connections 
   to the fermionic diagonal coinvariant ring described below.}
  \end{quote}
   
   The skein action as presented in \cite{Rhoades} had some drawbacks.
   The definition of this action was purely combinatorial and somewhat {\em ad hoc}; there 
   was little algebraic reason `why' these skein relations ought to hold.
   Checking that the action of the generators $s_i$ extended to a well-defined action of $\symm_n$
   involved extensive casework and a number of `miraculous' 16-term identities\footnote{In fact,
   the intricacy of these identities led to a couple cases which were missed in \cite{Rhoades}.
   A. Iraci \cite{Iraci} filled these gaps in his Master's Thesis at the University of Pisa.}.
   The complicated nature of this action led to difficulty in computing the sign in the local symmetry
   formulas
   $w \cdot \pi = \pm w(\pi)$ described above.
   Finally, it was unclear how to extend the skein action from $\symm_n$ to a wider 
   class of reflection groups $W$.
   In this paper we address these issues by relating the skein action to fermionic diagonal coinvariants.
   
   We turn to the algebraic module of study: the fermionic diagonal coinvariant ring.
   Let $\Theta_n = (\theta_1, \dots, \theta_n)$ and $\Xi_n = (\xi_1, \dots, \xi_n)$ be two lists of $n$
   anticommuting variables and let 
   \begin{equation}
   \wedge \{ \Theta_n, \Xi_n \} := \wedge \{ \theta_1, \dots, \theta_n, \xi_1, \dots, \xi_n \}
   \end{equation}
   be the exterior algebra generated by these symbols over $\CC$.
    The ring $\wedge \{ \Theta_n, \Xi_n \}$ has a  bigrading
    \begin{equation}
    \wedge \{ \Theta_n, \Xi_n \}_{i,j} := \wedge^i \{ \theta_1, \dots, \theta_n \} \otimes
    \wedge^j \{ \xi_1, \dots, \xi_n \}
    \end{equation}
     Adopting the language of physics,
     we refer to the variables $\theta_i, \xi_i$ as 
     {\em fermionic} and general elements $f \in \wedge \{ \Theta_n, \Xi_n \}$ as {\em fermions}\footnote{In physics,
     the equation $\theta_i^2 = 0$ is the {\em Pauli Exclusion Principle}:  two identical fermions
     cannot occupy State $i$ at the same time.}.
     
     The ring  $\wedge \{ \Theta_n, \Xi_n \} $ carries 
     a bigraded diagonal action of $\symm_n$
     via
     \begin{equation}
     w \cdot \theta_i := \theta_{w(i)} \quad \quad w \cdot \xi_i := \xi_{w(i)} \quad \quad (w \in \symm_n, \, \, 1 \leq i \leq n)
     \end{equation}
     If we let $\wedge \{ \Theta_n, \Xi_n\}^{\symm_n}_+$ be the subsapce of $\symm_n$-invariants with 
     vanishing constant term, 
     the second author and Jongwon Kim introduced \cite{KR} the {\em fermionic diagonal coinvariant ring}
     \begin{equation}
     FDR_n := \wedge \{ \Theta_n, \Xi_n \}  / \langle \wedge \{ \Theta_n, \Xi_n\}^{\symm_n}_+ \rangle
     \end{equation}
     The quotient $FDR_n$ is a bigraded $\symm_n$-module.
     
     The ring $FDR_n$ is an anticommutative version of the Garsia-Haiman 
     diagonal coinvariant ring $DR_n$
     which has an analogous definition  \cite{Haiman} involving  lists 
     $X_n = (x_1, \dots, x_n)$ and $Y_n = (y_1, \dots, y_n)$ 
     of commuting variables.
     Various authors 
     \cite{Bergeron, BRT, DIW, KR, OZ, PRR, RW, RW2, Swanson, SW, ZabrockiDelta, ZabrockiFermion}
     have considered versions of $DR_n$ involving mixtures of commuting and anticommuting variables.
     
     Kim and Rhoades describe \cite{KR} the bigraded $\symm_n$-isomorphism type of $FDR_n$ in terms
     of Kronecker products.
     In particular, the bigraded piece $(FDR_n)_{i,j}$ vanishes whenever $i + j \geq n$. If $i + j < n$ we have 
     the Frobenius image
     \begin{equation}
     \label{schur-kronecker}
     \Frob \, (FDR_n)_{i,j} = s_{(n-i,1^i)} * s_{(n-j,1^j)} - s_{(n-i-1,1^{i+1})} * s_{(n-j-1,1^{j+1})} 
     \end{equation}
     where $*$ denotes Kronecker product of Schur functions and we interpret $s_{(-1,1^n)} = 0$.
     Kim and Rhoades give a basis of $FDR_n$ indexed by a certain collection of lattice paths,
     but the combinatorics of $FDR_n$ was largely unexplored in \cite{KR}.
     
     Equation~\eqref{schur-kronecker} implies that whenever $i + j < n$ we have
     \begin{equation}
     \dim \,  (FDR_n)_{i,j} = {n-1 \choose i} {n-1 \choose j} - {n-1 \choose i+1} {n-1 \choose j+1}  
     \end{equation}
     so that for any $1 \leq  k \leq n$
     \begin{equation}
     \label{diagonal-component-dimensions}
     \dim \, (FDR_n)_{n-k,k-1} = \Nar(n,k) \quad \text{so that} \quad 
     \sum_{k = 1}^n \dim \, (FDR_n)_{n-k,k-1} = \Cat(n)
     \end{equation}
     and $FDR_n$ contains a natural Catalan-into-Narayana dimensional submodule by considering its 
     extreme bidegrees.
     We isolate this submodule as follows.
     
     \begin{defn}
     \label{fdr-bar}
     For $n \geq 0$, let $\overline{FDR_n}$ be the $\symm_n$-submodule of $FDR_n$ given by
     \begin{equation}
     \overline{FDR_n} := \bigoplus_{k = 1}^n (FDR_n)_{n-k,k-1}
     \end{equation}
     The module $FDR_n$ has dimension $\Cat(n)$ and its constituent piece 
     $(FDR_n)_{n-k,k-1}$ has dimension $\Nar(n,k)$.
     \end{defn}
     
     In this paper we establish isomorphisms (Corollary~\ref{quotient-identification}) of $\symm_n$-modules
     \begin{equation}
     \label{final-goal-isomorphisms}
     \CC[\NC(n)] \cong \overline{FDR_n} \quad \text{ and } \quad 
     \CC[\NC(n,k)] \cong (FDR_n)_{n-k,k-1}
     \end{equation}
     thus giving an algebraic model for the skein action in terms of fermionic diagonal coinvariants.
     To do this, we attach (Definition~\ref{f-and-F-operator-definitoin})
     a fermion $f_{\pi} \in \wedge \{ \Theta_n, \Xi_n \}$ to any set partition $\pi \in \Pi(n)$
     and prove (Theorem~\ref{fermion-skein-theorem})
     that the noncrossing fermions $\{ f_{\pi} \,:\, \pi \in \NC(n,k) \}$ satisfy the skein relations 
     and descend
     to a basis of $(FDR_n)_{n-k,k-1}$.
     This gives a basis (Theorem~\ref{quotient-basis-theorem})
     of $\overline{FDR_n}$ tied to the combinatorics of set partitions.
     Furthermore, the algebraic model of fermions sharpens a number of results on the skein action in \cite{Rhoades},
     as well as simplifying and clarifying their proofs.
     Finally, the methods in \cite{KR} extend naturally from $\symm_n$ to irreducible complex reflection groups $W$,
     thus giving an avenue for extending the skein action to other types.

     The rest of the paper is organized as follows.
     In {\bf Section~\ref{Background}} we give background on set partitions, exterior algebras, and
      $\symm_n$-representation theory.
      In {\bf Section~\ref{Fermion}} we define two fermions $F_{\pi}$ and $f_{\pi}$ attached to any set partition 
      $\pi$ of $[n]$ (noncrossing or otherwise); the fermions $F_{\pi}$ and $f_{\pi}$ are related 
      by a kind of differentiation. We also introduce the {\em block operators} $\rho_B$; these derivations
      of $\wedge \{ \Theta_n, \Xi_n \}$ will be useful in our proofs.
     In {\bf Section~\ref{Skein}} we define the skein action and 
     prove that the $F_{\pi}$ and  $f_{\pi}$ satisfy the skein relations
     of Figure~\ref{fig:skein}.
     {\bf Section~\ref{Basis}} studies submodules of the exterior algebra $\wedge \{ \Theta_n, \Xi_n \}$.
     We prove that the combinatorial skein action is isomorphic to the space spanned by the $F_{\pi}$
     (as well as the space spanned by the $f_{\pi}$).
     {\bf Section~\ref{Resolution}} applies the theory of fermions to resolve crossings in set partitions;
     this has a number of corollaries on the combinatorics of the skein action.
     {\bf Section~\ref{Quotient}} studies submodules of the quotient space $FDR_n$
     and proves the isomorphisms \eqref{final-goal-isomorphisms}.
     We close in {\bf Section~\ref{Conclusion}} with some open problems.

\section{Background}
\label{Background}

\subsection{Combinatorics}
We will be interested in set partitions of $[n]$ with a given number of blocks, and singletons 
will play a special role. To this end, we define three families of set partitions
\begin{align}
\Pi(n) &:= \{ \text{all set partitions of $[n]$} \} \\
\Pi(n,k) &:= \{ \text{all set partitions of $[n]$ with $k$ blocks} \} \\
\Pi(n,k,m) &:= \{ \text{all set partitions of $[n]$ with $k$ blocks and $m$ singletons} \} 
\end{align}
We also let
\begin{equation}
\NC(n) \subseteq \Pi(n) \quad \quad
\NC(n,k) \subseteq \Pi(n,k) \quad \quad
\NC(n,k,m) \subseteq \Pi(n,k,m)
\end{equation}
be the subfamilies consisting of noncrossing partitions.
The three sets $\Pi(n), \Pi(n,k),$ and $\Pi(n,k,m)$ are closed under the permutation 
action of $\symm_n$, but their subsets 
$\NC(n), \NC(n,k),$ and $\NC(n,k,m)$ are usually not.

Due to the signs involved in exterior algebras, it will  be convenient to work with 
``segmented permutations" rather than set partitions. These are defined as follows.

For $n \geq 0$, a {\em (strong) composition of $n$} is a sequence $\alpha = (\alpha_1, \dots, \alpha_k)$
of positive integers with $\alpha_1 + \cdots + \alpha_k = n$.  We write $\alpha \models n$ to denote that 
$\alpha$ is a composition of $n$ and say $\alpha$ has $k$ {\em parts}.
Formally, a {\em segmented permutation of size $n$} is a pair $(w, \alpha)$ where $w \in \symm_n$
is a permutation and $\alpha \models n$ is a composition.
We represent $(w, \alpha)$ by inserting dots just after positions $\alpha_1, \alpha_1 + \alpha_2, \dots$
in the one-line notation of $w$, yielding a figure $w[1] \cdot \ldots \cdot w[k]$.
The words $w[1], \dots, w[k]$ are the {\em segments} of $(w,\alpha)$; the length $\alpha_i$ word 
$w[i]$ will be written $w[i]_1 w[i]_2 \dots w[i]_{\alpha_i}$.
If $(w, \alpha)$ is a segmented permutation of size $n$, we let $\Pi(w,\alpha) \in \Pi(n)$ be the set partition
 whose blocks are the segments of $(w, \alpha)$.
 
 As an example of these concepts, the pair 
 \begin{equation*}
(w, \alpha) = (53672184, (3,1,2,2)) 
\end{equation*}
is a segmented permutation of size $8$,
written more succinctly as 
\begin{equation*}
w[1] \cdot w[2] \cdot w[3] \cdot w[4] = 536 \cdot 7 \cdot 21 \cdot 84
\end{equation*}
This segmented permutation
gives rise to the set partition 
\begin{equation*}
\Pi(w,\alpha) = \{ 3, \, 5, \, 6 \,  /  \, 7 \,  /  \, 1, \, 2 \,  /  \, 4 \, , \, 8 \} \in \Pi(8)
\end{equation*}
Other segmented permutations giving rise to the same set partition could be obtained by 
rearranging segments wholesale, or rearranging letters within segments.

\subsection{Exterior algebras}
As mentioned in the introduction, we use $\wedge \{ \Theta_n, \Xi_n \}$ to denote the exterior algebra
over $\CC$ generated by the $2n$ symbols $\theta_1, \dots, \theta_n, \xi_1, \dots, \xi_n$.
Given subsets $S, T \subseteq [n]$ with $S = \{s_1 < \cdots < s_a \}, T = \{ t_1 < \cdots < t_b \}$,
we let $\theta_S \cdot \xi_T \in \wedge \{ \Theta_n, \Xi_n \}$ denote the exterior monomial
\begin{equation}
\theta_S \cdot \xi_T := \theta_{s_1} \cdots \theta_{s_a} \cdot \xi_{t_1} \cdots \xi_{t_b}
\end{equation}
The set $\{ \theta_S \cdot \xi_T \,:\, S, T \subseteq [n] \}$ is a basis of $\wedge \{ \Theta_n, \Xi_n \}$.
By declaring
this basis
to be orthogonal, we obtain an inner product $\langle -, - \rangle$ on the space $\wedge \{ \Theta_n, \Xi_n \}$.

We will use a notion of {\em exterior differentiation} (or {\em contraction}).
If $\Omega_m = (\omega_1, \omega_2, \dots, \omega_m)$ is
an alphabet of fermionic variables, we define a $\CC$-bilinear
 action $\odot$ of the rank $m$ exterior algebra
$\wedge \{ \Omega_m \}$ on itself by the rule
\begin{equation}
\omega_i \odot (\omega_{j_1} \cdots \omega_{j_r}) := 
\begin{cases}
(-1)^{s-1} \omega_{j_1} \cdots \widehat{\omega_{j_s}} \cdots \omega_{j_r} & \text{if $j_s = i$} \\
0 & \text{if $s \neq j_1, \dots, j_r$}
\end{cases}
\end{equation}
whenever $1 \leq j_1, \dots, j_r \leq m$ are distinct indices.
Bilinearity yields $f \odot g \in \wedge \{ \Omega_m \}$ for any 
$f, g \in \wedge \{ \Omega_m \}$.
We apply the $\odot$-action over the size $2n$ alphabet $(\theta_1, \dots, \theta_n, \xi_1, \dots, \xi_n)$
of variables in $\wedge \{ \Theta_n, \Xi_n \}$.
We leave the following simple proposition to the reader; its second part characterizes the $\odot$-action.

\begin{proposition}
\label{exterior-basic-properties}
\begin{enumerate}
\item For any $w \in \symm_n$ and $f, g \in \wedge \{ \Theta_n, \Xi_n \}$ we have
\begin{equation*}
\langle f, g \rangle = \langle w \cdot f, w \cdot g \rangle
\end{equation*}
\item For any $f, g, h \in \wedge \{ \Theta_n, \Xi_n \}$ we have
\begin{equation*}
\langle f \cdot g, h \rangle = \langle g, f \odot h \rangle
\end{equation*}
\item Assume that $f, g \in \wedge \{ \Theta_n, \Xi_n \}$ where $f$ has homogeneous total degree $d$.
For any $1 \leq i \leq n$ we have 
\begin{equation*}
\theta_i \odot (f g) = ( \theta_i \odot f) g + (-1)^d f ( \theta_i \odot g) \quad \text{ and } \quad 
\xi_i \odot (f g) = ( \xi_i \odot f) g + (-1)^d f ( \xi_i \odot g) 
\end{equation*}
\end{enumerate}
\end{proposition}

Proposition~\ref{exterior-basic-properties} (3) is a sign-twisted version of the Leibniz rule.

\subsection{$\symm_n$-representation theory}
For $1 \leq i \leq n-1$, let $s_i := (i,i+1) \in \symm_n$ be the corresponding adjacent transposition.
We have the {\em Coxeter presentation}
of $\symm_n$ with generating set $\{s_1, s_2, \dots, s_{n-1}\}$ and relations
\begin{equation}
\begin{cases}
s_i^2 = e & 1 \leq i \leq n-1 \\
s_i s_j = s_j s_i & |i - j| > 1 \\
s_i s_{i+1} s_i = s_{i+1} s_i s_{i+1} & 1 \leq i \leq n-2
\end{cases}
\end{equation}
For any subset $X \subseteq \symm_n$,
we let 
\begin{equation}  [X]^+ := \sum_{w \in X} w \in \CC[\symm_n] \quad \text{and} \quad
[X]^- := \sum_{w \in X} \sign(w) \cdot w \in \CC[\symm_n]
\end{equation}
be the  group algebra elements which symmetrize and
antisymmetrize with respect to $X$.

A {\em partition of $n$} is a weakly decreasing sequence $\lambda = (\lambda_1 \geq \cdots \geq \lambda_k)$
of positive integers with $\lambda_1 + \cdots + \lambda_k = n$. 
We write $\lambda \vdash n$ to indicate that $\lambda$ is a partition of $n$. We endow partitions with the
{\em dominance order}  defined by $\lambda \leq \mu$
 if and only if $\lambda_1 + \cdots + \lambda_i \leq \mu_1 + \cdots + \mu_i$ for all $i$.
 
 If $\lambda = (\lambda_1, \lambda_2, \dots )$ is a partition of $n$,
we let $\symm_{\lambda}$ denote the corresponding {\em Young subgroup}
 $\symm_{\lambda} = \symm_{\lambda_1} \times \symm_{\lambda_2} \times \cdots \subseteq \symm_n$
  which permutes
 the first $\lambda_1$ letters of $[n]$, the next $\lambda_2$ letters of $[n]$, and so on separately.
 More generally, if $\pi = \{ B_1 \, / \, B_2 \, / \cdots \}$ is a {\bf set} partition of $n$, we let 
 $\symm_{\pi}$ denote the corresponding {\em parabolic subgroup}
 $\symm_{\pi} = \symm_{B_1} \times \symm_{B_2} \times \cdots \subseteq \symm_n$ 
 which permutes letters within
 the blocks of $\pi$.  Similarly, if $(w, \alpha)$ is a segmented permutation of size $n$, let 
 $\symm_{w,\alpha} \subseteq \symm_n$ be the subgroup which permutes letters within the segments;
 we have
 \begin{equation}
 \symm_{w,\alpha} = \symm_{\Pi(w,\alpha)}
 \end{equation}
 as subgroups of $\symm_n$.

 Let $\Lambda = \bigoplus_{n \geq 0} \Lambda_n$ be the ring of symmetric functions 
on the infinite variable set $\xx = (x_1, x_2, \dots )$ over the ground field $\CC$.
The $n^{th}$ graded piece $\Lambda_n$ has many interesting bases; we will only need 
the basis $\{ s_{\lambda} \,:\, \lambda \vdash n \}$ of {\em Schur functions}.

Irreducible representations of $\symm_n$ are in one-to-one correspondence with partitions 
$\lambda \vdash n$. Given a partition $\lambda$, let $S^{\lambda}$ be the corresponding
$\symm_n$-irreducible. For example, $S^{(n)}$ is the trivial representation and $S^{(1^n)}$ is the sign
representation.

If $V$ is any finite-dimensional $\symm_n$-module, there exist unique multiplicities $c_{\lambda} \geq 0$
such that $V \cong \bigoplus_{\lambda \vdash n} c_{\lambda} S^{\lambda}$. The {\em Frobenius image}
of $V$ is the symmetric function
\begin{equation}
\Frob \, V := \sum_{\lambda \vdash n} c_{\lambda} s_{\lambda} \in \Lambda_n
\end{equation}
obtained by replacing each irreducible
$S^{\lambda}$ with the corresponding Schur function $s_{\lambda}$.

If $V$ is an $\symm_n$-module and $W$ is an $\symm_m$-module, their {\em induction product}
$V \circ W$ is the $\symm_{n+m}$-module given by the induction
\begin{equation}
V \circ W := \mathrm{Ind}_{\symm_n \times \symm_m}^{\symm_{n + m}}(V \otimes W)
\end{equation}
where we embed $\symm_n \times \symm_m \subseteq \symm_{n+m}$ in the natural way and 
endow the tensor product $V \otimes W$ with the structure of an $\symm_n \times \symm_m$-module
by the rule
$(u,u') \cdot (v \otimes w) := (u \cdot v) \otimes (u' \cdot w)$ for $(u,u') \in \symm_n \times \symm_m$,
$v \in V$, and $w \in W$.
We have 
\begin{equation}
\Frob \, V \circ W \,  = \Frob \, V  \cdot \Frob \, W
\end{equation}
so that induction product corresponds to multiplication in the ring $\Lambda$.
We will also need the {\em Kronecker product} $*$ defined on each graded piece
 $\Lambda_n$ by bilinearity and the rule
\begin{equation}
\Frob \, V \otimes V' = \Frob \, V * \Frob \, V'
\end{equation}
where $V, V'$ are $\symm_n$-modules and their tensor product $V \otimes V'$ is the $\symm_n$-module
with diagonal action
$u \cdot (v \otimes v') := (u \cdot v) \otimes (u \cdot v')$ for $u \in \symm_n$, $v \in V$, and $v' \in V'$.

\section{Fermions for set partitions}
\label{Fermion}

In this section we attach two fermions $F_{\pi}$ and $f_{\pi}$ to set partitions $\pi \in \Pi(n)$.
These fermions are obtained
by applying certain operators $\rho_{B_1}, \dots, \rho_{B_k}$
indexed by the blocks $B_1, \dots, B_k$ of $\pi$ to the product $\theta_1 \cdots \theta_n$.

\subsection{Block operators $\rho$ and $\psi$, fermions $F$ and $f$}
Our key tool in defining $F_{\pi}$ and $f_{\pi}$ is a family of 
 derivations of the ring $\wedge \{ \Theta_n, \Xi_n \}$.
For
 $B \subseteq [n]$ nonempty, define the {\em block operator} 
$\rho_B: \wedge \{ \Theta_n, \Xi_ n\} \rightarrow \wedge \{ \Theta_n, \Xi_n \}$
by 
\begin{equation}
\rho_B(f) := \sum_{\substack{i, j \in B \\ i \neq j}} \xi_i \cdot (\theta_j \odot f)
\end{equation}
whenever $|B| > 1$ and 
\begin{equation}
\rho_B(f) := \xi_i \cdot (\theta_i \odot f)
\end{equation}
if $B = \{i\}$ is a singleton.  For any permutation $w \in \symm_n$ we have
\begin{equation}
\label{block-permutation}
w \cdot ( \rho_B(f) ) = \rho_{w(B)}(w \cdot f)
\end{equation}
which follows from the readily checked relation 
\begin{equation}
w \cdot (g \odot f) = (w \cdot g) \odot (w \cdot f)
\quad \quad (w \in \symm_n, \, \,  g, f \in \wedge \{ \Theta_n, \Xi_n \})
\end{equation}
A crucial property enjoyed by the block operators is as follows.

\begin{lemma}
\label{block-operators-commute}
Let $A, B \subseteq [n]$ be two nonempty subsets. The operators $\rho_A$ and 
$\rho_{B}$ commute.
\end{lemma}

\begin{proof}
The lemma reduces to the assertion that, for any fermion $f$, we have
\begin{equation}
\xi_a \cdot (\theta_{a'} \odot [ \xi_b \cdot  (\theta_{b'} \odot f)])  =
\xi_b \cdot (\theta_{b'} \odot [ \xi_a \cdot  (\theta_{a'} \odot f)])
\end{equation}
Using sign-twisted
Leibniz Rule of Proposition~\ref{exterior-basic-properties} (3) we compute
\begin{align}
\xi_a \cdot (\theta_{a'} \odot [ \xi_b \cdot  (\theta_{b'} \odot f)])  
&= - (\theta_{a'} \theta_{b'}) \odot (\xi_a \xi_b f) \\
&= - (\theta_{b'} \theta_{a'}) \odot (\xi_b \xi_a f) \\
&= \xi_b \cdot (\theta_{b'} \odot [ \xi_a \cdot  (\theta_{a'} \odot f)])
\end{align}
as required.
\end{proof}

By Lemma~\ref{block-operators-commute}, for any set partition 
$\pi = \{ B_1 \, / \, B_2 \, / \, \cdots \, / \, B_k \} \in \Pi(n)$, we have a well-defined linear operator
$\rho_{\pi}: \wedge \{ \Theta_n, \Xi_n \} \rightarrow \wedge \{ \Theta_n, \Xi_n \}$ given by
\begin{equation}
\rho_{\pi} := \rho_{B_1} \circ \rho_{B_2} \circ \cdots \circ \rho_{B_k}
\end{equation}
where
the order of  composition
is immaterial.
This facilitates the following definition.

\begin{defn}
\label{f-and-F-operator-definitoin}
Let $\pi = \{ B_1 \, / \, B_2 \, / \, \cdots \, / \, B_n \} \in \Pi(n)$ be a set partition.  We define
fermions $F_{\pi}, f_{\pi} \in \wedge \{ \Theta_n, \Xi_n \}$ by
\begin{equation*}
F_{\pi} := \rho_{\pi} (\theta_1 \theta_2 \cdots \theta_n) = 
(\rho_{B_1} \circ \rho_{B_2} \circ \cdots \circ \rho_{B_k}) (\theta_1 \theta_2 \cdots \theta_n)
\end{equation*}
and
\begin{equation*}
f_{\pi} := (\xi_1 + \xi_2 + \cdots + \xi_n) \odot F_{\pi}
\end{equation*}
\end{defn}

As an example of these objects, for $\pi = \{ 1, \, 3 \, / \, 2 \}$ we have
\begin{align*}
F_{\{1, \, 3 \, / \, 2 \}} &= \rho_{\{1, \, 3\}} \circ \rho_{\{2\}}(\theta_1 \theta_2 \theta_3) =
\rho_{\{1,3\}} (- \xi_2 \cdot \theta_1 \theta_3) = 
\xi_3 \xi_2 \theta_3 - \xi_1 \xi_2 \theta_1  \\
f_{\{1, \, 3 \, / \, 2 \}} &=
(\xi_1 + \xi_2 + \xi_3) \odot (\xi_3 \xi_2 \theta_3 - \xi_1 \xi_2 \theta_1) = 
\xi_2 \theta_3 - \xi_3 \theta_3 - \xi_2 \theta_1 + \xi_1 \theta_1
\end{align*}

The notation $F_{\pi}$ and $f_{\pi}$ is from calculus: the $f$'s are the derivatives of the $F$'s
If $\pi \in \Pi(n,k)$ has $k$ blocks, the fermion $F_{\pi}$ has bidegree $(n-k,k)$ whereas
$f_{\pi}$ has bidegree $(n-k,k-1)$. 
These fermions have similar algebraic properties. We focus mainly on the cleaner
$F_{\pi}$, but the $f_{\pi}$ will be useful in the study of $FDR_n$.

Most of our results on these fermions will hold at the level of the block operators 
$\rho_B$. 
For example, the following result describes how $\symm_n$ acts on 
the $F_{\pi}$.

\begin{proposition}
\label{prop:sn-action-on-F-and-f}
Let $\pi \in \Pi(n)$ and $w \in \symm_n$.  We have
\begin{equation*}
w \cdot F_{\pi} = \sign(w) \cdot F_{w(\pi)} \quad \quad \text{and} \quad \quad
w \cdot f_{\pi} = \sign(w) \cdot f_{w(\pi)}
\end{equation*}
\end{proposition}

\begin{proof}
Equation~\eqref{block-permutation} gives the equality of operators
\begin{equation}
\label{permutation-block-operator-identity}
w \cdot (\rho_{\pi} (-)) = \rho_{w(\pi)} (w \cdot (-))
\end{equation}
Applying both sides of Equation~\eqref{permutation-block-operator-identity} 
 to $\theta_1 \theta_2 \cdots \theta_n$ yields the statement about the $F$'s.
A further application of $(\xi_1 + \xi_2 + \cdots + \xi_n) \odot (-)$ implies the statement about the $f$'s.
\end{proof}

%Up to sign,  the $F_{\pi}$ are evaluations of the $\rho_{\pi}$.

The skein action treats singleton blocks differently from larger blocks, and we will
avoid casework with the following variant of the $\rho$-operators.
Given $B \subseteq [n]$,  define   
 $\psi_B: \wedge \{ \Theta_n, \Xi_n \} \rightarrow \wedge \{ \Theta_n, \Xi_n \}$ by
\begin{equation}
\psi_B = \begin{cases}
\rho_B & |B| > 1 \\
0 & |B| \leq 1
\end{cases}
\end{equation}
Lemma~\ref{block-operators-commute} implies
\begin{equation}
\psi_A \circ \psi_B = \psi_B \circ \psi_A   \text{ for all $A, B \subseteq [n]$}
\end{equation}
so for any set partition $\pi = \{ B_1 \, / \, B_2 \, / \, \cdots \, / \, B_k \} \in \Pi(n)$ we have a well-defined
linear operator
$\psi_{\pi} : \wedge \{ \Theta_n, \Xi_n \} \rightarrow \wedge \{ \Theta_n, \Xi_n \}$ given by
\begin{equation}
\psi_{\pi} := \psi_{B_1} \circ \psi_{B_2} \circ \cdots \circ \psi_{B_k}
\end{equation}
which does not depend on the order of composition factors.
We have
\begin{equation}
\psi_{\pi}(\theta_1 \theta_2 \cdots \theta_n) = 
\begin{cases}
F_{\pi} & \text{if $\pi$ has no singleton blocks} \\
0 & \text{if $\pi$ has at least one singleton block}
\end{cases}
\end{equation}

It will be convenient to have a version $\psi_{A,B}$ of the $\psi$-operators which depend
on two subsets $A, B \subseteq [n]$.  These are defined by
\begin{equation}
\psi_{A,B}(f) := \sum_{\substack{a \in A \\ b \in B}} \xi_a \cdot (\theta_b \odot f) + 
\sum_{\substack{a \in A \\ b \in B}} \xi_b \cdot (\theta_a \odot f)
\end{equation}
for any $f \in \wedge \{ \Theta_n, \Xi_n \}$, so that $\psi_{A,B} = \psi_{B,A}$.
When $A \cap B = \varnothing$, we have the useful  identity
\begin{equation}
\psi_{A \sqcup B} = \psi_A + \psi_{A,B} + \psi_B
\end{equation}
Like the $\rho$-operators, the $\psi$-operators commute.

\begin{lemma}
\label{psi-operators-commute}
Let $A, B, C, D \subseteq [n]$ be four subsets.  We have the following identities of
linear operators on $\wedge \{ \Theta_n, \Xi_n \}$
\begin{equation*}
\psi_A \circ \psi_B = \psi_B \circ \psi_A \quad \quad
\psi_{A,B} \circ \psi_C = \psi_C \circ \psi_{A,B} \quad \quad
\psi_{A,B} \circ \psi_{C,D} = \psi_{C,D} \circ \psi_{A,B}
\end{equation*}
\end{lemma}

The proof of Lemma~\ref{psi-operators-commute} is the same as that of
Lemma~\ref{block-operators-commute} and is left to the reader.

\begin{remark}
If $X_n = (x_1, \dots, x_n)$ and $Y_n = (y_1, \dots, y_n)$ are two lists of {\em commuting} variables,
the {\em polarization operator} on the polynomial ring $\CC[X_n, Y_n]$ acts by
\begin{equation}
\label{polarization-operator}
f \mapsto \sum_{i = 1}^n y_i \cdot \frac{\partial f}{\partial x_i}
\end{equation}
for any polynomial $f \in \CC[X_n, Y_n]$.
This operator lowers $x$-degree by 1 while raising $y$-degree by 1.
Similarly, the operators $\rho_B, \psi_B, \psi_{A,B}$ lower $\theta$-degree by 1 while raising 
$\xi$-degree by 1.
 Polarization operators on commuting variables are $\symm_n$-equivariant; 
Equation~\eqref{block-permutation} describes how the action of $\symm_n$ intertwines
with block operators.
Polarization on commuting variables has played a major role 
\cite{BergeronMulti, Bergeron, Haiman, RW} in the theory of diagonal symmetric group actions.
Our work suggests that block operators might be useful objects when dealing with 
anticommuting variables.
\end{remark}

\subsection{Antisymmetrization and the fermions $F$ and $f$}
Most of our results on the $F_{\pi}$ and $f_{\pi}$ will be provable at the level of the $\rho$
and $\psi$ operators. However, it will sometimes be useful to have a more explicit formula for these fermions.
For a composition $\alpha$,  let $|\alpha|_{\odd} := \alpha_1 + \alpha_3 + \cdots $ be the sum of the odd 
parts of $\alpha$.

\begin{defn}
\label{G-and-g-defn-perms}
Let $(w,\alpha)$ be a segmented permutation of size $n$ where $\alpha = (\alpha_1, \dots, \alpha_k) \models n$ 
has length $k$.
We define $G_{w,\alpha} \in \wedge \{ \Theta_n, \Xi_n \}$ by the formula
\begin{multline}
G_{w,\alpha} :=  \sign(w) \cdot (-1)^{|\alpha|_{\odd}} \cdot
 (\theta_{w[1]_1} \theta_{w[1]_2} \cdots \theta_{w[1]_{\alpha_1 - 1}}) \cdots
(\theta_{w[k]_1} \theta_{w[k]_2} \cdots \theta_{w[k]_{\alpha_k - 1}})  \times \\ \xi_{w[1]} \cdots \xi_{w[k]}
\end{multline}
where 
\begin{equation}
\xi_{w[i]} := \begin{cases}
\xi_{w[i]_1} + \cdots + \xi_{w[i]_{\alpha_i - 1}} & \alpha_i > 1 \\
\xi_{w[i]_1} & \alpha_i = 1
\end{cases}
\end{equation}
We define $g_{w,\alpha} \in \wedge \{ \Theta_n, \Xi_n \}$ by 
\begin{equation}
g_{w,\alpha} := (-1)^{n-k} (\xi_1 + \cdots + \xi_n) \odot G_{w,\alpha}
\end{equation}
\end{defn}

The sign $ \sign(w) \cdot (-1)^{|\alpha|_{\odd}}$ in Definition~\ref{G-and-g-defn-perms} 
are necessary to pass from segmented permutations to set partitions.
The $(-1)^{n-k}$ in the definition of $g_{w,\alpha}$ occurs `because' the derivative
$(\xi_1 + \cdots + \xi_n) \odot  (-)$ must commute past $n-k$ fermionic $\theta$-variables;
see Proposition~\ref{exterior-basic-properties} (3).

As an example of Definition~\ref{G-and-g-defn-perms}, let
 $(w, \alpha) = 536 \cdot 7 \cdot 21 \cdot 84$.
We have
\begin{equation*}
\sign(w) = \sign(53672184) = +1 \quad \text{and} \quad 
|\alpha|_{\odd} = \alpha_1 + \alpha_3 = 3 + 2 = 5
\end{equation*}
so that
\begin{align*}
G_{w,\alpha} &= (+1) \cdot (-1)^5 \cdot (\theta_5 \theta_3) \cdot 1 \cdot (\theta_2) \cdot (\theta_8) \cdot 
(\xi_5 + \xi_3) \cdot (\xi_7) \cdot (\xi_2) \cdot (\xi_8) \\
&= - (\theta_5 \theta_3) \cdot 1 \cdot (\theta_2) \cdot (\theta_8) \cdot 
(\xi_5 + \xi_3) \cdot (\xi_7) \cdot (\xi_2) \cdot (\xi_8)
\end{align*}
Applying $(-1)^{8-4} (\xi_1 + \cdots + \xi_8) \odot (-)$ to both sides of this equation yields
\begin{equation*}
g_{w,\alpha}  
= - (\theta_5 \theta_3) \cdot 1 \cdot (\theta_2) \cdot (\theta_8) \times \\
\left[ 2 \cdot  \xi_7 \xi_2 \xi_8 - (\xi_5 + \xi_3)  \xi_2 \xi_8
+ (\xi_5 + \xi_3) \xi_7 \xi_8  - (\xi_5 + \xi_3) \xi_7 \xi_2 \right]
\end{equation*}

By antisymmetrizing the $g_{w,\alpha}$ and $G_{w,\alpha}$, we obtain our new formulation for the 
$F$ and $f$ fermions.

\begin{defn}
\label{f-w-a-definition}
Let $(w, \alpha)$ be a segmented permutation where $\alpha = (\alpha_1, \dots, \alpha_k) \models n$.
We define elements $\widetilde{F}_{w,\alpha}, \widetilde{f}_{w,\alpha} \in \wedge \{ \Theta_n, \Xi_n \}$ by
\begin{equation}
\widetilde{F}_{w,\alpha} := 
\frac{[\symm_{w,\alpha}]^- \cdot G_{w,\alpha} }{(\alpha_1 - 1)! \cdots (\alpha_k - 1)!} \in \wedge \{\Theta_n, \Xi_n\}
\end{equation}
and
\begin{equation}
\widetilde{f}_{w,\alpha} := 
\frac{[\symm_{w,\alpha}]^- \cdot g_{w,\alpha} }{(\alpha_1 - 1)! \cdots (\alpha_k - 1)!} \in \wedge \{\Theta_n, \Xi_n\}
\end{equation}
\end{defn}

In our example 
$(w, \alpha) =  536 \cdot 7 \cdot 21 \cdot 84$ we have 
\begin{equation*}
\symm_{w,\alpha} = \symm_{\{5,3,6\}} \times \symm_{\{7\}} \times \symm_{\{2,1\}} \times \symm_{\{8,4\}}
\end{equation*}
so that 
\begin{equation*}
\widetilde{F}_{w,\alpha} = \frac{[\symm_{\{5,3,6\}}]^- \cdot [\symm_{\{7\}}]^- \cdot [\symm_{\{2,1\}}]^- \cdot [\symm_{\{8,4\}}]^- \cdot
G_{w,\alpha}}{2! \cdot 0! \cdot 1! \cdot 1!}
\end{equation*}
and
\begin{equation*}
\widetilde{f}_{w,\alpha} = \frac{[\symm_{\{5,3,6\}}]^- \cdot [\symm_{\{7\}}]^- \cdot [\symm_{\{2,1\}}]^- \cdot [\symm_{\{8,4\}}]^- \cdot
g_{w,\alpha}}{2! \cdot 0! \cdot 1! \cdot 1!}
\end{equation*}

\begin{proposition}
\label{block-operator-formulation}
Let $(w, \alpha)$ be a segmented permutation where $\alpha = (\alpha_1, \dots, \alpha_k) \models n$
and let $\pi = \Pi(w,\alpha) \in \Pi(n,k)$ be the corresponding set partition.
We have
\begin{equation}
\label{f-pi-block}
\widetilde{F}_{w,\alpha} =  
\begin{cases}
F_{\pi} & k \equiv 0, 3 \mod 4 \\
- F_{\pi} & k \equiv 1, 2 \mod 4
\end{cases}
\end{equation}
\end{proposition}

\begin{proof}
For $v \in \symm_n$, it follows from the definitions that $v \dot G_{w,\alpha} = \sign(v) \cdot G_{vw,\alpha}$.
Using this and the identity $[\symm_{vw,\alpha}]^- = v [\symm_{w,\alpha}]^- v^{-1}$, we calculate
\begin{equation}
v [\symm_{w,\alpha}]^- \cdot G_{w,\alpha} = 
[\symm_{vw,\alpha}]^- v \cdot G_{w,\alpha} = 
\sign(v) [\symm_{vw,\alpha}]^- G_{vw,\alpha}
\end{equation}
Dividing both sides by $(\alpha_1 - 1)! (\alpha_2 - 1)! \cdots (\alpha_k-1)!$ implies 
\begin{equation}
\label{tilde-f-transformation}
v \cdot \widetilde{F}_{w,\alpha} = \sign(v) \cdot \widetilde{F}_{vw,\alpha}
\end{equation}
which agrees with the $\symm_n$ action on the $F$'s in Proposition~\ref{prop:sn-action-on-F-and-f}.

By Equation~\eqref{tilde-f-transformation} and 
Proposition~\ref{prop:sn-action-on-F-and-f}, it is enough to 
verify Equation~\eqref{f-pi-block} when
the segmented permutation $(w,\alpha)$ has the form
\begin{equation}
(w,\alpha) = 1, \, 2, \, \ldots, \, \alpha_1 \, \cdot \, \alpha_1 + 1, \, \alpha_1 + 2, \, \ldots, \, \alpha_1 + \alpha_2 \, \cdot  \, \cdots 
\, \cdot \, n - \alpha_k + 1, \, \ldots, \, n-1, \, n  
\end{equation}
If we set 
\begin{equation}
B_i := \{ \alpha_1 + \cdots + \alpha_{i-1} + 1, \, \alpha_1 + \cdots + \alpha_{i-1} + 2, \, \ldots , \,
\alpha_1 + \cdots + \alpha_{i-1} + \alpha_i \},
\end{equation}
 then
 \begin{align}
 F_{\pi} &= ( \rho_{B_k} \circ \cdots \circ \rho_{B_2} \circ \rho_{B_1}) (\theta_1 \cdots \theta_n) \\
 &= \epsilon \cdot  [\symm_{w,\alpha}]^- \cdot  (\xi_n \cdots \xi_{\alpha_1 + \alpha_2} \xi_{\alpha_1} \cdot
 \theta_1 \theta_2 \cdots \theta_{\alpha_1 - 1} \theta_{\alpha_1 + 1} \theta_{\alpha_1 + 2} \cdots
 \theta_{\alpha_1 + \alpha_2 - 1} \theta_{\alpha_1 + \alpha_2 + 1} \cdots \theta_{n-1})
 \end{align}
 where the sign $\epsilon$ is given by
 \begin{equation}
 \epsilon = (-1)^{(k-1) \cdot \alpha_1 + (k-2) \cdot \alpha_2 + \cdots + 1 \cdot \alpha_{k-1} + 0 \cdot \alpha_k } =
 \begin{cases}
 (-1)^{|\alpha|_{\odd}} & \text{$k$ even} \\
 (-1)^{|\alpha|_{\odd} + n} & \text{$k$ odd}
 \end{cases}
 \end{equation}
 which simplifies to
 \begin{equation}
 \epsilon = (-1)^{|\alpha|_{\odd} + k \cdot n}
 \end{equation}
 Reversing the order of the $k$ factors $\xi_n \cdots \xi_{\alpha_1 + \alpha_2} \xi_{\alpha_1}$ 
 introduces a sign change by $(-1)^{{k \choose 2}}$ and moving these $k$ factors past 
 the $n-k$ factors 
 $\theta_1 \theta_2 \cdots \theta_{\alpha_1 - 1} \theta_{\alpha_1 + 1} \theta_{\alpha_1 + 2} \cdots
 \theta_{\alpha_1 + \alpha_2 - 1} \theta_{\alpha_1 + \alpha_2 + 1} \cdots \theta_{n-1}$
 changes sign by $(-1)^{k \cdot (n-k)}$.  We conclude that 
 \begin{equation}
 F_{\pi} = (-1)^{{k \choose 2} + k \cdot (n-k) + k \cdot n}   \cdot \widetilde{F}_{w,\alpha} = 
 (-1)^{{k \choose 2} - k^2} \cdot \widetilde{F}_{w,\alpha}
 \end{equation}
 which is equivalent to the statement of the proposition.
\end{proof}

Proposition~\ref{block-operator-formulation} implies that the $\widetilde{F}_{w,\alpha}$
and $\widetilde{f}_{w,\alpha}$ depend only on the set partition 
$\Pi(w,\alpha)$ rather than the segmented permutation $(w,\alpha)$ itself. This facilitates the definitions
\begin{equation}
\widetilde{F}_{\pi} := \widetilde{F}_{w,\alpha} \quad \quad \text{and} \quad \quad
\widetilde{f}_{\pi} := \widetilde{f}_{w,\alpha}
\end{equation}
where $(w,\alpha)$ is any segmented permutation such that $\Pi(w,\alpha) = \pi$.
In this language, Proposition~\ref{block-operator-formulation} informally reads
\begin{equation}
\label{F-tilde-sign}
F_{\pi} = \pm \widetilde{F}_{\pi} \quad \quad
f_{\pi} = \pm \widetilde{f}_{\pi}
\end{equation}
We will typically deal with set partitions having a fixed number of blocks, so 
the precise signs appearing in Equation~\eqref{F-tilde-sign} will usually not
play a significant role in our work.

\subsection{Restriction properties}
Given any set partition $\pi$ of $[n]$, we can form a set partition $\overline{\pi}$
of $[n-1]$ by removing $n$
(and its block, if $\{n\}$ is a singleton).
The effect of this operation on set partition fermions depends on the size of the block 
of $\pi$ containing $n$.
The answer is more attractive for the $\widetilde{F}$'s;
applying $(\xi_1 + \cdots + \xi_n) \odot (-)$
we can obtain a corresponding result
for the $\widetilde{f}$'s.

\begin{proposition}
\label{n-removal}
Let $\pi \in \Pi(n,k)$ be a set partition of $[n]$ with $k$ blocks. Let 
$\overline{\pi}$ be the set partition of $[n-1]$ obtained by removing $n$ from $\pi$ 
(and the block containing $n$, if $\{n\}$ is a singleton). Let $B$ be the block of $\pi$ containing $n$.
The fermions $F_{\pi}$ and $F_{\overline{\pi}}$ are related by
\begin{equation}
\widetilde{F}_{\overline{\pi}} = \begin{cases}
(-1)^n \theta_n \odot (\widetilde{F}_{\pi} \mid_{\xi_n = 0}) & |B| \geq 3 \\
(-1)^{n-1} \theta_i \odot \widetilde{F}_{\pi}  & |B| = 2 \text{ and } B = \{i, n\} \\
(-1)^{k-1} \xi_n \odot \widetilde{F}_{\pi} & |B| = 1
\end{cases}
\end{equation}
where $\widetilde{F}_{\pi} \mid_{\xi_n = 0}$ 
evaluates $\widetilde{F}_{\pi} \in \wedge \{ \Theta_n, \Xi_n \}$ at $\xi_n \rightarrow 0$.
\end{proposition}

\begin{proof}
Consider a segmented permutation $(w, \alpha)$ such that $\Pi(w,\alpha) = \pi$ which has the form
\begin{equation}
(n a_1 \cdots a_j) \cdot (b_1 \cdots b_m) \cdot \ldots \cdot (c_1 \cdots c_p)
\end{equation}
where the parentheses around segments are for readability.
The fermion $F_{\pi}$ is given by
\begin{multline}
\widetilde{F}_{\pi} = (-1)^{|\alpha|_{\odd}} \sign(n a_1 \dots a_j b_1 \dots b_m \cdots c_1 \dots c_p) \times \\
[\symm_{w,\alpha}]^- \cdot
(\theta_n \theta_{a_1} \cdots \theta_{a_{j-1}})
\cdots (\theta_{c_1} \cdots \theta_{c_{p-1}})
(\xi_n + \xi_{a_1} + \cdots +  \xi_{a_{j-1}}) \cdots (\xi_{c_1} + \cdots + \xi_{c_{p-1}})
\end{multline}
\label{restrict-first}
If $|B| \geq 2$, then $j \geq 1$, a segmented permutation representing $\overline{\pi}$ is 
\begin{equation}
(\overline{w}, \overline{\alpha}) := (a_1 \cdots a_j) \cdot (b_1 \cdots b_m) \cdot \cdots \cdot (c_1 \cdots c_p)
\end{equation}
and since $|\alpha|_{\odd} = |\overline{\alpha}|_{\odd} - 1$ and 
$\sign(\overline{w}) = (-1)^{n-1} \sign(w)$ we have
\begin{multline}
\label{restrict-second}
\widetilde{F}_{\overline{\pi}} = (-1)^{|\alpha|_{\odd} - 1} (-1)^{n-1}
\sign(n a_1 \dots a_j b_1 \dots b_m \cdots c_1 \dots c_p) \times \\
[\symm_{\overline{w},\overline{\alpha}}]^- \cdot
(\theta_{a_1} \cdots \theta_{a_{j-1}})
\cdots (\theta_{c_1} \cdots \theta_{c_{p-1}})
(\xi_{a_1} + \cdots  + \xi_{a_{j-1}}) \cdots (\xi_{c_1} + \cdots + \xi_{c_{p-1}})
\end{multline}
and comparing these formulas gives the result (In the case $|B| = 2$ and $j = 1$
the product $\theta_{a_1} \cdots \theta_{a_{j-1}}$ appearing in $F_{\overline{\pi}}$ is empty,
while a permutation in $\symm_{w,\alpha}$ of negative sign is required
to send $\xi_n + \xi_{a_1} + \cdots + \xi_{a_{j-1}} = \xi_n$ to $\xi_i$ in $F_{\pi}$.)

If $|B| = 1$ and $j = 0$, a segmented permutation representing $\overline{\pi}$ is 
\begin{equation}
(\overline{w}, \overline{\alpha}) :=  (b_1 \cdots b_m) \cdot \ldots \cdot (c_1 \cdots c_p)
\end{equation}
and the relations $|\overline{\alpha}|_{\odd} = n - |\alpha|_{\odd}$ and 
$\sign(\overline{w}) = (-1)^{n-1} \sign(w)$ give
\begin{multline}
\label{removal-equation-one}
\widetilde{F}_{\overline{\pi}} = (-1)^{n - |\alpha|_{\odd}} (-1)^{n-1} 
\sign(n b_1 \dots b_m \cdots c_1 \dots c_p) \times \\
[\symm_{\overline{w},\overline{\alpha}}]^- \cdot
(\theta_{b_1} \cdots \theta_{b_{m-1}})
\cdots (\theta_{c_1} \cdots \theta_{c_{p-1}})
(\xi_{b_1} + \cdots + \xi_{b_{m-1}}) \cdots (\xi_{c_1} + \cdots + \xi_{c_{p-1}})
\end{multline}
whereas 
\begin{multline}
\label{removal-equation-two}
\widetilde{F}_{\pi} = (-1)^{|\alpha|_{\odd}} 
\sign(n b_1 \dots b_m \cdots c_1 \dots c_p) \times \\
[\symm_{\overline{w},\overline{\alpha}}]^- \cdot
(\theta_{b_1} \cdots \theta_{b_{m-1}})
\cdots (\theta_{c_1} \cdots \theta_{c_{p-1}}) (\xi_n)
(\xi_{b_1} + \cdots \xi_{b_{m-1}}) \cdots (\xi_{c_1} + \cdots + \xi_{c_{p-1}})
\end{multline}
Applying $\xi_n \odot (-)$ to remove the $\xi_n$ from $F_{\pi}$ involves 
$k$ sign changes (the number of blocks of $\pi$, or the number of $\theta$-variables).
The top lines of Equations~\eqref{removal-equation-one}
and \eqref{removal-equation-two} differ in an additional factor of $(-1)^{2n - 1} = -1$.
\end{proof}

%\begin{remark}
%\label{restriction-remark}
%What happens when we try to apply the top branch of Proposition~\ref{n-removal} when 
%$|B| = 2$? When $B = \{i,n\}$ has size 2 we have (maintaining the notation in the proof)
%\begin{multline}
%F_{\pi} = (-1)^{|\alpha|_{\odd}} \sign(n i b_1 \dots b_m \cdots c_1 \dots c_p) \times \\
%[\symm_{w,\alpha}]^- \cdot
%(\theta_n) (\theta_{b_1} \cdots \theta_{b_{m-1}})
%\cdots (\theta_{c_1} \cdots \theta_{c_{p-1}})
%(\xi_n) (\xi_{b_1} + \cdots + \xi_{b_{m-1}})
%\cdots (\xi_{c_1} + \cdots + \xi_{c_{p-1}})
%\end{multline}
%so that setting $\xi_n = 0$ annihilates every term involving $\theta_n$.  We conclude that 
%\begin{center}
%$\theta_n \odot (F_{\pi} \mid_{\xi_n = 0})  = 0$ whenever the block $B$ containing $n$ in $\pi$
%has size 2.
%\end{center}
%\end{remark}

%Proposition~\ref{n-removal} and Remark~\ref{restriction-remark} have application to the
%set partition skein relations shown in Figure~\ref{fig:skein}.
%The 2-term and 3-term skein relations are `degenerations' of the 4-term skein relation
%on the bottom which arise when the blocks being crossed have fewer than three elements.
%The vanishing in Remark~\ref{restriction-remark} reflects this degeneration at the level of the $F_{\pi}$.
%Both Proposition~\ref{n-removal} and Remark~\ref{restriction-remark} will be helpful in proving
%that the $F_{\pi}$ satisfy the skein relations 
%(Theorem~\ref{fermion-skein-theorem}). 

\section{Fermions and skein relations}
\label{Skein}

\subsection{Almost noncrossing partitions and the skein action}
We present a modified version of the skein action of $\symm_n$ on $\CC[\NC(n)]$ defined 
in \cite[Sec. 3]{Rhoades}. The heart of this construction is a resolution of crossings in set partitions
which are almost, but not quite, noncrossing.

A set partition $\pi \in \Pi(n)$ is {\em almost noncrossing} if $\pi$ is not noncrossing but there exists an index
$1 \leq i \leq n-1$ such that $s_i(\pi)$ is noncrossing. The index $i$ is not always uniquely determined
by $\pi$:
if $\pi = \{1 \,, 3, \, 5 \, /  \, 2, \,  4 \} \in \Pi(5)$ then 
both $s_1(\pi)$ and $s_3(\pi)$ are noncrossing, so that $\pi \in \ANC(5)$.

Let $\ANC(n)$ be the family of almost noncrossing set partitions of $[n]$.
We define a set map
\begin{equation}
\sigma: \ANC(n) \longrightarrow \CC[\NC(n)]
\end{equation}
as follows; see Figure~\ref{fig:skein}. 

\begin{defn}
\label{skein-map-definition}
Let $\pi \in \ANC(n)$ be such that $s_i(\pi) \in \NC(n)$. Then $i$ and $i+1$ are  
in different blocks of $\pi$; let $B_i$ be the block of $\pi$ containing $i$ and 
$B_{i+1}$ be the block of $\pi$ containing $i+1$.  The blocks $B_i$ and $B_{i+1}$ both have size at least 2.
We set
\begin{equation}
\sigma(\pi) := \begin{cases}
\pi_1 + \pi_2 & \text{if $|B_i| = |B_{i+1}| = 2$} \\
\pi_1 + \pi_2 - \pi_3 & \text{if $|B_i| > 2$ and $|B_{i+1}| = 2$} \\
\pi_1 + \pi_2  \quad \quad \,   - \pi_4 & \text{if $|B_i| = 2$ and $|B_{i+1}| > 2$} \\
\pi_1 + \pi_2 - \pi_3 - \pi_4 & \text{if $|B_i|, \, |B_{i+1}| > 2$}
\end{cases}
\end{equation}
where the set partitions $\pi_1, \dots, \pi_4 \in \NC(n)$ are obtained from $\pi$ by replacing $B_i$ and
$B_{i+1}$ with the new pair of blocks
\begin{itemize}
\item $(B_i - \{i\}) \cup \{i+1\}$ and $(B_{i+1} - \{i+1\}) \cup \{i\}$ for $\pi_1$,
\item $(B_i \cup B_{i+1}) - \{i, i+1\}$ and $\{i,i+1\}$ for $\pi_2$,
\item  $B_i - \{i\}$ and $B_{i+1} \cup \{i\}$ for $\pi_3$, and
\item  $B_{i+1} - \{i+1\}$ and $B_i \cup \{i+1\}$ for $\pi_4$.
\end{itemize}
\end{defn}

It is proven in \cite[Lem. 3.3]{Rhoades} that if $s_i(\pi)$ is noncrossing for more than one value of $i$,
the above procedure yields the same element $\sigma(\pi) \in \CC[\NC(n)]$. In other words,
the set map $\sigma$ is well-defined.

\begin{defn}
\label{skein-action-definition}
For $1 \leq i \leq n-1$, the {\em skein action} of the adjacent transposition $s_i$ on $\CC[\NC(n)]$ is given by
\begin{equation}
s_i \cdot \pi := \begin{cases}
- s_i(\pi) & \text{if $s_i(\pi)$ is noncrossing} \\
\sigma(s_i(\pi)) & \text{otherwise}
\end{cases}
\end{equation}
\end{defn}

The sign conventions in Definition~\ref{skein-action-definition} are slightly different from those in
\cite[Eqn. (4.1)]{Rhoades}.
 The action of $s_i$ on $\pi \in \NC(n)$ in \cite{Rhoades} did not introduce a sign when 
at least one of $i, i+1$ formed a singleton block of $\pi$.
The calculation-intensive
arguments of \cite[Lem. 4.1, Lem. 4.2, Lem. 4.3]{Rhoades} go through to show that the action of 
Definition~\ref{skein-action-definition} satisfies the Coxeter relations and we have an induced 
action of $\symm_n$ on $\CC[\NC(n)]$.
Fermions will give a more conceptual proof
(Theorem~\ref{fermion-skein-theorem}, Theorem~\ref{module-isomorphism}) that this action is well-defined.

\subsection{Block operators and skein relations}
In this subsection we prove our first major result: a link between fermions and skein relations.  
We first state our result  at the level of the block operators $\rho_{\pi}$.

For notational convenience, if $\pi \in \ANC(n)$ is an almost noncrossing partition such that 
$s_i(\pi)$ is noncrossing, we define a linear operator
\begin{equation}
\rho_{\sigma(\pi)}: \wedge \{ \Theta_n, \Xi_n \} \longrightarrow \wedge \{ \Theta_n, \Xi_n \}
\end{equation}
by the formula
\begin{equation}
\rho_{\sigma(\pi)} := \begin{cases}
\rho_{\pi_1} + \rho_{\pi_2} & \text{if $|B_i| = |B_{i+1}| = 2$} \\
\rho_{\pi_1} + \rho_{\pi_2} - \rho_{\pi_3} & \text{if $|B_i| > 2$ and $|B_{i+1}| = 2$} \\
\rho_{\pi_1} + \rho_{\pi_2}  \hspace{0.41in} - \rho_{\pi_4} & \text{if $|B_i| = 2$ and $|B_{i+1}| > 2$} \\
\rho_{\pi_1} + \rho_{\pi_2} - \rho_{\pi_3} - \rho_{\pi_4} & \text{if $|B_i|, \, |B_{i+1}| > 2$}
\end{cases}
\end{equation}
where $B_i$ is the block of $\pi$ containing $i$, $B_{i+1}$ is the block of $\pi$ containing $i+1$, 
and $\pi_1, \dots, \pi_4 \in \NC(n)$ are as in Definition~\ref{skein-map-definition}.
The following result states that the block operators satisfy the skein relations.

\begin{theorem}
\label{operator-theorem}
Suppose $\pi \in \ANC(n)$ is an almost noncrossing partition such that $s_i(\pi)$ is noncrossing.
We have 
\begin{equation}
\rho_{\pi} + \rho_{\sigma(\pi)} = 0
\end{equation}
as operators on $\wedge \{ \Theta_n, \Xi_n \}$.
\end{theorem}

\begin{proof}
Suppose $A \sqcup \{i+1\}$ and $B \sqcup \{i\}$ are blocks of $\pi$.
By the definition of $\rho_{\pi}$ and $\rho_{\sigma(\pi)}$
and the commutativity statement in the last paragraph, it suffices to show
the operator identity
\begin{equation}
\label{skein-operator-goal}
\psi_{A \sqcup \{i+1\}} \circ \psi_{B \sqcup \{i\}} + 
\psi_{A \sqcup \{i\}} \circ \psi_{B \sqcup \{i+1\}} + 
\psi_{A \sqcup B} \circ \psi_{\{i, i+1\}} \\ - 
\psi_{A } \circ \psi_{B \sqcup \{i, i+1\}} -
\psi_{A \sqcup \{i, i+1\}} \circ \psi_{B} = 0
\end{equation}
where the $\psi$-operators avoid the branching in the definition of $\rho_{\sigma(\pi)}$.
We prove Equation~\eqref{skein-operator-goal} by a sign-reversing involution.

In terms of the $\psi_{S,T}$-operators, the desired Equation~\eqref{skein-operator-goal} reads
\begin{multline}
\label{new-skein-operator-goal}
(\psi_A + \psi_{A,\{i+1\}}) \circ (\psi_B + \psi_{B,\{i\}}) + 
(\psi_A + \psi_{A.\{i\}}) \circ (\psi_B + \psi_{B,\{i+1\}})  \\ +
(\psi_A + \psi_{A,B} + \psi_{B}) \circ \psi_{\{i,i+1\}} - 
\psi_A \circ ( \psi_B + \psi_{B, \{i\}} + \psi_{B, \{i+1\}} + \psi_{\{i.i+1\}}) \\ -
\psi_B \circ ( \psi_A + \psi_{A, \{i\}} + \psi_{A, \{i+1\}} + \psi_{\{i.i+1\}}) = 0
\end{multline}
Expanding the LHS of Equation~\eqref{new-skein-operator-goal}, applying Lemma~\ref{psi-operators-commute}, and simplifying gives
\begin{equation}
\label{newer-skein-operator-goal}
\psi_{A,\{i+1\}} \circ \psi_{B,\{i\}} + \psi_{A,\{i\}} \circ \psi_{B,\{i+1\}} + \psi_{A,B} \circ \psi_{\{i,i+1\}}
\end{equation}
so that given $f \in \wedge \{ \Theta_n, \Xi_n \}$
the action of the LHS of Equation~\eqref{new-skein-operator-goal} on $f$ is 
\begin{equation}
\label{action-expression}
\frac{1}{2} \sum_{(t_1, t_2, t_3, t_4)} 
\xi_{t_1} \cdot (\theta_{t_2} \odot ( \xi_{t_3} \cdot (\theta_{t_4} \odot f ) ) ) 
\end{equation}
where the sum is over all quadruples $1 \leq t_1, \dots, t_4 \leq n$ such that precisely one 
$t_j$ lies in each of the four sets $A, B, \{i\}, \{i+1\}$.
The factor of $\frac{1}{2}$ in the \eqref{action-expression} 
arises from the double counting $(t_1, t_2) \leftrightarrow (t_3, t_4)$
involved in applying the $\psi$-expression \eqref{newer-skein-operator-goal} to $f$.
Equation~\eqref{new-skein-operator-goal} and the theorem will be proved if we can show that
the expression \eqref{action-expression} vanishes. 
Anticommutativity yields
\begin{equation}
\xi_{t_1} \cdot (\theta_{t_2} \odot ( \xi_{t_3} \cdot (\theta_{t_4} \odot f ) ) )  = 
- \xi_{t_3} \cdot (\theta_{t_2} \odot ( \xi_{t_1} \cdot (\theta_{t_4} \odot f ) ) ) 
\end{equation}
which sets up a sign-reversing involution on the terms in  \eqref{action-expression}.
\end{proof}

The sign-reversing involution in the proof
Theorem~\ref{operator-theorem}
relied on anticommutativity in a crucial way.
We regard this as evidence that fermions are a good setting for studying resolution of set partition
crossings.

The fact that the $F_{\pi}$ and $f_{\pi}$ satisfy the skein relations is easily 
deduced from Theorem~\ref{operator-theorem}.
In analogy with the case of block operators, 
if $\pi \in \ANC(n)$ is almost noncrossing and $s_i(\pi)$ is noncrossing,  
we define $F_{\sigma(\pi)} \in \wedge \{ \Theta_n, \Xi_n \}$ by
\begin{equation}
F_{\sigma(\pi)} := \begin{cases}
F_{\pi_1} + F_{\pi_2} & \text{if $|B_i| = |B_{i+1}| = 2$} \\
F_{\pi_1} + F_{\pi_2} - F_{\pi_3} & \text{if $|B_i| > 2$ and $|B_{i+1}| = 2$} \\
F_{\pi_1} + F_{\pi_2}  \hspace{0.41in} - F_{\pi_4} & \text{if $|B_i| = 2$ and $|B_{i+1}| > 2$} \\
F_{\pi_1} + F_{\pi_2} - F_{\pi_3} - F_{\pi_4} & \text{if $|B_i|, \, |B_{i+1}| > 2$}
\end{cases}
\end{equation}
where $B_i$ is the block of $\pi$ containing $i$, $B_{i+1}$ is the block of $\pi$ containing $i+1$, 
and $\pi_1, \dots, \pi_4 \in \NC(n)$ are as in Definition~\ref{skein-map-definition}.
Similarly, we define  $f_{\sigma(\pi)} \in \wedge \{ \Theta_n, \Xi_n \}$ by 
\begin{equation}
f_{\sigma(\pi)} := \begin{cases}
f_{\pi_1} + f_{\pi_2} & \text{if $|B_i| = |B_{i+1}| = 2$} \\
f_{\pi_1} + f_{\pi_2} - f_{\pi_3} & \text{if $|B_i| > 2$ and $|B_{i+1}| = 2$} \\
f_{\pi_1} + f_{\pi_2}  \hspace{0.39in} - f_{\pi_4} & \text{if $|B_i| = 2$ and $|B_{i+1}| > 2$} \\
f_{\pi_1} + f_{\pi_2} - f_{\pi_3} - f_{\pi_4} & \text{if $|B_i|, \, |B_{i+1}| > 2$}
\end{cases}
\end{equation}

\begin{theorem}
\label{fermion-skein-theorem}
Let $\pi \in \NC(n)$ and $1 \leq i \leq n-1$. Then
\begin{equation}
s_i \cdot F_{\pi} := \begin{cases}
- F_{s_i(\pi)} & \text{if $s_i(\pi)$ is noncrossing} \\
F_{\sigma(s_i(\pi))} & \text{otherwise}
\end{cases} 
\end{equation}
and
\begin{equation}
s_i \cdot f_{\pi} := \begin{cases}
- f_{s_i(\pi)} & \text{if $s_i(\pi)$ is noncrossing} \\
f_{\sigma(s_i(\pi))} & \text{otherwise}
\end{cases} 
\end{equation}
\end{theorem}

%Our proof of Theorem~\ref{fermion-skein-theorem} makes heavy use of 
%Proposition~\ref{basic-observation} (4) which allows us to conjugate
%equations involving the $F_{\pi}$'s by permutations $w \in \symm_n$
%and Proposition~\ref{n-removal}/Remark~\ref{restriction-remark}
%which allow us to restrict equations from $\Pi(n)$ to $\Pi(n-1)$.

\begin{proof}
 Proposition~\ref{prop:sn-action-on-F-and-f} implies
 \begin{equation}
 s_i \cdot F_{\pi} = \sign(s_i) F_{s_i(\pi)} = - F_{s_i(\pi)}
 \end{equation}
 so we are done if $s_i(\pi) \in \NC(n)$ is noncrossing.
We therefore assume $s_i(\pi) \in \ANC(n)$ is almost noncrossing.
The desired formula follows from applying both sides of the operator identity
of Theorem~\ref{operator-theorem} to $\theta_1 \cdots \theta_n$.
\end{proof}

Thanks to fermions and block operators, 
 the proofs in this section were 
much faster and  cleaner  
than the corresponding proofs in \cite[Sec. 3]{Rhoades}.
The proofs in \cite[Sec. 3]{Rhoades} were brute force and involved extensive casework depending on 
block sizes; the $\psi$-operators in the proof of Theorem~\ref{operator-theorem} unify this casework.

Theorem~\ref{operator-theorem}  yields other families of 
fermions labeled by set partitions which satisfy the skein relations.
Suppose  $h \in \wedge \{ \Theta_n, \Xi_n \}$ 
is alternating, 
and let
 $\mathcal{T} := \{ \rho_{\pi}(h) \,:\, \pi \in \Pi(n) \}$.
Theorem~\ref{operator-theorem} shows that $\mathrm{span}  \,  \mathcal{T}$ is $\symm_n$-stable,
and that the polynomials appearing in $\mathcal{T}$ satisfy the skein relations under
the action of $s_i$.
This construction also makes sense in the presence of more than two sets 
$\Theta_n, \Xi_n, \dots, \Omega_n$
of fermionic variables;
this might help
in the multidiagonal context of
Problem~\ref{multigraded-problem} below.

\section{Noncrossing bases in $\wedge \{ \Theta_n, \Xi_n \}$}
\label{Basis}

\subsection{The modules $V$ and $W$}
Given $n, k, m \geq 0$, we define six subspaces of  
$\wedge \{ \Theta_n, \Xi_n \}$
as follows.
\begin{equation}
\begin{cases}
W(n) :=  \mathrm{span} \{ F_{\pi} \,:\, \pi \in \Pi(n) \} \\
W(n,k) := \mathrm{span} \{ F_{\pi} \,:\, \pi \in \Pi(n,k) \} \\
W(n,k,m) := \mathrm{span} \{ F_{\pi} \,:\, \pi \in \Pi(n,k,m) \} 
\end{cases}   
\begin{cases}
V(n) :=  \mathrm{span} \{ f_{\pi} \,:\, \pi \in \Pi(n) \} \\
V(n,k) := \mathrm{span} \{ f_{\pi} \,:\, \pi \in \Pi(n,k) \} \\
V(n,k,m) := \mathrm{span} \{ f_{\pi} \,:\, \pi \in \Pi(n,k,m) \} 
\end{cases}
\end{equation}
Degree considerations imply that the sums 
\begin{equation}
\label{direct-sums}
W(n) = \bigoplus_{k = 0}^n W(n,k) \quad \quad \text{and} \quad \quad
V(n) = \bigoplus_{k = 0}^n V(n,k)
\end{equation}
of subspaces are direct.  We shall see (Theorem~\ref{upstairs-basis-theorem}) that the sums 
\begin{equation}
W(n,k) = \sum_{m = 0}^k W(n,k,m) \quad \quad \text{and} \quad \quad
V(n,k) = \sum_{m = 0}^k V(n,k,m)
\end{equation}
are also direct.
We record some additional structural properties of these spaces.

\begin{proposition}
\label{w-v-structure}
The six spaces $W(n), W(n,k), W(n,k,m), V(n), V(n,k),$ and $V(n,k,m)$ are closed under the 
action of $\symm_n$ on $\wedge \{ \Theta_n, \Xi_n \}$.  Furthermore, these spaces are spanned by 
noncrossing fermions. That is, we have 
\begin{equation*}
\begin{cases}
W(n) =  \mathrm{span} \{ F_{\pi} \,:\, \pi \in \NC(n) \} \\
W(n,k) = \mathrm{span} \{ F_{\pi} \,:\, \pi \in \NC(n,k) \} \\
W(n,k,m) = \mathrm{span} \{ F_{\pi} \,:\, \pi \in \NC(n,k,m) \} 
\end{cases}  \quad 
\begin{cases}
V(n) =  \mathrm{span} \{ f_{\pi} \,:\, \pi \in \NC(n) \} \\
V(n,k) = \mathrm{span} \{ f_{\pi} \,:\, \pi \in \NC(n,k) \} \\
V(n,k,m) = \mathrm{span} \{ f_{\pi} \,:\, \pi \in \NC(n,k,m) \} 
\end{cases}
\end{equation*}
\end{proposition}

\begin{proof}
The $\symm_n$-closure follows from Proposition~\ref{prop:sn-action-on-F-and-f}.
To see that the noncrossing fermions span $W(n)$, we argue as follows. Let $\pi \in \Pi(n)$ be an arbitrary
set partition. There exists $w \in \symm_n$ such that $w(\pi) \in \NC(n)$ is noncrossing. We have 
\begin{equation}
F_{\pi} = \sign(w) \cdot w^{-1} \cdot F_{w(\pi)}
\end{equation}
by Proposition~\ref{prop:sn-action-on-F-and-f}. Writing $w^{-1}$ as a product of adjacent transpositions $s_i$
and applying them to $F_{w(\pi)}$ in succession, 
Theorem~\ref{fermion-skein-theorem} guarantees the we obtain a $\ZZ$-linear combination of 
$F_{\mu}$'s for $\mu \in \NC(n)$ noncrossing.
For the case of $W(n,k)$ and $W(n,k,m)$, observe that the skein relations in 
Figure~\ref{fig:skein} preserve the total number of blocks and the number of singleton blocks.
The corresponding statements for the $V$-spaces follow from 
an application of $(\xi_1 + \cdots + \xi_n) \odot (-)$.
\end{proof}

We will see that the six spanning sets in Proposition~\ref{w-v-structure} are in fact bases.
The linear independence of these sets could in principle be established by examining expansions
in the monomial basis $\theta_S \cdot \xi_T$ of $\wedge \{ \Theta_n, \Xi_n \}$,  but
the coefficients involved obstruct this approach.
We employ a more conceptual method hinging on 
a careful analysis of the singleton-free cases
$W(n,k,0)$ and $V(n,k,0)$ of these $\symm_n$-modules.

\subsection{Singleton-free partitions and flag-shaped irreducibles}
A partition $\lambda \vdash n$ is {\em flag-shaped} if it is of the form $\lambda = (k, k, 1^{n-2k})$ for some 
$1 \leq k \leq n/2$.
The goal of this subsection is to show that the $\symm_n$-modules $W(n,k,0)$ and $V(n,k,0)$
spanned by singleton-free $k$-block fermions are flag-shaped irreducibles.
We begin with a general criterion for when a Young symmetrizer or antisymmetrizer annihilates an 
$\symm_n$-irreducible.
Here we compare partitions in dominance order.

\begin{lemma}
\label{pincer-lemma}
Let $\lambda, \mu \vdash n$ be partitions.
 We have $[\symm_{\mu}]^+ \cdot S^{\lambda} \neq 0$ if and only if $\mu \leq \lambda$.
\end{lemma}

Lemma~\ref{pincer-lemma} is equivalent to the fact that the {\em Kostka number} 
$K_{\lambda,\mu}$ counting semistandard tableaux of shape $\lambda$ and content
$\mu$ is nonzero if and only if $\mu \leq \lambda$ in dominance order.

\begin{proposition}
\label{flag-shaped-isomorphism-type}
For $k \leq n/2$, both of the $\symm_n$-modules $W(n,k,0)$ and $V(n,k,0)$ are isomorphic
to the flag-shaped irreducible $S^{\lambda}$ where $\lambda  = (k, k, 1^{n-2k})$.
\end{proposition}

\begin{proof}
It follows from the hook-length formula and independent
observations of O'Hara and Zeilberger that $\dim \, S^{\lambda} = |\NC(n,k,0)|$; see
\cite{Pechenik} for details. We verify that 
$[\symm_{\lambda}]^+$ {\bf does not} annihilate $W(n,k,0)$, 
but that $[\symm_{\mu}]^+$ {\bf does} annihilate $W(n,k,0)$ whenever $\lambda < \mu$.

To prove $[\symm_{\lambda}]^+ \cdot W(n,k,0) \neq 0$, 
by Proposition~\ref{block-operator-formulation}
it suffices to find a single $\pi_0 \in \Pi(n,k,0)$
with $[\symm_{\lambda}]^+ \cdot \widetilde{F}_{\pi_0} \neq 0$.  We let
\begin{equation}
\pi_0 := \{ 1, \, 2k, \, 2k+1, \, \cdots, \, n-1, \, n \, / \, 2, \, 2k-1 \, /  \, 3, \, 2k-2 \, / \, \cdots \, / \, k, \, k+1 \}
\end{equation}
so that 
\begin{multline}
\label{pi-0-F}
\widetilde{F}_{\pi_0} = C (\theta_2 \xi_2 - \theta_{2k-1} \xi_{2k-1}) (\theta_3 \xi_3 - \theta_{2k-2} \xi_{2k-2}) \cdots
(\theta_k \xi_k - \theta_{k+1} \xi_{k+1}) \times \\
[ \symm_{\{1, 2k, 2k+1, \dots, n-1, n\}} ]^- \cdot 
(\theta_1 \theta_{2k} \theta_{2k+1} \cdots \theta_{n-1}) (\xi_1 + \xi_{2k} + \xi_{2k+1} + \cdots + \xi_{n-1})
\end{multline}
where $C$ is a nonzero constant.
Equation~\eqref{pi-0-F} may be rewritten as 
\begin{multline}
\label{pi-0-F-new}
\widetilde{F}_{\pi_0} = C' (\theta_2 \xi_2 - \theta_{2k-1} \xi_{2k-1}) (\theta_3 \xi_3 - \theta_{2k-2} \xi_{2k-2}) \cdots
(\theta_k \xi_k - \theta_{k+1} \xi_{k+1}) \times \\
\left(e - \sum_{i \in \{1, 2k, 2k+1, \dots, n-1 \}} (i,n) \right) \cdot 
(\theta_1 \theta_{2k} \theta_{2k+1} \cdots \theta_{n-1}) (\xi_1 + \xi_{2k} + \xi_{2k+1} + \cdots + \xi_{n-1})
\end{multline}
where $C' = (n-2k+1)! \cdot C$ is also nonzero.
We claim that the coefficient of 
\begin{equation}
\theta_1 \xi_1 \theta_2 \xi_2  \cdots \theta_k \xi_k \cdot
 \theta_{2k} \theta_{2k+1} \cdots \theta_{n-1}
 \end{equation}
 in $[\symm_{\lambda}]^+ \cdot \widetilde{F}_{\pi_0}$
 is nonzero.
A permutation $w \in \symm_{\lambda} =  \symm_k \times \symm_k$
can only contribute to the desired coefficient when  $w(2k) = 2k$, and the permutation involved in the action
 on the second line of \eqref{pi-0-F-new} is the identity $e$ (rather than one of the transpositions $(i,n)$).
 We claim that all permutations $w$ in this (nonempty) set  contribute to the desired coefficient
 with the same sign.
 Indeed, for any such $w$, in order to contribute to the desired coefficient we must select the 
 first term (with positive sign) in each of the $k-1$ factors on the top line of 
 Equation~\eqref{pi-0-F-new}, and then choose the first term in the $\xi$-sum on the bottom
 line of Equation~\eqref{pi-0-F-new}.
 A uniform sign of $(-1)^{n-2k}$ is involved in moving the variables $\theta_{w(1)}$ and $\xi_{w(1)}$ 
 next to each other in the second line.
 Once this is done, the factors $\theta_i \xi_i$ for $1 \leq i \leq k$ commute signlessly.

 Now let $\mu \vdash n$ be a partition with $\lambda \leq \mu$ such that 
 $[\symm_{\mu}]^+ \cdot W(n,k,0) \neq 0$. 
By Proposition~\ref{w-v-structure} there exists $\pi \in \NC(n,k,0)$ such that 
$[\symm_{\mu}]^+ \cdot F_{\pi} \neq 0$.
If $i \sim i+1$ in $\pi$, we have $(1 + s_i) \cdot \pi = \pi - \pi = 0$ by 
Theorem~\ref{fermion-skein-theorem}.
 We argue that $\mu = \lambda$ as follows. 
 
 Since $\lambda \leq \mu$ we have $\mu_1 \geq k$. Because $[\symm_{\mu}]^+ \cdot F_{\pi} \neq 0$
 and $\pi$ has $k$ blocks, this forces $\mu_1 = k$ and implies that 
 $1, 2, \dots, k$ are in distinct blocks of $\pi$. 
 Since $\lambda \leq \mu$ we must also have $\mu_2 = k$ and furthermore 
 $k+1, k+2, \dots, 2k$ are in distinct blocks of $\pi$. Since $\pi$ is noncrossing and has $k$ blocks, this 
 forces $\pi = \pi_0$.  We have $\mu = (k, k, \mu_3, \dots )$, and if $\mu_3 > 1$ then 
 $s_{2k} \in \symm_{\mu}$ and so that $[\symm_{\mu}]^+ \cdot F_{\pi} = 0$, a contradiction.
 We conclude that $\lambda = \mu$. Lemma~\ref{pincer-lemma} applies to prove the
 $\symm_n$-isomorphism $W(n,k,0) \cong S^{\lambda}$.
 To obtain the corresponding statement about the $V(n,k,0)$, observe that the map
 $F \mapsto (\xi_1 + \cdots + \xi_n) \odot F$ is an $\symm_n$-equivariant  surjection
 $W(n,k,0) \twoheadrightarrow V(n,k,0)$.
 By irreducibility, we have $V(n,k,0) \cong S^{\lambda}$, as well.
\end{proof}

\begin{remark}
There is a faster proof of Proposition~\ref{flag-shaped-isomorphism-type} relying on
  results in \cite[Sec. 5]{Rhoades}.
 Theorem~\ref{fermion-skein-theorem} shows that $W(n,k,0)$ and $V(n,k,0)$
  are  nonzero quotients of 
 $\CC[\NC(n,k,0)]$.  By \cite[Prop. 5.2]{Rhoades}, we have $\CC[\NC(n,k,0)] \cong S^{\lambda}$
 where $\lambda = (k,k,1^{n-2k})$, and irreducibility forces
 $W(n,k,0) \cong V(n,k,0) \cong S^{\lambda}$.
 We presented the argument above to illustrate how fermions give an easier proof than that of 
 \cite[Prop. 5.2]{Rhoades}.
\end{remark}

\subsection{Linear independence}
By Proposition~\ref{w-v-structure}, noncrossing fermions form a spanning
set for the $W$- and $V$-modules. The next result states that they form a basis.

\begin{theorem}
\label{upstairs-basis-theorem}
Given $n, k, m \geq 0$, the sets 
\begin{equation}
\{ F_{\pi} \,:\, \pi \in \NC(n) \}, \quad 
\{ F_{\pi} \,:\, \pi \in \NC(n,k) \}, \quad \text{ and } \quad 
\{ F_{\pi} \,:\, \pi \in \NC(n,k,m) \}
\end{equation}
are bases of the $\symm_n$-modules $W(n), W(n,k),$ and $W(n,k,m)$, respectively.
Similarly, the sets
\begin{equation}
\{ f_{\pi} \,:\, \pi \in \NC(n) \}, \quad 
\{ f_{\pi} \,:\, \pi \in \NC(n,k) \}, \quad \text{ and } \quad 
\{ f_{\pi} \,:\, \pi \in \NC(n,k,m) \}
\end{equation}
are bases of $V(n), V(n,k),$ and $V(n,k,m)$, respectively.
\end{theorem}

\begin{proof}
By Proposition~\ref{w-v-structure} it suffices to verify the linear independence of these six sets.
We start with the case of $W(n,k)$. 

Suppose that we have $c_{\pi} \in \CC$ such that
\begin{equation}
\label{first-linear-relation}
\sum_{\pi \in \NC(n,k)} c_{\pi} \cdot  F_{\pi} = 0
\end{equation}
For $1 \leq i \leq n$, let $U_i$ be the subspace of $\wedge \{ \Theta_n, \Xi_n \}$
spanned by monomials $\theta_S \cdot \xi_T$ for which $i \notin S$ but $i \in T$. There is a linear projection
\begin{equation}
\tau_i : \wedge \{ \Theta_n, \Xi_n \} \twoheadrightarrow U_i
\end{equation}
which fixes any monomial $\theta_S \cdot \xi_T \in U_i$ and sends any monomial
 $\theta_S \cdot \xi_T \notin U_i$ to zero.
 
 For $\pi \in \NC(n,k)$, what does $\tau_i(F_{\pi})$ look like? If $\{i\}$ is not a singleton in $\pi$, it follows
 from the definition of $F_{\pi}$ that $\theta_i$ appears in any monomial of
 $F_{\pi}$ whenever $\xi_i$ does, so that
 $\tau_i(F_{\pi}) = 0$ in this case.
 If $\{i\}$ is a singleton in $\pi$, it is not hard to check that 
 \begin{equation}
 \tau_i(F_{\pi}) = \pm \xi_i \cdot \overline{F}_{\pi^{(i)}}
 \end{equation}
 where $\pi^{(i)}$ is the set partition of $\{1, \dots, \hat{i}, \dots, n \}$ obtained by removing the singleton
 $\{i\}$ from $\pi$ and $\overline{F}$ is defined in the same way as $F$ but over the variable set 
 $( \theta_1, \dots, \widehat{\theta_i}, \dots, \theta_n, \xi_1, \dots, \widehat{\xi_i}, \dots, \xi_n)$.
 Applying $\tau_i$ to both sides of Equation~\eqref{first-linear-relation}
 gives 
 \begin{equation}
\label{second-linear-relation}
\sum_{\substack{\pi \in \NC(n,k) \\ \{i\} \text{ is a block of $\pi$ }}}  \pm c_{\pi}  \cdot \xi_i \cdot 
\overline{F}_{\pi^{(i)}} = 0.
\end{equation}

By induction on $n$ and Equation~\eqref{second-linear-relation},
we have $c_{\pi} = 0$ whenever $\pi$ has singleton blocks, so that 
Equation~\eqref{first-linear-relation} reads
\begin{equation}
\label{third-linear-relation}
\sum_{\pi \in \NC(n,k,0)} c_{\pi} \cdot F_{\pi} = 0
\end{equation}
Proposition~\ref{flag-shaped-isomorphism-type} implies $\dim \, W(n,k,0) = | \NC(n,k,0) |$ so
that the spanning set $\{ F_{\pi} \,:\, \pi \in \NC(n,k,0) \}$ of $W(n,k,0)$ must also be linearly independent.
The coefficients $c_{\pi}$ appearing in Equation~\eqref{third-linear-relation} are therefore also zero and the 
set $\{ F_{\pi} \,:\, \pi \in \NC(n,k) \}$ is linearly independent.
Its subsets $\{ F_{\pi} \,:\, \pi \in \NC(n,k,m) \}$ must also be linearly independent for any $m$, and 
the directness of the sums in \eqref{direct-sums} implies that $\{ F_{\pi} \,:\, \pi \in \NC(n) \}$ is also linearly
independent.

The proof for the $f$'s and $V$'s is almost identical to the proof for the $F$'s and $W$'s.
One need only verify the  identity
 \begin{equation}
 \tau _i(f_{\pi}) = 
 \begin{cases}
 \pm \xi_i \cdot \overline{f}_{\pi^{(i)}} & \text{$\{i\}$ is a singleton block of $\pi$} \\
 0 & \text{otherwise}
 \end{cases}
 \end{equation}
 where $\overline{f}_{\pi^{(i)}}$ is defined over the variable set 
 $( \theta_1, \dots, \widehat{\theta_i}, \dots, \theta_n, \xi_1, \dots, \widehat{\xi_i}, \dots, \xi_n)$
 and apply the same argument.
\end{proof}

\subsection{Module structure}
The fermion modules $V$ and $W$ coincide with the skein modules.

\begin{theorem}
\label{module-isomorphism}
For any $n, k, m \geq 0$ we have isomorphisms of $\symm_n$ modules
\begin{equation}
\begin{cases}
\CC[\NC(n)] \cong V(n) \cong W(n) \\
\CC[\NC(n,k)] \cong V(n,k) \cong W(n,k)  \\
\CC[\NC(n,k,m)] \cong V(n,k,m) \cong W(n,k,m) 
\end{cases}
\end{equation}
where $\CC[\NC(n)], \CC[\NC(n,k)],$ and $\CC[\NC(n,k,m)]$ are endowed with the skein action.
These isomorphisms are given by $\pi \leftrightarrow f_{\pi} \leftrightarrow F_{\pi}$ for $\pi$
a noncrossing partition in each case.
\end{theorem}

\begin{proof}
Apply Theorem~\ref{fermion-skein-theorem} and
Theorem~\ref{upstairs-basis-theorem}.
\end{proof}

We record the Frobenius images of the modules involved 
in Theorem~\ref{module-isomorphism}.

\begin{corollary}
\label{frobenius-image}
For any $n, k, m \geq 0$ the Frobenius images of 
 $\CC[\NC(n,k,m)]$, 
$V(n,k,m),$ and  $W(n,k,m)$ 
are given by the common symmetric function
\begin{equation}
s_{(k-m,k-m,1^{n-2k+m)}} \cdot s_{(1^m)}
\end{equation}
These modules admit $\symm_n$-decompositions
\begin{equation}
\CC[\NC(n)] = \bigoplus_{k = 0}^n \CC[\NC(n,k)]
\quad \quad \text{and} \quad \quad
\CC[\NC(n,k)] = \bigoplus_{m = 0}^k \CC[\NC(n,k,m)]
\end{equation}
\begin{equation}
V(n) = \bigoplus_{k = 0}^n V(n,k)  
\quad \quad \text{and} \quad \quad
V(n,k) = \bigoplus_{m = 0}^k V(n,k,m) 
\end{equation}
\begin{equation}
W(n) = \bigoplus_{k = 0}^n W(n,k)  
\quad \quad \text{and} \quad \quad
W(n,k) = \bigoplus_{m = 0}^k W(n,k,m) 
\end{equation}
\end{corollary}

\begin{proof}
Proposition~\ref{flag-shaped-isomorphism-type} shows  $\CC[\NC(n,k,0)] \cong S^{(k,k,1^{n-2k})}$.
The definition of the skein action makes it clear that we have an induction product
\begin{equation}
\CC[\NC(n,k,m)] \cong \CC[\NC(n-m,k-m,0)] \circ \sign_{\symm_m}
\end{equation}
so that 
\begin{equation}
\Frob \, \CC[\NC(n,k,m)] = s_{(k-m,k-m,1^{n-2k+m})} \cdot s_{(1^m)}
\end{equation}
Theorem~\ref{module-isomorphism} proves the first statement.
The direct sum decompositions are clear for the skein modules; Theorem~\ref{module-isomorphism}
implies their truth for the $V$-modules and $W$-modules as well.
\end{proof}

The Frobenius images in Corollary~\ref{frobenius-image} may be easily calculated using the (dual) Pieri rule.
As an example, we have
\begin{align*}
\Frob \, V(9,5,1)  &= s_{44} \cdot s_1 = s_{54} + s_{441} \\
\Frob \, V(9,5,2)  &= s_{331} \cdot s_{11} = s_{441} + s_{432} + s_{4311} + s_{3321} 
+ s_{331^3} \\
\Frob \, V(9,5,3)  &= s_{2211} \cdot s_{1^3} = s_{3321} + s_{331^3} + s_{32^3} + s_{32211} + s_{321^4}
+ s_{2^4 1} + s_{2^3 1^3} + s_{221^5} \\
\Frob \, V(9,5,4)  &= s_{1^5} \cdot s_{1^4} = s_{2^4 1} + s_{2^3 1^3} + s_{221^5} + s_{2 1^7} + s_{1^9}
\end{align*}
We have $\Frob \, V(9,5,0) = \Frob \, V(9,5,5) = 0$ since any 5-block set partition of $[9]$ has at least one 
and at most four singleton blocks.
Summing these expressions gives the decomposition
\begin{multline*}
\Frob \, V(9,5) = s_{54} + 2 s_{441} + s_{432} + s_{4311} + 2 s_{3321} + 2 s_{331^3} + s_{32^3} \\ + 
s_{32211} + s_{321^4} + 2 s_{2^4 1} + 2 s_{2^3 1^3} + 2 s_{221^5} + s_{2 1^7} + s_{1^9}
\end{multline*}
of the Narayana-dimensional module $V(9,5)$.
The reader may notice that the coefficients in $\Frob \, V(9,5)$ are all in the set $\{0,1,2\}$.
This holds for any $\Frob \, V(n,k)$, as may be verified with the dual Pieri rule.
Similarly, if $s_{\lambda}$ appears in $\Frob \, V(n,k)$ we must have $\lambda_2 \geq \lambda_1 - 1$
and $\lambda_3 < 3$.

\section{Resolution of crossings in set partitions}
\label{Resolution}

A {\em chord diagram} of size $n$ (or a {\em perfect matching}) is a set partition of $[2n]$ 
consisting of $n$ blocks, each of size two.
In many mathematical contexts, one resolves crossings in chord diagrams by repeated
applications of 
\begin{center}

\begin{tikzpicture}[scale=0.5]

    \newcommand{\squa}
        {
        \foreach \i in {1,...,4}
                {
                \coordinate (p\i) at (90*\i + 45:\r);
                \filldraw(p\i) circle (\pointradius pt);
                }
        }
            
    \def\r{1}           % pentagon radius
    \def\pointradius{3} % points radius

    \foreach \a/\b/\c/\d [count=\shft from 0] in   
       {1/3/2/4,
        1/2/3/4,
        1/4/2/3}
        {
        \begin{scope}[xshift=100*\shft,yshift=0]
            \squa
            \foreach \i in {1,...,4}
                \draw[thick] (p\a) -- (p\b);
                \draw[thick] (p\c) -- (p\d);
            \squa
        \end{scope}
        }

	\node at (1.8,0) {$\mapsto$};
	\node at (5.2,0) {$+$};

    \end{tikzpicture}
    
    \end{center}
   Our skein action extends this technique to resolve crossings in arbitrary set partitions.
   Some of the 
   definitions and results in this subsection are more precise and cleaner versions of the results in 
   \cite[Sec. 6, Sec. 7]{Rhoades}\footnote{The reader following
   along in \cite{Rhoades} should note that the action of $\symm_n$
   on $\CC[\Pi(n)]$ denoted $\star$ there is different from -- and more complicated than -- our $\star$-action
   (which is simply the sign-twisted permutation action).}.

   \subsection{The crossing resolution $p$}
   Our resolution of a set partition as a linear combination of noncrossing set partitions 
   is packaged as a projection map.
   
   \begin{defn}
   \label{projection-map-definition}
   We define a $\CC$-linear map $p: \CC[\Pi(n)] \rightarrow \CC[\NC(n)]$ by setting
   \begin{equation*}
   p(\pi) := \sum_{\mu \in \NC(n)} c_{\pi,\mu} \cdot \mu
   \end{equation*}
   where $\pi \in \Pi(n)$ is a set partition and
   \begin{equation*}
   F_{\pi} = \sum_{\pi \in \NC(n)} c_{\pi, \mu} \cdot F_{\mu}
   \end{equation*}
   \end{defn}

   Given $\pi \in \Pi(n)$, we view 
   $p(\pi) = \sum_{\mu \in \NC(n)} c_{\pi,\mu} \cdot \mu$, or its notationally abusive alias
   \begin{equation}
   `` \, \, \pi = \sum_{\mu \in \NC(n)} c_{\pi,\mu} \cdot \mu \, \, "
   \end{equation}
   as resolving the crossings in the set partition $\pi$.
   Theorem~\ref{upstairs-basis-theorem} guarantees that the linear map $p$ is well-defined 
   and that $p(\pi) = \pi$ whenever $\pi$ is noncrossing.

    Definition~\ref{projection-map-definition} can be cumbersome to apply 
   since it involves basis expansions in $\wedge \{ \Theta_n, \Xi_n \}$.
   However,
   there is a purely combinatorial algorithm for computing 
   $p(\pi)$; this is presented after Corollary~\ref{quadratic-extension} 
   and Observation~\ref{untangle-observation} below.
   There is also an algebraic characterization of 
   the linear map $p: \CC[\Pi(n)] \rightarrow \CC[\NC(n)]$.
   This is given after Theorem~\ref{projection-equivariance}.

    For arbitrary set partitions
    $\pi$, the coefficients $c_{\pi,\mu}$ appearing in Definition~\ref{projection-map-definition}
   are all integers.
   Indeed, if $w \in \symm_n$ is chosen such that $w(\pi) \in \Pi(n)$ is noncrossing, 
   Proposition~\ref{prop:sn-action-on-F-and-f} yields
   \begin{equation}
   F_{\pi} = F_{w^{-1}(w(\pi))} =  \sign(w) \cdot w^{-1} \cdot F_{w(\pi)}
   \end{equation}
   Since $w^{-1} \cdot w(\pi)$ is a $\ZZ$-linear combination
   of noncrossing partitions under the skein action and $\pi \mapsto F_{\pi}$ affords an isomorphism
   $\CC[\NC(n)] \xrightarrow{\, \, \sim \, \, } W(n)$, we have $c_{\pi, \mu} \in \ZZ$ always.
   
   By Theorem~\ref{module-isomorphism}, for any $\pi \in \Pi(n)$ we have 
      \begin{equation}
   F_{\pi} = \sum_{\pi \in \NC(n)} c_{\pi, \mu} \cdot F_{\mu} \quad \quad 
   \Leftrightarrow \quad \quad 
   f_{\pi} = \sum_{\pi \in \NC(n)} c_{\pi, \mu} \cdot f_{\mu}
   \end{equation}
   so that Definition~\ref{projection-map-definition} could have been stated using the $f$'s rather than the
   $F$'s.  The $F$'s are more  convenient to use when computing $p(\pi)$.
   For our first example of crossing resolution, we revisit the classical setting of chord diagrams.

   \begin{example}  {\rm (Chord diagrams)}
   Consider the simplest crossing chord diagram $\pi = \{ 1, \, 3 \, / \,  2, \, 4 \}$.
   The equation
   \begin{equation*}
   F_{ \{1, \, 3 \, / \, 2, \, 4 \}} = -    F_{ \{1, \, 2 \, / \, 3, \, 4 \}}  -    F_{ \{1, \, 4 \, / \, 2, \, 3 \}} 
   \end{equation*}
   gives rise to the crossing resolution
   
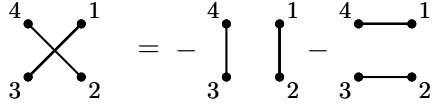
\begin{center}

\begin{tikzpicture}[scale=0.5]

    \newcommand{\squa}
        {
        \foreach \i in {1,...,4}
                {
                \coordinate (p\i) at (-90*\i + 135:\r);
                \filldraw(p\i) circle (\pointradius pt);
                \node[] at (-90*\i + 135:1.5) {\footnotesize $\i$};
                }
        }
            
    \def\r{1}           % pentagon radius
    \def\pointradius{3} % points radius
    
      \foreach \a/\b/\c/\d [count=\shft from 0] in   
       {1/3/2/4}
        {
        \begin{scope}[xshift=100*\shft,yshift=0]
            \squa
            \foreach \i in {1,...,4}
                \draw[thick] (p\a) -- (p\b);
                \draw[thick] (p\c) -- (p\d);
            \squa
        \end{scope}
        }

    \foreach \a/\b/\c/\d [count=\shft from 0] in   
       {
        1/2/3/4,
        1/4/2/3}
        {
        \begin{scope}[xshift=150+100*\shft,yshift=0]
            \squa
            \foreach \i in {1,...,4}
                \draw[thick] (p\a) -- (p\b);
                \draw[thick] (p\c) -- (p\d);
            \squa
        \end{scope}
        }

	\node at (3.5,0) {$-$};
	\node at (7,0) {$-$};
	\node at (2.5,0) {$=$};

    \end{tikzpicture}
    
    \end{center}
    where the RHS is the image $p(\pi)$.
    The reader may be disturbed by the global minus sign appearing in $p(\pi)$.
    This is accounted for by Proposition~\ref{prop:sn-action-on-F-and-f}
    which yields
    \begin{equation*}
   s_2 \cdot F_{ \{1, \, 2 \, / \, 3, \, 4 \}} = \sign(s_2) \cdot F_{\pi} = - F_{\pi}
   \end{equation*}
   Since the expression $s_2 \cdot F_{ \{1, \, 2 \, / \, 3, \, 4 \}}$ is positive in the $F$-basis,
   the resolution $p(\pi)$ must be negative in the skein basis.

   For a general chord diagram $\pi \in \Pi(2n)$, the resolution $p(\pi)$ agrees with the `classical'
   chord diagram crossing resolution up to a global sign. To show that coefficients other than $\pm 1$
   can occur in this expansion, let $\pi = \{ 1, \, 5 \, / \, 2, \, 6 \, / \, 3, \, 7 \, / \, 4, \, 8 \}$ be the 
   `asterisk of order 4'.
   Resolving the crossings in $\pi$ results in 
   
      
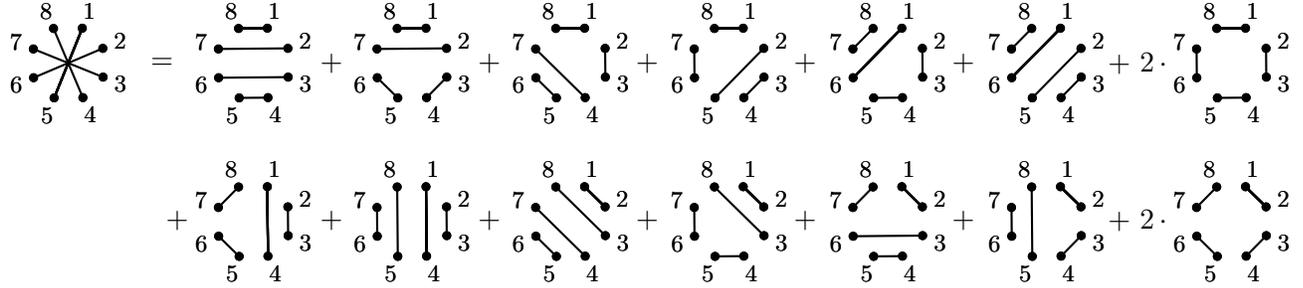
\begin{center}

\begin{tikzpicture}[scale=0.5]

    \newcommand{\squa}
        {
        \foreach \i in {1,...,8}
                {
                \coordinate (p\i) at (-45*\i + 113:\r);
                \filldraw(p\i) circle (\pointradius pt);
                \node[] at (-45*\i + 113:1.5) {\footnotesize $\i$};
                }
        }
            
    \def\r{1}           % pentagon radius
    \def\pointradius{3} % points radius
    
      \foreach \a/\b/\c/\d/\e/\f/\g/\h [count=\shft from 0] in   
       {1/2/3/4/5/6/7/8}
        {
        \begin{scope}[xshift=100*\shft,yshift=0]
            \squa
            \foreach \i in {1,...,8}
                \draw[thick] (p\a) -- (p\e);
                \draw[thick] (p\b) -- (p\f);
                \draw[thick] (p\c) -- (p\g);
                \draw[thick] (p\d) -- (p\h);
            \squa
        \end{scope}
        }
        
        \node at (2.5,0) {$=$};
         \node at (7,0) {$+$};
          \node at (11.2,0) {$+$};
            \node at (15.4,0) {$+$};
            \node at (19.6,0) {$+$};
            \node at (23.8,0) {$+$};
             \node at (28.5,0) {$+ \, \,  2\,  \cdot \, $};

               \node at (2.9,-4.2) {$+$};
             
         \node at (7,-4.2) {$+$};
          \node at (11.2,-4.2) {$+$};
            \node at (15.4,-4.2) {$+$};
            \node at (19.6,-4.2) {$+$};
            \node at (23.8,-4.2) {$+$};
             \node at (28.5,-4.2) {$+ \, \,  2\,  \cdot \, $};
        
              \foreach \a/\b/\c/\d/\e/\f/\g/\h [count=\shft from 0] in   
       {1/2/3/4/5/6/7/8}
        {
        \begin{scope}[xshift=140+100*\shft,yshift=0]
            \squa
            \foreach \i in {1,...,8}
                \draw[thick] (p\a) -- (p\h);
                \draw[thick] (p\b) -- (p\g);
                \draw[thick] (p\c) -- (p\f);
                \draw[thick] (p\d) -- (p\e);
            \squa
        \end{scope}
        }
        
       \foreach \a/\b/\c/\d/\e/\f/\g/\h [count=\shft from 0] in   
       {1/2/3/4/5/6/7/8}
        {
        \begin{scope}[xshift=140+100*\shft,yshift=-120]
            \squa
            \foreach \i in {1,...,8}
                \draw[thick] (p\a) -- (p\d);
                \draw[thick] (p\b) -- (p\c);
                \draw[thick] (p\e) -- (p\f);
                \draw[thick] (p\g) -- (p\h);
            \squa
        \end{scope}
        }
    
      \foreach \a/\b/\c/\d/\e/\f/\g/\h [count=\shft from 0] in   
       {1/2/3/4/5/6/7/8}
        {
        \begin{scope}[xshift=260+100*\shft,yshift=0]
            \squa
            \foreach \i in {1,...,8}
                \draw[thick] (p\a) -- (p\h);
                \draw[thick] (p\b) -- (p\g);
                \draw[thick] (p\c) -- (p\d);
                \draw[thick] (p\f) -- (p\e);
            \squa
        \end{scope}
        }
        
               \foreach \a/\b/\c/\d/\e/\f/\g/\h [count=\shft from 0] in   
       {1/2/3/4/5/6/7/8}
        {
        \begin{scope}[xshift=260+100*\shft,yshift=-120]
            \squa
            \foreach \i in {1,...,8}
                \draw[thick] (p\a) -- (p\d);
                \draw[thick] (p\b) -- (p\c);
                \draw[thick] (p\e) -- (p\h);
                \draw[thick] (p\g) -- (p\f);
            \squa
        \end{scope}
        }
   
      \foreach \a/\b/\c/\d/\e/\f/\g/\h [count=\shft from 0] in   
       {1/2/3/4/5/6/7/8}
        {
        \begin{scope}[xshift=380+100*\shft,yshift=0]
            \squa
            \foreach \i in {1,...,8}                
                \draw[thick] (p\a) -- (p\h);
                \draw[thick] (p\b) -- (p\c);
                \draw[thick] (p\d) -- (p\g);
                \draw[thick] (p\e) -- (p\f);
            \squa
        \end{scope}
        }    
       
        \foreach \a/\b/\c/\d/\e/\f/\g/\h [count=\shft from 0] in   
       {1/2/3/4/5/6/7/8}
        {
        \begin{scope}[xshift=380+100*\shft,yshift=-120]
            \squa
            \foreach \i in {1,...,8}
                \draw[thick] (p\a) -- (p\b);
                \draw[thick] (p\c) -- (p\h);
                \draw[thick] (p\d) -- (p\g);
                \draw[thick] (p\e) -- (p\f);
            \squa
        \end{scope}
        }
        
         \foreach \a/\b/\c/\d/\e/\f/\g/\h [count=\shft from 0] in   
       {1/2/3/4/5/6/7/8}
        {
        \begin{scope}[xshift=500+100*\shft,yshift=0]
            \squa
            \foreach \i in {1,...,8}
                  \draw[thick] (p\a) -- (p\h);
                \draw[thick] (p\b) -- (p\e);
                \draw[thick] (p\c) -- (p\d);
                \draw[thick] (p\f) -- (p\g);
            \squa
        \end{scope}
        }    

        \foreach \a/\b/\c/\d/\e/\f/\g/\h [count=\shft from 0] in   
       {1/2/3/4/5/6/7/8}
        {
        \begin{scope}[xshift=500+100*\shft,yshift=-120]
            \squa
            \foreach \i in {1,...,8}
                \draw[thick] (p\a) -- (p\b);
                \draw[thick] (p\c) -- (p\h);
                \draw[thick] (p\d) -- (p\e);
                \draw[thick] (p\f) -- (p\g);
            \squa
        \end{scope}
        }        
        
         \foreach \a/\b/\c/\d/\e/\f/\g/\h [count=\shft from 0] in   
       {1/2/3/4/5/6/7/8}
        {
        \begin{scope}[xshift=620+100*\shft,yshift=0]
            \squa
            \foreach \i in {1,...,8}
                \draw[thick] (p\a) -- (p\f);
                \draw[thick] (p\g) -- (p\h);
                \draw[thick] (p\b) -- (p\c);
                \draw[thick] (p\d) -- (p\e);
            \squa
        \end{scope}
        }    
        
        \foreach \a/\b/\c/\d/\e/\f/\g/\h [count=\shft from 0] in   
       {1/2/3/4/5/6/7/8}
        {
        \begin{scope}[xshift=620+100*\shft,yshift=-120]
            \squa
            \foreach \i in {1,...,8}
                \draw[thick] (p\a) -- (p\b);
                \draw[thick] (p\c) -- (p\f);
                \draw[thick] (p\d) -- (p\e);
                \draw[thick] (p\g) -- (p\h);
            \squa
        \end{scope}
        }      
        
     \foreach \a/\b/\c/\d/\e/\f/\g/\h [count=\shft from 0] in   
       {1/2/3/4/5/6/7/8}
        {
        \begin{scope}[xshift=740+100*\shft,yshift=0]
            \squa
            \foreach \i in {1,...,8}
                \draw[thick] (p\a) -- (p\f);
                \draw[thick] (p\g) -- (p\h);
                \draw[thick] (p\b) -- (p\e);
                \draw[thick] (p\d) -- (p\c);
            \squa
        \end{scope}
        }    
        
          \foreach \a/\b/\c/\d/\e/\f/\g/\h [count=\shft from 0] in   
       {1/2/3/4/5/6/7/8}
        {
        \begin{scope}[xshift=740+100*\shft,yshift=-120]
            \squa
            \foreach \i in {1,...,8}
                \draw[thick] (p\a) -- (p\b);
                \draw[thick] (p\c) -- (p\d);
                \draw[thick] (p\e) -- (p\h);
                \draw[thick] (p\f) -- (p\g);
            \squa
        \end{scope}
        }

               \foreach \a/\b/\c/\d/\e/\f/\g/\h [count=\shft from 0] in   
       {1/2/3/4/5/6/7/8}
        {
        \begin{scope}[xshift=880+100*\shft,yshift=0]
            \squa
            \foreach \i in {1,...,8}
                \draw[thick] (p\a) -- (p\h);
                \draw[thick] (p\b) -- (p\c);
                \draw[thick] (p\d) -- (p\e);
                \draw[thick] (p\f) -- (p\g);
            \squa
        \end{scope}
        }    
        
       \foreach \a/\b/\c/\d/\e/\f/\g/\h [count=\shft from 0] in   
       {1/2/3/4/5/6/7/8}
        {
        \begin{scope}[xshift=880+100*\shft,yshift=-120]
            \squa
            \foreach \i in {1,...,8}
                \draw[thick] (p\a) -- (p\b);
                \draw[thick] (p\c) -- (p\d);
                \draw[thick] (p\e) -- (p\f);
                \draw[thick] (p\h) -- (p\g);
            \squa
        \end{scope}
        }

    \end{tikzpicture}
    
    \end{center}
   \end{example}

   Our next example resolves crossings in set partitions $\pi$ which are not chord diagrams.
   The relevant $F$-polynomials and basis expansion were calculated by computer.
   %The \textsc{sage} program (CITE)
   %computes $p(\pi)$ for any $\pi \in \Pi(n)$, and was used to find the resolution below.
   %For efficiency and ease of coding, the package uses Proposition~\ref{block-operator-formulation} to
   % compute $F$-polynomials.

   \begin{example}
   {\rm (Beyond Chord Diagrams)}
   If $\pi \in \Pi(n)$ is a set partition which is not a chord diagram, the resolution $p(\pi)$
   can involve both positive and negative signs.
    For example, consider
   $\pi = \{ 1, \, 2, \, 6 \, / \, 3, \, 4, \, 8 \, / \, 5, \, 7 \} \in \Pi(8)$.
   
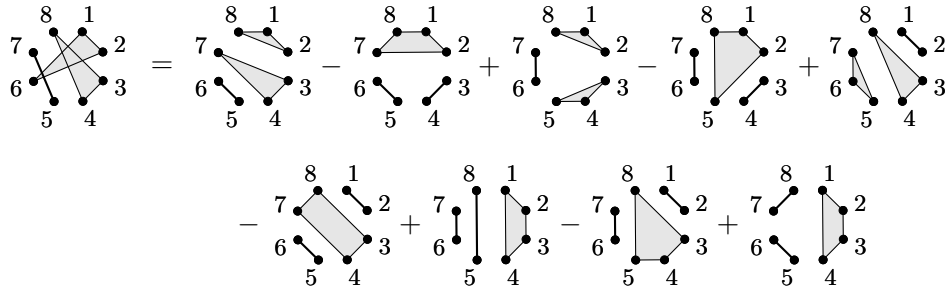
\begin{center}
   \begin{tikzpicture}[scale = 0.5]
   
       \newcommand{\squa}
        {
        \foreach \i in {1,...,8}
                {
                \coordinate (p\i) at (-45*\i + 113:\r);
                \filldraw(p\i) circle (\pointradius pt);
                \node[] at (-45*\i + 113:1.5) {\footnotesize $\i$};
                }
        }
            
    \def\r{1}           % pentagon radius
    \def\pointradius{3} % points radius
    
      \foreach \a/\b/\c/\d/\e/\f/\g/\h [count=\shft from 0] in   
       {1/2/3/4/5/6/7/8}
        {
        \begin{scope}[xshift=100*\shft,yshift=0]
            \squa
            \foreach \i in {1,...,8}             
                \fill[black!10] (p\a) -- (p\b) -- (p\f) -- (p\a);
             \fill[black!10] (p\c) -- (p\d) -- (p\h) -- (p\c);
                  \draw[thin] (p\c) -- (p\d) -- (p\h) -- (p\c);
                 \draw[thin] (p\a) -- (p\b) -- (p\f) -- (p\a);
              \draw[thick] (p\e) -- (p\g);
            \squa
        \end{scope}
        }
        
        \node at (2.5,0) {$=$};
         \node at (7,0) {$-$};
           \node at (11.2,0) {$+$};
           \node at (15.4,0) {$-$};
             \node at (19.7,0) {$+$};
             
                 \node at (4.8,-4.2) {$-$};
                 
                   \node at (9.1,-4.2) {$+$};
           \node at (13.35,-4.2) {$-$};             
      \node at (17.55,-4.2) {$+$};   
              
      \foreach \a/\b/\c/\d/\e/\f/\g/\h [count=\shft from 0] in   
       {1/2/3/4/5/6/7/8}
        {
        \begin{scope}[xshift=140+100*\shft,yshift=0]
            \squa
            \foreach \i in {1,...,8}             
                \fill[black!10] (p\a) -- (p\b) -- (p\h) -- (p\a);
             \fill[black!10] (p\c) -- (p\d) -- (p\g) -- (p\c);
                  \draw[thin] (p\c) -- (p\d) -- (p\g) -- (p\c);
                 \draw[thin] (p\a) -- (p\b) -- (p\h) -- (p\a);
              \draw[thick] (p\e) -- (p\f);
            \squa
        \end{scope}
        }
        
    \foreach \a/\b/\c/\d/\e/\f/\g/\h [count=\shft from 0] in   
       {1/2/3/4/5/6/7/8}
        {
        \begin{scope}[xshift=200+100*\shft,yshift=-120]
            \squa
            \foreach \i in {1,...,8}             
                \fill[black!10] (p\c) -- (p\d) -- (p\g) -- (p\h) -- (p\c);
                 \draw[thin] (p\c) -- (p\d) -- (p\g) -- (p\h) -- (p\c);
              \draw[thick] (p\a) -- (p\b);
                \draw[thick] (p\e) -- (p\f);
            \squa
        \end{scope}
        }
   
        \foreach \a/\b/\c/\d/\e/\f/\g/\h [count=\shft from 0] in   
       {1/2/3/4/5/6/7/8}
        {
        \begin{scope}[xshift=260+100*\shft,yshift=0]
            \squa
            \foreach \i in {1,...,8}             
                \fill[black!10] (p\a) -- (p\b) -- (p\g) -- (p\h) -- (p\a);
                 \draw[thin] (p\a) -- (p\b) -- (p\g) -- (p\h) -- (p\a);
              \draw[thick] (p\c) -- (p\d);
                \draw[thick] (p\e) -- (p\f);
            \squa
        \end{scope}
        }
        
       \foreach \a/\b/\c/\d/\e/\f/\g/\h [count=\shft from 0] in   
       {1/2/3/4/5/6/7/8}
        {
        \begin{scope}[xshift=320+100*\shft,yshift=-120]
            \squa
            \foreach \i in {1,...,8}             
                \fill[black!10] (p\a) -- (p\b) -- (p\c) -- (p\d) -- (p\a);
                 \draw[thin] (p\a) -- (p\b) -- (p\c) -- (p\d) -- (p\a);
              \draw[thick] (p\e) -- (p\h);
                \draw[thick] (p\f) -- (p\g);
            \squa
        \end{scope}
        }

         \foreach \a/\b/\c/\d/\e/\f/\g/\h [count=\shft from 0] in   
       {1/2/3/4/5/6/7/8}
        {
        \begin{scope}[xshift=380+100*\shft,yshift=0]
            \squa
            \foreach \i in {1,...,8}             
                \fill[black!10] (p\a) -- (p\b) -- (p\h) -- (p\a);
             \fill[black!10] (p\c) -- (p\d) -- (p\e) -- (p\c);
                  \draw[thin] (p\c) -- (p\d) -- (p\e) -- (p\c);
                 \draw[thin] (p\a) -- (p\b) -- (p\h) -- (p\a);
              \draw[thick] (p\f) -- (p\g);
            \squa
        \end{scope}
        }
        
       \foreach \a/\b/\c/\d/\e/\f/\g/\h [count=\shft from 0] in   
       {1/2/3/4/5/6/7/8}
        {
        \begin{scope}[xshift=440+100*\shft,yshift=-120]
            \squa
            \foreach \i in {1,...,8}             
                \fill[black!10] (p\c) -- (p\d) -- (p\e) -- (p\h) -- (p\c);
                 \draw[thin] (p\c) -- (p\d) -- (p\e) -- (p\h) -- (p\c);
              \draw[thick] (p\a) -- (p\b);
                \draw[thick] (p\f) -- (p\g);
            \squa
        \end{scope}
        }

       \foreach \a/\b/\c/\d/\e/\f/\g/\h [count=\shft from 0] in   
       {1/2/3/4/5/6/7/8}
        {
        \begin{scope}[xshift=500+100*\shft,yshift=0]
            \squa
            \foreach \i in {1,...,8}             
                \fill[black!10] (p\a) -- (p\b) -- (p\e) -- (p\h) -- (p\a);
                 \draw[thin] (p\a) -- (p\b) -- (p\e) -- (p\h) -- (p\a);
              \draw[thick] (p\c) -- (p\d);
                \draw[thick] (p\f) -- (p\g);
            \squa
        \end{scope}
        }

        \foreach \a/\b/\c/\d/\e/\f/\g/\h [count=\shft from 0] in   
       {1/2/3/4/5/6/7/8}
        {
        \begin{scope}[xshift=620+100*\shft,yshift=0]
            \squa
            \foreach \i in {1,...,8}             
                \fill[black!10] (p\c) -- (p\d) -- (p\h) -- (p\c);
             \fill[black!10] (p\e) -- (p\f) -- (p\g) -- (p\e);
                  \draw[thin] (p\c) -- (p\d) -- (p\h) -- (p\c);
                 \draw[thin] (p\e) -- (p\f) -- (p\g) -- (p\e);
              \draw[thick] (p\a) -- (p\b);
            \squa
        \end{scope}
        }
        
               \foreach \a/\b/\c/\d/\e/\f/\g/\h [count=\shft from 0] in   
       {1/2/3/4/5/6/7/8}
        {
        \begin{scope}[xshift=560+100*\shft,yshift=-120]
            \squa
            \foreach \i in {1,...,8}             
                \fill[black!10] (p\a) -- (p\b) -- (p\c) -- (p\d) -- (p\a);
                 \draw[thin] (p\a) -- (p\b) -- (p\c) -- (p\d) -- (p\a);
              \draw[thick] (p\e) -- (p\f);
                \draw[thick] (p\g) -- (p\h);
            \squa
        \end{scope}
        }

   \end{tikzpicture}
   \end{center}
   \end{example}

   \subsection{Two-block crossing resolution}
    Finding a combinatorial or algebraic interpretation of the coefficients 
    $c_{\mu}$  in $p(\pi) = \sum_{\mu \in \NC(n)} c_{\mu} \cdot \mu$
    for general set partitions $\pi \in \Pi(n)$ is an open problem. 
    However, we can give such an interpretation when the set partition 
    $\pi = \{ A \, / \, B \}$ consists of just two blocks.
    This `quadratic relation' may be useful in finding algebraic interpretations of the 
    skein modules; see Problem~\ref{applications-problem} and the discussion thereafter.

    To state our resolution for two-block set partitions, we need some notation.
    If $1 \leq i, j \leq r$, we let $[i, j]_n$ denote the (closed) cyclic interval from $i$ to $j$
    in the cycle $(1, 2, \dots, n)$.
    Explicitly, we have
    \begin{equation}
    [i,j]_n = \begin{cases}
    \{i, i+1, \dots, j-1, j \} & i \leq j \\
    \{i, i+1, \dots, n, 1, 2, \dots, j-1, j \} & i > j
    \end{cases}
    \end{equation}
    If $\pi = \{ A \, / \, B \}$ is a two-block set partition of $[n]$, there exist unique maximal nonempty
    cyclic intervals $A_1, A_2, \dots, A_m$ and $B_1, B_2, \dots, B_m$ in $(1, 2, \dots, n)$ such that
    \begin{equation}
    A = A_1 \sqcup A_2 \sqcup \cdots \sqcup A_m \quad \quad \text{and} \quad \quad
    B = B_1 \sqcup B_2 \sqcup \cdots \sqcup B_m
    \end{equation}
    and $(A_1, B_1, A_2, B_2, \dots, A_m, B_m)$ is cyclically sequential.
    The set partition $\pi = \{ A \, / \, B \}$ is noncrossing if and only if $m < 2$.

    As an example of these concepts, let 
    $\pi = \{ A \, / \, B \} \in \Pi(16)$ where
    \begin{equation*}
    A = \{ 1, \, 2, \, 4,  \, 8, \, 9, \, 10, \, 12, \, 13, \, 14, \, 15, \, 16 \}  \quad \text{and} \quad
    B = \{ 3, \, 5, \, 6, \, 7, \, 11 \}
    \end{equation*}
    We have the disjoint union decompositions
    \begin{equation*}
    A = A_1 \sqcup A_2 \sqcup A_3 \quad \text{and} \quad B = B_1 \sqcup B_2 \sqcup B_3
    \end{equation*}
    where the sets
       \begin{equation*}
    A_1 = \{1, \, 2, \, 12, \, 13, \,  14, \, 15, \, 16 \}, \quad A_2 = \{4\}, \quad A_3 = \{8, \, 9, \, 10 \}
    \end{equation*}
   and
   \begin{equation*}
   B_1 = \{ 3 \}, \quad B_2 = \{5, \, 6, \, 7 \}, \quad B_3 = \{11 \}
   \end{equation*}
   are all cyclic intervals and the sets $(A_1, B_1, A_2, B_2, A_3, B_3)$ are cyclically sequential.

    \begin{proposition}
    \label{quadratic-relation}
    Let $\pi = \{ A \, / \, B \} \in \Pi(n)$ and the decompositions 
    $A = A_1 \sqcup A_2 \sqcup \cdots \sqcup A_m$ and 
    $B = B_1 \sqcup B_2 \sqcup \cdots \sqcup B_m$ be as above. We have
    \begin{equation}
    p(\pi) = \sum_{(S,T)} \epsilon(S,T) \cdot \{ S \, / \, T \}
    \end{equation}
    where the sum is over all two-block noncrossing
    set partitions $\{ S \, / \, T \} \in \NC(n)$
    % where both $S$ and $T$
    %are unions of cyclically consecutive sets in $(A_1, B_1, A_2, B_2, \dots, A_m, B_m)$
    and the coefficient $\epsilon(S,T) \in \{+1, 0, -1 \}$ is 
    \begin{equation}
    \epsilon(S,T) = \begin{cases}
    0 & \text{if $|S| < 2$ or $|T| < 2$,} \\
    +1 & \begin{array}{c}\text{if $|S|, |T| \geq 2$ and $S, T$ are 
    both unions of an {\bf odd} number} \\ \text{ of sets in 
    $(A_1, B_1, A_2, B_2, \dots, A_m, B_m)$,}
    \end{array} \\
    -1 & \begin{array}{c}\text{if $|S|, |T| \geq 2$ and $S, T$ are 
    both unions of an {\bf even} number} \\ \text{ of sets in 
    $(A_1, B_1, A_2, B_2, \dots, A_m, B_m)$,}
    \end{array} 
    \end{cases}
    \end{equation}
    \end{proposition}

   We give an example of Proposition~\ref{quadratic-relation} before proving it.
    Suppose $\pi = \{ A \, / \, B \}$ is as before the statement of the proposition so that $m = 3$.
    The sets $S, T$ involved in the expansion of $p(\pi)$ are complementary cyclic intervals 
    in $(A_1, B_1, A_2, B_2, A_3, B_3)$.
    %We have 
    %\begin{multline*}
    %p(\pi) = \{ A_1 \, / \, A_{23} \sqcup B_{123} \} + 
    %\{ B_2 \, / \, A_{123} \sqcup B_{13} \} + 
    %\{ A_3 \, / \, A_{23} \sqcup B_{123} \}  \\
    %- \{ A_1 \sqcup B_1 \, / \, A_{23} \sqcup B_{23} \}
    %- \{ B_1 \sqcup A_2 \, / \, A_{13} \sqcup B_{23} \}
    % - \{ A_2 \sqcup B_2 \, / \, A_{13} \sqcup B_{13} \} \\
    % - \{ B_2 \sqcup A_3 \, / \, A_{12} \sqcup B_{13} \}
    %- \{ A_3 \sqcup B_3 \, / \, A_{12} \sqcup B_{12} \}
    % - \{ B_3 \sqcup A_1 \, / \, A_{23} \sqcup B_{12} \} \\
  %\hspace{1.31in}   + \{ A_{12} \sqcup B_1 \, / \, A_3 \sqcup B_{23} \} 
  %   + \{ A_{2} \sqcup B_{12} \, / \, A_{13} \sqcup B_3 \} 
  %   + \{ A_{23} \sqcup B_2 \, / \, A_1 \sqcup B_{13} \}   
  %  \end{multline*}
  %  where we abbreviate
  %  \begin{equation*}
  %  A_{i_1, \dots, i_r} := A_{i_1} \sqcup \cdots \sqcup A_{i_r} \quad \text{and} \quad
  %  B_{i_1, \dots, i_r} := B_{i_1} \sqcup \cdots \sqcup B_{i_r}
  %  \end{equation*}
  %  and the terms $\{ B_1 \, / \, A_{123} \sqcup B_{23} \}$,
  %  $\{ A_2 \, / \, A_{13} \sqcup B_{123} \}$, and
  %  $\{ B_3\, / \, A_{123} \sqcup B_{12} \}$ do not appear because 
  %  $B_1, A_2,$ and $B_3$ are singletons.
    Collapsing the cyclic intervals $A_i, B_i$ to points, this resolution $p(\pi)$ is shown
    in Figure~\ref{fig:resolve}.
    The partitions $\{B_1 \, /  [n] - B_1 \}, \{ A_2 \, / \, [n] - A_2 \},$ and 
    $\{ B_3 \, / \, [n] - B_3 \}$ do not appear because the sets $B_1, A_2,$ and $B_3$ 
    are singletons so that $\epsilon = 0$ for these terms.
    We now prove Proposition~\ref{quadratic-relation}.
  
  \begin{proof}
  This is a more complicated version of the proof of Theorem~\ref{operator-theorem}.
  When $m < 2$, the set partition $\pi$ is noncrossing, we have $p(\pi) = \pi$ and the proposition follows,
  so we assume $m \geq 2$.
  In particular, neither of the blocks of $\pi$ are singletons.
  
  Define four disjoint nonempty subsets $I,J,S,T \subseteq [n]$ by
  \begin{equation}
  I := A_1, \quad J := B_1, \quad S:= A_2 \sqcup A_3 \sqcup \cdots \sqcup A_m, \quad 
  T := B_2 \sqcup B_3 \sqcup \cdots \sqcup B_m
  \end{equation}
  We claim the following identity of linear endomorphisms of $\wedge \{ \Theta_n, \Xi_n \}$:
  \begin{multline}
  \label{psi-quadratic}
  \psi_{I \sqcup S} \circ \psi_{J \sqcup T} + 
    \psi_{I \sqcup J} \circ \psi_{S \sqcup T} + 
      \psi_{I \sqcup T} \circ \psi_{J \sqcup S} \\  - 
        \psi_{I} \circ \psi_{J \sqcup S \sqcup T} - 
          \psi_{J} \circ \psi_{I \sqcup S \sqcup T} - 
            \psi_{S} \circ \psi_{I \sqcup J \sqcup T} - 
            \psi_{T} \circ \psi_{I \sqcup J \sqcup S}  = 0
  \end{multline}
  Expressed in terms of the $\psi_{S,T}$ operators, the LHS of Equation~\eqref{psi-quadratic} is
  \begin{multline}
  \label{quadratic-LHS}
  (\psi_I + \psi_{I,S} + \psi_S) \circ (\psi_J + \psi_{J,T} + \psi_T) + 
  (\psi_I + \psi_{I,J} + \psi_J) \circ (\psi_S + \psi_{S,T} + \psi_T) \\ +
  (\psi_I + \psi_{I,T} + \psi_T) \circ (\psi_J + \psi_{J,S} + \psi_S) \\ - 
  \psi_I \circ (\psi_J + \psi_S + \psi_T + \psi_{J,S} + \psi_{J,T} + \psi_{S,T}) -
   \psi_J \circ (\psi_I + \psi_S + \psi_T + \psi_{I,S} + \psi_{I,T} + \psi_{S,T}) \\ -
     \psi_S \circ (\psi_I + \psi_J + \psi_T + \psi_{I,J} + \psi_{I,T} + \psi_{J,T}) -
      \psi_T \circ (\psi_I + \psi_J + \psi_S + \psi_{I,J} + \psi_{I,S} + \psi_{J,S})
  \end{multline}
  which simplifies (by Lemma~\ref{psi-operators-commute}) to
  \begin{equation}
  \label{new-quadratic-LHS}
  \psi_{I,S} \circ \psi_{J,T} + \psi_{I,J} \circ \psi_{S,T} + \psi_{I,T} \circ \psi_{J,S} 
  \end{equation}
  and the claimed Equation~\eqref{psi-quadratic} will be proved if we can show that 
  \eqref{new-quadratic-LHS} vanishes as an operator on $\wedge \{ \Theta_n, \Xi_n \}$.
  To show this, let $f \in \wedge \{ \Theta_n, \Xi_n \}$. The image of $f$ under 
  the operator \eqref{new-quadratic-LHS} is 
  \begin{equation}
  \label{half-quadratic}
  \frac{1}{2} \sum_{(a,b,c,d)} \xi_a \cdot (\theta_b \odot ( \xi_c  \cdot (\theta_{d} \odot f)))
  \end{equation}
  where $(a,b,c,d)$ range over all quadruples which have precisely one element in each
  of $I,J,S,T$.  As in the proof of Proposition~\ref{block-operator-formulation},
  the terms in \eqref{half-quadratic} corresponding to $(a,b,c,d)$ and $(c,b,a,d)$ cancel
  so that \eqref{half-quadratic} vanishes and Equation~\eqref{psi-quadratic} is proven.

  Equation~\eqref{psi-quadratic} and Theorem~\ref{operator-theorem}
  prove the proposition immediately when $m = 2$;
  we may rearrange its terms as
  \begin{multline}
  \label{psi-quadratic-new}
  \psi_{I \sqcup S} \circ \psi_{J \sqcup T} = 
  -  \psi_{I \sqcup J} \circ \psi_{S \sqcup T} -
      \psi_{I \sqcup T} \circ \psi_{J \sqcup S} \\  +
        \psi_{I} \circ \psi_{J \sqcup S \sqcup T} + 
          \psi_{J} \circ \psi_{I \sqcup S \sqcup T} + 
            \psi_{S} \circ \psi_{I \sqcup J \sqcup T} +
            \psi_{T} \circ \psi_{I \sqcup J \sqcup S} 
  \end{multline}
  and apply both sides of Equation~\eqref{psi-quadratic-new} to $\theta_1 \theta_2 \cdots \theta_n$.
  When $m > 2$, we may still apply both sides of Equation~\eqref{psi-quadratic-new}
  to $\theta_1 \theta_2 \cdots \theta_n$. 
  The LHS evaluation is $F_{\pi}$ while the RHS evaluation is a linear combination of 
  $F_{\mu}$'s for singleton-free (recall that $\psi_S = 0$ when $S$ is a singleton) two-block
  set partitions $\mu \in \Pi(n)$ with strictly shorter sizes $m$
  of their cyclic interval decompositions. The proposition follows from induction on $m$.
  \end{proof}

    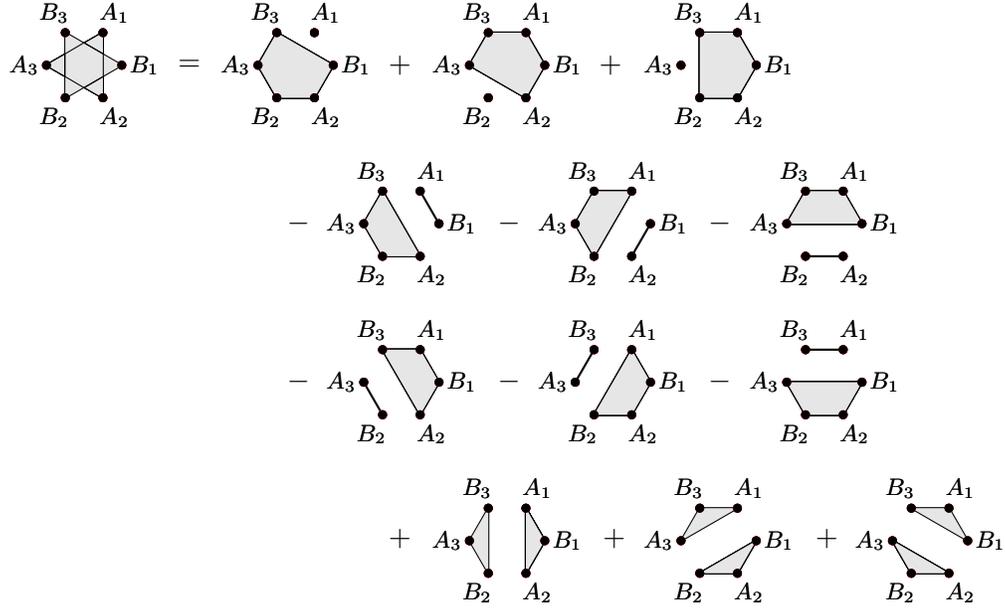
\begin{figure}
    \begin{center}
    \begin{tikzpicture}[scale = 0.5]
        \newcommand{\squa}
        {
        \foreach \i in {1,...,6}
                {
                \coordinate (p\i) at (-60*\i + 120:\r);
                \filldraw(p\i)[color = red] circle (\pointradius pt);
                }
                \node[] at (-60*1 + 120:1.6) {\footnotesize $A_1$};
                \filldraw[color = black] (-60*1 + 120:1)  circle (\pointradius pt);
                \node[] at (-60*2 + 120:1.6) {\footnotesize $B_1$};
                \filldraw[color = black] (-60*2 + 120:1)  circle (\pointradius pt);
                \node[] at (-60*3 + 120:1.6) {\footnotesize $A_2$};
                \filldraw[color = black] (-60*3 + 120:1)  circle (\pointradius pt);
                \node[] at (-60*4 + 120:1.6) {\footnotesize $B_2$};
                \filldraw[color = black] (-60*4 + 120:1)  circle (\pointradius pt);
                \node[] at (-60*5 + 120:1.6) {\footnotesize $A_3$};
                \filldraw[color = black] (-60*5 + 120:1)  circle (\pointradius pt);
                \node[] at (-60*6 + 120:1.6) {\footnotesize $B_3$};
                \filldraw[color = black] (-60*6 + 120:1)  circle (\pointradius pt);
        }
            
    \def\r{1}           % pentagon radius
    \def\pointradius{3} % points radius
    
      \foreach \a/\b/\c/\d/\e/\f in   
       {1/2/3/4/5/6}
        {
        \begin{scope}[xshift=0,yshift=0]
            \squa
            \foreach \i in {1,...,6}
                  \draw[thick] (p\a) -- (p\c) -- (p\e) -- (p\a);
		 \fill[black!10] (p\a) -- (p\c) -- (p\e) -- (p\a) -- cycle;
                \draw[thick] (p\b) -- (p\d) -- (p\f) -- (p\b);
                \fill[black!10] (p\b) -- (p\d) -- (p\f) -- (p\b) -- cycle;
                \draw[] (-0.51,-0.3) -- (-0.51,0.3) -- (0,0.6) -- (0.51,0.3) -- (0.51,-0.3) -- (0,-0.6) -- (-0.51,-0.3) -- cycle;
               % \fill[black!40] (-0.51,-0.3) -- (-0.51,0.3) -- (0,0.6) -- (0.51,0.3) -- (0.51,-0.3) -- (0,-0.6) -- (-0.51,-0.3) -- cycle;
            \squa
        \end{scope}}
        
        \node at (2.8,0) {$=$};
        \node at (8.4,0) {$+$};
        \node at (14.0,0) {$+$};
        \node at (5.7,-4.2) {$-$};
        \node at (11.3,-4.2) {$-$};
        \node at (16.9,-4.2) {$-$};
        \node at (5.7,-8.4) {$-$};
        \node at (11.3,-8.4) {$-$};
        \node at (16.9,-8.4) {$-$};
        \node at (8.4,-12.6) {$+$};
        \node at (14.1,-12.6) {$+$};
        \node at (19.75,-12.6) {$+$};
        
              \foreach \a/\b/\c/\d/\e/\f in   
       {1/2/3/4/5/6}
        {
        \begin{scope}[xshift=160,yshift=0]
            \squa
            \foreach \i in {1,...,6}
                  \draw[thick] (p\b) -- (p\c) -- (p\d) -- (p\e) -- (p\f) --  (p\b);
                \fill[black!10] (p\b) --  (p\c) -- (p\d) -- (p\e) -- (p\f) -- (p\b) -- cycle;
            \squa
        \end{scope}}

       \foreach \a/\b/\c/\d/\e/\f in   
       {1/2/3/4/5/6}
        {
        \begin{scope}[xshift=320,yshift=0]
            \squa
            \foreach \i in {1,...,6}
                  \draw[thick] (p\a) -- (p\b) -- (p\c) -- (p\e) -- (p\f) -- (p\a);
                \fill[black!10] (p\a) --  (p\b) -- (p\c) -- (p\e) -- (p\f) -- (p\a) -- cycle;
            \squa
        \end{scope}}

      \foreach \a/\b/\c/\d/\e/\f in   
       {1/2/3/4/5/6}
        {
        \begin{scope}[xshift=480,yshift=0]
            \squa
            \foreach \i in {1,...,6}
                  \draw[thick] (p\a) -- (p\b) -- (p\c) -- (p\d) -- (p\f) -- (p\a);
                \fill[black!10] (p\a) --  (p\b) -- (p\c) -- (p\d) -- (p\f) -- (p\a) -- cycle;
            \squa
        \end{scope}}

               \foreach \a/\b/\c/\d/\e/\f in   
       {1/2/3/4/5/6}
        {
        \begin{scope}[xshift=240,yshift=-120]
            \squa
            \foreach \i in {1,...,6}
                \draw[thick]  (p\c) -- (p\d) -- (p\e) -- (p\f) -- (p\c) -- cycle;
                \fill[black!10]  (p\c) -- (p\d) -- (p\e) -- (p\f) -- (p\c) -- cycle;
                \draw[thick] (p\a) -- (p\b);
            \squa
        \end{scope}}

       \foreach \a/\b/\c/\d/\e/\f in   
       {1/2/3/4/5/6}
        {
        \begin{scope}[xshift=400,yshift=-120]
            \squa
            \foreach \i in {1,...,6}
                \draw[thick]  (p\a) -- (p\d) -- (p\e) -- (p\f) -- (p\a) -- cycle;
                \fill[black!10]  (p\a) -- (p\d) -- (p\e) -- (p\f) -- (p\a) -- cycle;
                \draw[thick] (p\b) -- (p\c);
            \squa
        \end{scope}}

       \foreach \a/\b/\c/\d/\e/\f in   
       {1/2/3/4/5/6}
        {
        \begin{scope}[xshift=560,yshift=-120]
            \squa
            \foreach \i in {1,...,6}
                \draw[thick]  (p\a) -- (p\b) -- (p\e) -- (p\f) -- (p\a) -- cycle;
                \fill[black!10]  (p\a) -- (p\b) -- (p\e) -- (p\f) -- (p\a) -- cycle;
                \draw[thick] (p\c) -- (p\d);
            \squa
        \end{scope}}

        \foreach \a/\b/\c/\d/\e/\f in   
       {1/2/3/4/5/6}
        {
        \begin{scope}[xshift=240,yshift=-240]
            \squa
            \foreach \i in {1,...,6}
                \draw[thick]  (p\a) -- (p\b) -- (p\c) -- (p\f) -- (p\a) -- cycle;
                \fill[black!10]  (p\a) -- (p\b) -- (p\c) -- (p\f) -- (p\a) -- cycle;
                \draw[thick] (p\d) -- (p\e);
            \squa
        \end{scope}}

       \foreach \a/\b/\c/\d/\e/\f in   
       {1/2/3/4/5/6}
        {
        \begin{scope}[xshift=400,yshift=-240]
            \squa
            \foreach \i in {1,...,6}
                \draw[thick]  (p\a) -- (p\b) -- (p\c) -- (p\d) -- (p\a) -- cycle;
                \fill[black!10]  (p\a) -- (p\b) -- (p\c) -- (p\d) -- (p\a) -- cycle;
                \draw[thick] (p\e) -- (p\f);
            \squa
        \end{scope}}

       \foreach \a/\b/\c/\d/\e/\f in   
       {1/2/3/4/5/6}
        {
        \begin{scope}[xshift=560,yshift=-240]
            \squa
            \foreach \i in {1,...,6}
                \draw[thick]  (p\b) -- (p\c) -- (p\d) -- (p\e) -- (p\b) -- cycle;
                \fill[black!10]  (p\b) -- (p\c) -- (p\d) -- (p\e) -- (p\b) -- cycle;
                \draw[thick] (p\a) -- (p\f);
            \squa
        \end{scope}}

      \foreach \a/\b/\c/\d/\e/\f in   
       {1/2/3/4/5/6}
        {
        \begin{scope}[xshift=320,yshift=-360]
            \squa
            \foreach \i in {1,...,6}
                  \draw[thick] (p\a) -- (p\b) -- (p\c) -- (p\a) -- cycle;
                  \fill[black!10] (p\a) -- (p\b) -- (p\c) -- (p\a) -- cycle;
                  \draw[thick] (p\d) -- (p\e) -- (p\f) -- (p\d) -- cycle;
                  \fill[black!10] (p\d) -- (p\e) -- (p\f) -- (p\d) -- cycle;                                   
            \squa
        \end{scope}}

       \foreach \a/\b/\c/\d/\e/\f in   
       {1/2/3/4/5/6}
        {
        \begin{scope}[xshift=480,yshift=-360]
            \squa
            \foreach \i in {1,...,6}
                  \draw[thick] (p\b) -- (p\c) -- (p\d) -- (p\b) -- cycle;
                  \fill[black!10] (p\b) -- (p\c) -- (p\d) -- (p\b) -- cycle;
                  \draw[thick] (p\a) -- (p\e) -- (p\f) -- (p\a) -- cycle;
                  \fill[black!10] (p\a) -- (p\e) -- (p\f) -- (p\a) -- cycle;                                   
            \squa
        \end{scope}}

       \foreach \a/\b/\c/\d/\e/\f in   
       {1/2/3/4/5/6}
        {
        \begin{scope}[xshift=640,yshift=-360]
            \squa
            \foreach \i in {1,...,6}
                  \draw[thick] (p\c) -- (p\d) -- (p\e) -- (p\c) -- cycle;
                  \fill[black!10] (p\c) -- (p\d) -- (p\e) -- (p\c) -- cycle;
                  \draw[thick] (p\a) -- (p\b) -- (p\f) -- (p\a) -- cycle;
                  \fill[black!10] (p\a) -- (p\b) -- (p\f) -- (p\a) -- cycle;                                   
            \squa
        \end{scope}}

    \end{tikzpicture}
    \end{center} 
        \caption{The noncrossing resolution of a two-block  set partition $\pi = \{ A \, / \, B \}$ where
        $A = A_1 \sqcup A_2 \sqcup A_3$, $B = B_1 \sqcup B_2 \sqcup B_3$, the sets $B_1, A_2, B_3$
         are singletons, and the sets $A_1, B_2, A_3$ have more than one element.}
         \label{fig:resolve}
    \end{figure}

   \subsection{Equivariance and symmetries}
    The resolution projection $p$ 
    intertwines the skein action on $\CC[\NC(n)]$ with the sign-twisted permutation action of $\symm_n$
    on $\CC[\Pi(n)]$.
   
   \begin{theorem}
   \label{projection-equivariance}
   Endow $\CC[\Pi(n)]$ with the $\symm_n$-action 
   \begin{equation}
   w \star \pi := \sign(w) \cdot w(\pi) 
   \end{equation}
   and 
   $\CC[\NC(n)]$ with the skein action. The  projection
   $p: \CC[\Pi(n)] \twoheadrightarrow \CC[\NC(n)]$ is $\symm_n$-equivariant.
   \end{theorem}
   
   \begin{proof}
   By Proposition~\ref{prop:sn-action-on-F-and-f}, the surjection $\CC[\Pi(n)] \twoheadrightarrow W(n)$
   given by $\pi \mapsto F_{\pi}$ is $\symm_n$-equivariant, so the theorem is a direct 
   consequence of Definition~\ref{projection-map-definition}.
   \end{proof}

   Theorem~\ref{projection-equivariance} characterizes $p$ as follows.
   Suppose $f: \CC[\Pi(n)] \rightarrow \CC[\NC(n)]$ is a linear map such that 
   \begin{itemize}
   \item we have $f(s_i \star \pi) = s_i \cdot f(\pi)$ for all $1 \leq i \leq n-1$ and $\pi \in \Pi(n)$ and
   \item for any partition $\lambda \vdash n$ we have $f(\pi_{\lambda}) = \pi_{\lambda}$ where
   \begin{equation*}
   \pi_{\lambda} := \{ 1, \, 2, \, \cdots, \, \lambda_1 \, / \, \lambda_1 + 1, \, \lambda_1 + 2, \, 
   \cdots, \, \lambda_1 + \lambda_2 \, / \, \cdots \}
   \end{equation*}
   \end{itemize}
   Since any $\pi \in \Pi(n)$ is $\symm_n$-conjugate to a set partition of the form $\pi_{\lambda}$,
   we have $f = p$.
   This characterization may be helpful in finding other applications of the map $p$;
   see Problem~\ref{applications-problem}.

   The set $\NC(n)$ of noncrossing partitions has a dihedral group of `global' symmetries.
   That is,
   if $w_0 \in \symm_n$ is the long element satisfying $w_0(i) = n-i+1$ 
   and $c = (1, 2, \dots, n) \in \symm_n$ is the long cycle,
   then $\NC(n)$ is closed under the action of $w_0$ and $c$.
   The following result is a generalization of \cite[Prop. 6.1, Thm. 6.4]{Rhoades}; thanks 
   to fermions its proof is much shorter.
   
   \begin{corollary}
   \label{global-symmetries}
   For any $\pi \in \NC(n)$, with respect to the skein action $w_0, c \in \symm_n$ act on $\pi$ by
   \begin{equation}
   w_0 \cdot \pi = (-1)^{{n \choose 2}} \cdot w_0(\pi) \quad \quad \text{ and } \quad \quad 
   c \cdot \pi = (-1)^{n-1} \cdot c(\pi)
   \end{equation}
   \end{corollary}
   
   \begin{proof}
   Apply Theorem~\ref{projection-equivariance} and the identities $\sign(w_0) = (-1)^{{n \choose 2}}$
   and $\sign(c) = (-1)^{n-1}$.
   \end{proof}
   
   We give an example to illustrate Corollary~\ref{global-symmetries} and the power
   of the equivariance  Theorem~\ref{projection-equivariance}.
   
   \begin{example}  {\em (Global symmetries)}
   Let $n = 6$ and 
   $\pi = \{ 1, \, 5, \, 6 \, / \, 2, \, 4 \, / \, 3 \}$.  In order to compute $c \cdot \pi$ under the skein 
   action, we write $c = s_1 s_2 s_3 s_4 s_5$ and apply each adjacent transposition successively
   using the skein relations.
   
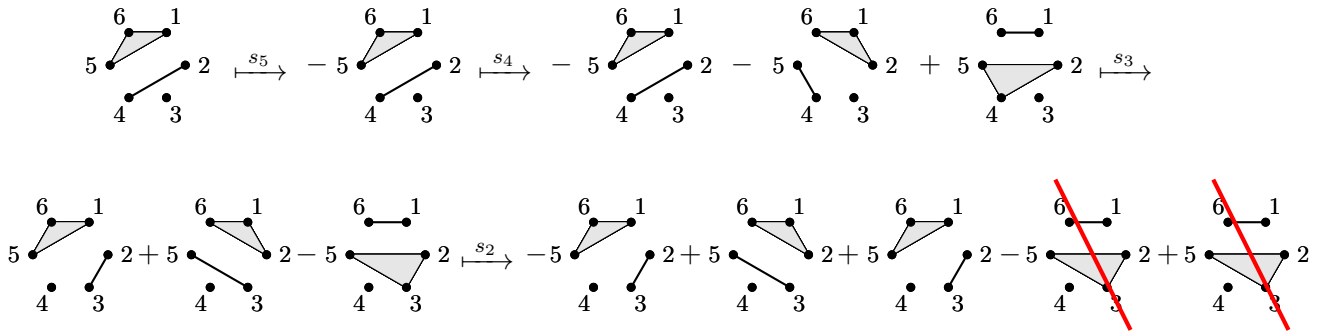
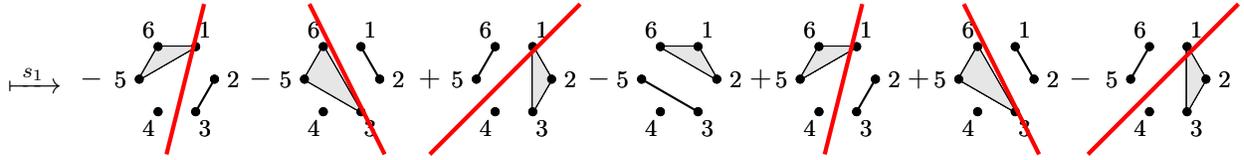
\begin{center}
   \begin{tikzpicture}[scale=0.5]

    \newcommand{\squa}
        {
        \foreach \i in {1,...,6}
                {
                \coordinate (p\i) at (-60*\i + 120:\r);
                \filldraw(p\i) circle (\pointradius pt);
                \node[] at (-60*\i + 120:1.5) {\footnotesize $\i$};
                }
        }
            
    \def\r{1}           % pentagon radius
    \def\pointradius{3} % points radius
    
      \foreach \a/\b/\c/\d/\e/\f [count=\shft from 0] in   
       {1/2/3/4/5/6}
        {
        \begin{scope}[xshift=100*\shft,yshift=0]
            \squa
            \foreach \i in {1,...,6}
                  \draw[thick] (p\a) -- (p\e) -- (p\f) -- (p\a) -- cycle;
		 \fill[black!10] (p\a) -- (p\e) -- (p\f) -- (p\a) -- cycle;
                \draw[thick] (p\b) -- (p\d);
            \squa
        \end{scope}
        
        \node at (3,0) {$\xmapsto{ \, \, s_5 \, \,  }$};
         \node at (4.5,0) {$-$};
        
        \foreach \a/\b/\c/\d/\e/\f [count=\shft from 0] in   
       {1/2/3/4/5/6}
        {
        \begin{scope}[xshift=190+100*\shft,yshift=0]
            \squa
            \foreach \i in {1,...,6}
                  \draw[thick] (p\a) -- (p\e) -- (p\f) -- (p\a) -- cycle;
		 \fill[black!10] (p\a) -- (p\e) -- (p\f) -- (p\a) -- cycle;
                \draw[thick] (p\b) -- (p\d);
            \squa
        \end{scope}
        }
        
         \node at (9.5,0) {$\xmapsto{ \, \, s_4 \, \,  }$};
          \node at (11,0) {$-$};
           \node at (15.8,0) {$-$};
            \node at (20.8,0) {$+$};
         
        \foreach \a/\b/\c/\d/\e/\f [count=\shft from 0] in   
       {1/2/3/4/5/6}
        {
        \begin{scope}[xshift=380+100*\shft,yshift=0]
            \squa
            \foreach \i in {1,...,6}
                  \draw[thick] (p\a) -- (p\e) -- (p\f) -- (p\a) -- cycle;
		 \fill[black!10] (p\a) -- (p\e) -- (p\f) -- (p\a) -- cycle;
                \draw[thick] (p\b) -- (p\d);
            \squa
        \end{scope}
        }

        \foreach \a/\b/\c/\d/\e/\f [count=\shft from 0] in   
       {1/2/3/4/5/6}
        {
        \begin{scope}[xshift=520+100*\shft,yshift=0]
            \squa
            \foreach \i in {1,...,6}
                  \draw[thick] (p\a) -- (p\b) -- (p\f) -- (p\a) -- cycle;
		 \fill[black!10] (p\a) -- (p\b) -- (p\f) -- (p\a) -- cycle;
                \draw[thick] (p\e) -- (p\d);
            \squa
        \end{scope}
        }
        
       \foreach \a/\b/\c/\d/\e/\f [count=\shft from 0] in   
       {1/2/3/4/5/6}
        {
        \begin{scope}[xshift=660+100*\shft,yshift=0]
            \squa
            \foreach \i in {1,...,6}
                  \draw[thick] (p\b) -- (p\d) -- (p\e) -- (p\b) -- cycle;
		 \fill[black!10] (p\b) -- (p\d) -- (p\e) -- (p\b) -- cycle;
                \draw[thick] (p\a) -- (p\f);
            \squa
        \end{scope}
        }
          \node at (26,0) {$\xmapsto{ \, \, s_3 \, \,  }$};
        
        }
    
    \end{tikzpicture}
    \end{center}

   \begin{center}
   \begin{tikzpicture}[scale=0.5]
   
     \newcommand{\squa}
        {
        \foreach \i in {1,...,6}
                {
                \coordinate (p\i) at (-60*\i + 120:\r);
                \filldraw(p\i) circle (\pointradius pt);
                \node[] at (-60*\i + 120:1.5) {\footnotesize $\i$};
                }
        }
            
    \def\r{1}           % pentagon radius
    \def\pointradius{3} % points radius

        \node at (4.90,0) {$+$};
        
        \node at (9.1,0) {$-$};
        \node at (13.9,0) {$\xmapsto{ \, \, s_2 \, \,  }$};
         \node at (15.2,0) {$-$};
         
       \foreach \a/\b/\c/\d/\e/\f [count=\shft from 0] in   
       {1/2/3/4/5/6}
        {
        \begin{scope}[xshift=80+100*\shft,yshift=0]
            \squa
            \foreach \i in {1,...,6}
                  \draw[thick] (p\a) -- (p\f) -- (p\e) -- (p\a) -- cycle;
		 \fill[black!10] (p\a) -- (p\f) -- (p\e) -- (p\a) -- cycle;
                \draw[thick] (p\b) -- (p\c);
            \squa
        \end{scope}}
        
             \foreach \a/\b/\c/\d/\e/\f [count=\shft from 0] in   
       {1/2/3/4/5/6}
        {
        \begin{scope}[xshift=200+100*\shft,yshift=0]
            \squa
            \foreach \i in {1,...,6}
                  \draw[thick] (p\a) -- (p\b) -- (p\f) -- (p\a) -- cycle;
		 \fill[black!10] (p\a) -- (p\b) -- (p\f) -- (p\a) -- cycle;
                \draw[thick] (p\e) -- (p\c);
            \squa
        \end{scope}
        }
        
       \foreach \a/\b/\c/\d/\e/\f [count=\shft from 0] in   
       {1/2/3/4/5/6}
        {
        \begin{scope}[xshift=320+100*\shft,yshift=0]
            \squa
            \foreach \i in {1,...,6}
                  \draw[thick] (p\b) -- (p\c) -- (p\e) -- (p\b) -- cycle;
		 \fill[black!10] (p\b) -- (p\c) -- (p\e) -- (p\b) -- cycle;
                \draw[thick] (p\a) -- (p\f);
            \squa
        \end{scope}
        }
        
       \foreach \a/\b/\c/\d/\e/\f [count=\shft from 0] in   
       {1/2/3/4/5/6}
        {
        \begin{scope}[xshift=490+100*\shft,yshift=0]
            \squa
            \foreach \i in {1,...,6}
                  \draw[thick] (p\a) -- (p\e) -- (p\f) -- (p\a) -- cycle;
		 \fill[black!10] (p\a) -- (p\e) -- (p\f) -- (p\a) -- cycle;
                \draw[thick] (p\b) -- (p\c);
            \squa
        \end{scope}
        
         \node at (19.3,0) {$+$};
          \node at (23.5,0) {$+$};
          \node at (27.8,0) {$-$};
           \node at (32,0) {$+$};

                \foreach \a/\b/\c/\d/\e/\f [count=\shft from 0] in   
       {1/2/3/4/5/6}
        {
        \begin{scope}[xshift=610+100*\shft,yshift=0]
            \squa
            \foreach \i in {1,...,6}
                  \draw[thick] (p\a) -- (p\b) -- (p\f) -- (p\a) -- cycle;
		 \fill[black!10] (p\a) -- (p\b) -- (p\f) -- (p\a) -- cycle;
                \draw[thick] (p\e) -- (p\c);
            \squa
        \end{scope}}
        
        \foreach \a/\b/\c/\d/\e/\f [count=\shft from 0] in   
       {1/2/3/4/5/6}
        {
        \begin{scope}[xshift=730+100*\shft,yshift=0]
            \squa
            \foreach \i in {1,...,6}
                  \draw[thick] (p\a) -- (p\e) -- (p\f) -- (p\a) -- cycle;
		 \fill[black!10] (p\a) -- (p\e) -- (p\f) -- (p\a) -- cycle;
                \draw[thick] (p\b) -- (p\c);
            \squa
        \end{scope}}
        
           \foreach \a/\b/\c/\d/\e/\f [count=\shft from 0] in   
       {1/2/3/4/5/6}
        {
        \begin{scope}[xshift=850+100*\shft,yshift=0]
            \squa
            \foreach \i in {1,...,6}
                  \draw[thick] (p\b) -- (p\c) -- (p\e) -- (p\b) -- cycle;
		 \fill[black!10] (p\b) -- (p\c) -- (p\e) -- (p\b) -- cycle;
                \draw[thick] (p\a) -- (p\f);
            \squa
        \end{scope}}

                   \foreach \a/\b/\c/\d/\e/\f [count=\shft from 0] in   
       {1/2/3/4/5/6}
        {
        \begin{scope}[xshift=970+100*\shft,yshift=0]
            \squa
            \foreach \i in {1,...,6}
                  \draw[thick] (p\b) -- (p\c) -- (p\e) -- (p\b) -- cycle;
		 \fill[black!10] (p\b) -- (p\c) -- (p\e) -- (p\b) -- cycle;
                \draw[thick] (p\a) -- (p\f);
            \squa
        \end{scope}}
        
         \draw[ultra thick, red] (33.2,2) -- (35.2,-2);
         
          \draw[ultra thick, red] (29,2) -- (31,-2);
         
        }

    \end{tikzpicture}
    \end{center}

      \begin{center}
   \begin{tikzpicture}[scale=0.5]
   
     \newcommand{\squa}
        {
        \foreach \i in {1,...,6}
                {
                \coordinate (p\i) at (-60*\i + 120:\r);
                \filldraw(p\i) circle (\pointradius pt);
                \node[] at (-60*\i + 120:1.5) {\footnotesize $\i$};
                }
        }
            
    \def\r{1}           % pentagon radius
    \def\pointradius{3} % points radius

    \node at (14.5,0) {$\xmapsto{ \, \, s_1 \, \, }$};
        \node at (16,0) {$-$};

       \foreach \a/\b/\c/\d/\e/\f [count=\shft from 0] in   
       {1/2/3/4/5/6}
        {
        \begin{scope}[xshift=520+100*\shft,yshift=0]
            \squa
            \foreach \i in {1,...,6}
                  \draw[thick] (p\a) -- (p\e) -- (p\f) -- (p\a) -- cycle;
		 \fill[black!10] (p\a) -- (p\e) -- (p\f) -- (p\a) -- cycle;
                \draw[thick] (p\b) -- (p\c);
            \squa
        \end{scope}
        
         \node at (20.5,0) {$-$};
          \node at (25,0) {$+$};
          \node at (29.5,0) {$-$};
           \node at (33.8,0) {$+$};
            \node at (38,0) {$+$};
              \node at (42.3,0) {$-$};

   \foreach \a/\b/\c/\d/\e/\f [count=\shft from 0] in   
       {1/2/3/4/5/6}
        {
        \begin{scope}[xshift=645+100*\shft,yshift=0]
            \squa
            \foreach \i in {1,...,6}
                  \draw[thick] (p\c) -- (p\e) -- (p\f) -- (p\c) -- cycle;
		 \fill[black!10] (p\c) -- (p\e) -- (p\f) -- (p\c) -- cycle;
                \draw[thick] (p\a) -- (p\b);
            \squa
        \end{scope}}
        
        \foreach \a/\b/\c/\d/\e/\f [count=\shft from 0] in   
       {1/2/3/4/5/6}
        {
        \begin{scope}[xshift=775+100*\shft,yshift=0]
            \squa
            \foreach \i in {1,...,6}
                  \draw[thick] (p\a) -- (p\b) -- (p\c) -- (p\a) -- cycle;
		 \fill[black!10] (p\a) -- (p\b) -- (p\c) -- (p\a) -- cycle;
                \draw[thick] (p\f) -- (p\e);
            \squa
        \end{scope}}
        
           \foreach \a/\b/\c/\d/\e/\f [count=\shft from 0] in   
       {1/2/3/4/5/6}
        {
        \begin{scope}[xshift=900+100*\shft,yshift=0]
            \squa
            \foreach \i in {1,...,6}
                  \draw[thick] (p\a) -- (p\b) -- (p\f) -- (p\a) -- cycle;
		 \fill[black!10] (p\a) -- (p\b) -- (p\f) -- (p\a) -- cycle;
                \draw[thick] (p\c) -- (p\e);
            \squa
        \end{scope}}

                   \foreach \a/\b/\c/\d/\e/\f [count=\shft from 0] in   
       {1/2/3/4/5/6}
        {
        \begin{scope}[xshift=1020+100*\shft,yshift=0]
            \squa
            \foreach \i in {1,...,6}
                  \draw[thick] (p\a) -- (p\f) -- (p\e) -- (p\a) -- cycle;
		 \fill[black!10] (p\a) -- (p\f) -- (p\e) -- (p\a) -- cycle;
                \draw[thick] (p\b) -- (p\c);
            \squa
        \end{scope}}
        
     \foreach \a/\b/\c/\d/\e/\f [count=\shft from 0] in   
       {1/2/3/4/5/6}
        {
        \begin{scope}[xshift=1140+100*\shft,yshift=0]
            \squa
            \foreach \i in {1,...,6}
                  \draw[thick] (p\c) -- (p\e) -- (p\f) -- (p\c) -- cycle;
		 \fill[black!10] (p\c) -- (p\e) -- (p\f) -- (p\c) -- cycle;
                \draw[thick] (p\a) -- (p\b);
            \squa
        \end{scope}}
        
        \foreach \a/\b/\c/\d/\e/\f [count=\shft from 0] in   
       {1/2/3/4/5/6}
        {
        \begin{scope}[xshift=1270+100*\shft,yshift=0]
            \squa
            \foreach \i in {1,...,6}
                  \draw[thick] (p\a) -- (p\b) -- (p\c) -- (p\a) -- cycle;
		 \fill[black!10] (p\a) -- (p\b) -- (p\c) -- (p\a) -- cycle;
                \draw[thick] (p\e) -- (p\f);
            \squa
        \end{scope}}

         \draw[ultra thick, red] (39.2,2) -- (41.2,-2);
         
          \draw[ultra thick, red] (21.8,2) -- (23.8,-2);
          
         \draw[ultra thick, red] (29,2) -- (25,-2);
         
          \draw[ultra thick, red] (46.5,2) -- (42.5,-2);
          
         \draw[ultra thick, red] (36.5,2) -- (35.5,-2);
         
          \draw[ultra thick, red] (19,2) -- (18,-2);

        }

    \end{tikzpicture}
    \end{center}
    As predicted by Corollary~\ref{global-symmetries}, the set partition
    $c(\pi) = \{ 1, \, 2, \, 6 \, / \, 3, \, 5 \, / \, 4 \} \in \NC(6)$ is the only surviving term 
    with sign $(-1)^{6-1} = -1$.
    We leave it for the reader to verify   
    \begin{equation*}
    w_0 \cdot \pi = (-1)^{{6 \choose 2}} w_0(\pi) = - w_0(\pi) = - \{ 1, \, 2, \, 6 \, / \, 3, \, 5 \, / \, 4 \}
    \end{equation*} 
    directly from the skein action.
    \end{example}
   
   Corollary~\ref{global-symmetries} may be used in conjunction with 
   Corollary~\ref{frobenius-image} to obtain fixed-point counts for the action 
   of the dihedral group $\langle w_0, c \rangle$ on the set $\NC(n)$.
   This gives rise to cyclic sieving phenomena for the rotational action of $\langle c \rangle$,
   as explained in \cite{Rhoades}.
   Another immediate corollary of Theorem~\ref{projection-equivariance} is that the 
   skein action respects `local symmetries' of $\symm_n$ on $\NC(n)$.
   The following result is a sharpening of \cite[Cor. 7.3]{Rhoades}.
   
   \begin{corollary}
   \label{local-symmetries}
   Let $w \in \symm_n$ and $\pi \in \NC(n)$ be such that $w(\pi) \in \NC(n)$. We have the skein action 
   \begin{equation}
   w \cdot \pi = \sign(w) \cdot w(\pi)
   \end{equation}
   \end{corollary}

   \begin{example}
   {\em (Local symmetries)}
  Suppose $\pi = \{ 1, \, 5, \, 6 \, / \, 2, \, 4 \, / \, 3 \} \in \NC(6)$ as above.
  Letting $w = \begin{pmatrix} 1 & 2 & 3 & 4 & 5 & 6 \\ 3 & 5 & 2 & 4 & 1 & 6  \end{pmatrix} \in \symm_6$, 
  the set partition
  $w(\pi) = \{ 3, \, 1, \, 6 \, / \, 5, \, 4 \, / \, 2 \} \in \Pi(6)$ is noncrossing. Also we have $\sign(w) = -1$.
  Using the decomposition
  $w = s_2 s_1 s_4 s_3 s_2 s_3 s_4$ we calculate
  
  
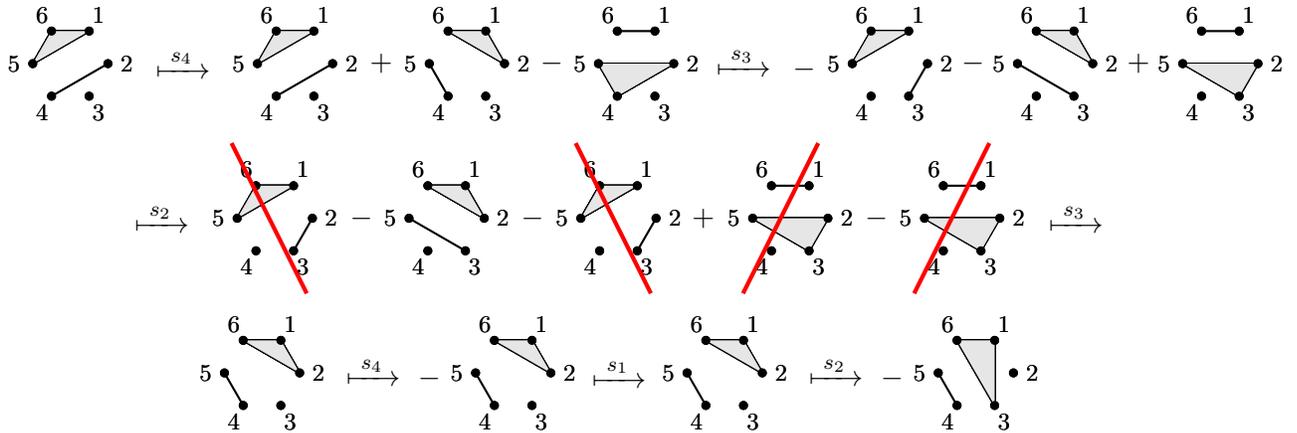
\begin{center}
  \begin{tikzpicture}[scale = 0.5]
  
       \newcommand{\squa}
        {
        \foreach \i in {1,...,6}
                {
                \coordinate (p\i) at (-60*\i + 120:\r);
                \filldraw(p\i) circle (\pointradius pt);
                \node[] at (-60*\i + 120:1.5) {\footnotesize $\i$};
                }
        }
            
    \def\r{1}           % pentagon radius
    \def\pointradius{3} % points radius
    
       \foreach \a/\b/\c/\d/\e/\f [count=\shft from 0] in   
       {1/2/3/4/5/6}
        {
        \begin{scope}[xshift=0+100*\shft,yshift=0]
            \squa
            \foreach \i in {1,...,6}
                  \draw[thick] (p\a) -- (p\e) -- (p\f) -- (p\a) -- cycle;
		 \fill[black!10] (p\a) -- (p\e) -- (p\f) -- (p\a) -- cycle;
                \draw[thick] (p\b) -- (p\d);
            \squa
        \end{scope}}
        
       \node at (3,0) {$\xmapsto{ \, \, s_4 \, \, }$};

       \foreach \a/\b/\c/\d/\e/\f [count=\shft from 0] in   
       {1/2/3/4/5/6}
        {
        \begin{scope}[xshift=170+100*\shft,yshift=0]
            \squa
            \foreach \i in {1,...,6}
                  \draw[thick] (p\a) -- (p\e) -- (p\f) -- (p\a) -- cycle;
		 \fill[black!10] (p\a) -- (p\e) -- (p\f) -- (p\a) -- cycle;
                \draw[thick] (p\b) -- (p\d);
            \squa
        \end{scope}}

        \node at (8.27,0) {$+$};
        
        \foreach \a/\b/\c/\d/\e/\f [count=\shft from 0] in   
       {1/2/3/4/5/6}
        {
        \begin{scope}[xshift=300+100*\shft,yshift=0]
            \squa
            \foreach \i in {1,...,6}
                  \draw[thick] (p\a) -- (p\b) -- (p\f) -- (p\a) -- cycle;
		 \fill[black!10] (p\a) -- (p\b) -- (p\f) -- (p\a) -- cycle;
                \draw[thick] (p\d) -- (p\e);
            \squa
        \end{scope}} 
        
        \node at (12.8,0) {$-$};       

       \foreach \a/\b/\c/\d/\e/\f [count=\shft from 0] in   
       {1/2/3/4/5/6}
        {
        \begin{scope}[xshift=428+100*\shft,yshift=0]
            \squa
            \foreach \i in {1,...,6}
                  \draw[thick] (p\b) -- (p\d) -- (p\e) -- (p\b) -- cycle;
		 \fill[black!10] (p\b) -- (p\d) -- (p\e) -- (p\b) -- cycle;
                \draw[thick] (p\a) -- (p\f);
            \squa
        \end{scope}}
        
        \node at (18.5, 0) {$\xmapsto{ \, \, s_3 \, \, } \, \, \, -$};

        \foreach \a/\b/\c/\d/\e/\f [count=\shft from 0] in   
       {1/2/3/4/5/6}
        {
        \begin{scope}[xshift=620+100*\shft,yshift=0]
            \squa
            \foreach \i in {1,...,6}
                  \draw[thick] (p\a) -- (p\e) -- (p\f) -- (p\a) -- cycle;
		 \fill[black!10] (p\a) -- (p\e) -- (p\f) -- (p\a) -- cycle;
                \draw[thick] (p\b) -- (p\c);
            \squa
        \end{scope}}
        
        \node at (24,0) {$-$};

        \foreach \a/\b/\c/\d/\e/\f [count=\shft from 0] in   
       {1/2/3/4/5/6}
        {
        \begin{scope}[xshift=745+100*\shft,yshift=0]
            \squa
            \foreach \i in {1,...,6}
                  \draw[thick] (p\a) -- (p\b) -- (p\f) -- (p\a) -- cycle;
		 \fill[black!10] (p\a) -- (p\b) -- (p\f) -- (p\a) -- cycle;
                \draw[thick] (p\c) -- (p\e);
            \squa
        \end{scope}}
        
         \node at (28.4,0) {$+$};
         
          \foreach \a/\b/\c/\d/\e/\f [count=\shft from 0] in   
       {1/2/3/4/5/6}
        {
        \begin{scope}[xshift=870+100*\shft,yshift=0]
            \squa
            \foreach \i in {1,...,6}
                  \draw[thick] (p\b) -- (p\c) -- (p\e) -- (p\b) -- cycle;
		 \fill[black!10] (p\b) -- (p\c) -- (p\e) -- (p\b) -- cycle;
                \draw[thick] (p\a) -- (p\f);
            \squa
        \end{scope}}

  \end{tikzpicture}
  \end{center}
  
  \begin{center}
  \begin{tikzpicture} [scale = 0.5]
  
  \node at (-3,0) {$\xmapsto{\, \, s_2 \, \, }$};
  
         \newcommand{\squa}
        {
        \foreach \i in {1,...,6}
                {
                \coordinate (p\i) at (-60*\i + 120:\r);
                \filldraw(p\i) circle (\pointradius pt);
                \node[] at (-60*\i + 120:1.5) {\footnotesize $\i$};
                }
        }
            
    \def\r{1}           % pentagon radius
    \def\pointradius{3} % points radius
    
       \foreach \a/\b/\c/\d/\e/\f [count=\shft from 0] in   
       {1/2/3/4/5/6}
        {
        \begin{scope}[xshift=0+100*\shft,yshift=0]
            \squa
            \foreach \i in {1,...,6}
                  \draw[thick] (p\a) -- (p\e) -- (p\f) -- (p\a) -- cycle;
		 \fill[black!10] (p\a) -- (p\e) -- (p\f) -- (p\a) -- cycle;
                \draw[thick] (p\b) -- (p\c);
            \squa
        \end{scope}}
        
        \node at (2.3,0) {$-$};

       \foreach \a/\b/\c/\d/\e/\f [count=\shft from 0] in   
       {1/2/3/4/5/6}
        {
        \begin{scope}[xshift=130+100*\shft,yshift=0]
            \squa
            \foreach \i in {1,...,6}
                  \draw[thick] (p\a) -- (p\b) -- (p\f) -- (p\a) -- cycle;
		 \fill[black!10] (p\a) -- (p\b) -- (p\f) -- (p\a) -- cycle;
                \draw[thick] (p\c) -- (p\e);
            \squa
        \end{scope}}
        
         \node at (6.85,0) {$-$};
  
        \foreach \a/\b/\c/\d/\e/\f [count=\shft from 0] in   
       {1/2/3/4/5/6}
        {
        \begin{scope}[xshift=260+100*\shft,yshift=0]
            \squa
            \foreach \i in {1,...,6}
                  \draw[thick] (p\a) -- (p\e) -- (p\f) -- (p\a) -- cycle;
		 \fill[black!10] (p\a) -- (p\e) -- (p\f) -- (p\a) -- cycle;
                \draw[thick] (p\b) -- (p\c);
            \squa
        \end{scope}}
        
        \node at (11.4,0) {$+$};
        
         \foreach \a/\b/\c/\d/\e/\f [count=\shft from 0] in   
       {1/2/3/4/5/6}
        {
        \begin{scope}[xshift=390+100*\shft,yshift=0]
            \squa
            \foreach \i in {1,...,6}
                  \draw[thick] (p\b) -- (p\c) -- (p\e) -- (p\b) -- cycle;
		 \fill[black!10] (p\b) -- (p\c) -- (p\e) -- (p\b) -- cycle;
                \draw[thick] (p\a) -- (p\f);
            \squa
        \end{scope}}
        
        \node at (16,0) {$-$};
        
          \foreach \a/\b/\c/\d/\e/\f [count=\shft from 0] in   
       {1/2/3/4/5/6}
        {
        \begin{scope}[xshift=520+100*\shft,yshift=0]
            \squa
            \foreach \i in {1,...,6}
                  \draw[thick] (p\b) -- (p\c) -- (p\e) -- (p\b) -- cycle;
		 \fill[black!10] (p\b) -- (p\c) -- (p\e) -- (p\b) -- cycle;
                \draw[thick] (p\a) -- (p\f);
            \squa
        \end{scope}}

        \node at (21.3, 0)  {$\xmapsto{\, \, s_3 \, \, }$};
        
        \draw[ultra thick, red] (19,2) -- (17,-2);
        
  \draw[ultra thick, red] (14.45,2) -- (12.45,-2);

    \draw[ultra thick, red] (8,2) -- (10,-2);

       \draw[ultra thick, red] (-1.15,2) -- (0.85,-2);

  \end{tikzpicture}
  \end{center}

  \begin{center}
  \begin{tikzpicture}[scale = 0.5]

       \newcommand{\squa}
        {
        \foreach \i in {1,...,6}
                {
                \coordinate (p\i) at (-60*\i + 120:\r);
                \filldraw(p\i) circle (\pointradius pt);
                \node[] at (-60*\i + 120:1.5) {\footnotesize $\i$};
                }
        }
            
    \def\r{1}           % pentagon radius
    \def\pointradius{3} % points radius

      \foreach \a/\b/\c/\d/\e/\f [count=\shft from 0] in   
       {1/2/3/4/5/6}
        {
        \begin{scope}[xshift=0+100*\shft,yshift=0]
            \squa
            \foreach \i in {1,...,6}
                  \draw[thick] (p\a) -- (p\b) -- (p\f) -- (p\a) -- cycle;
		 \fill[black!10] (p\a) -- (p\b) -- (p\f) -- (p\a) -- cycle;
                \draw[thick] (p\d) -- (p\e);
            \squa
        \end{scope}}

        \node at (3.5,0) {$\xmapsto{ \, \, s_4 \, \, } \, \, - $};

      \foreach \a/\b/\c/\d/\e/\f [count=\shft from 0] in   
       {1/2/3/4/5/6}
        {
        \begin{scope}[xshift=190+100*\shft,yshift=0]
            \squa
            \foreach \i in {1,...,6}
                  \draw[thick] (p\a) -- (p\b) -- (p\f) -- (p\a) -- cycle;
		 \fill[black!10] (p\a) -- (p\b) -- (p\f) -- (p\a) -- cycle;
                \draw[thick] (p\d) -- (p\e);
            \squa
        \end{scope}}

        \node at (9.5,0) {$\xmapsto{ \, \, s_1 \, \, }$};
        
      \foreach \a/\b/\c/\d/\e/\f [count=\shft from 0] in   
       {1/2/3/4/5/6}
        {
        \begin{scope}[xshift=350+100*\shft,yshift=0]
            \squa
            \foreach \i in {1,...,6}
                  \draw[thick] (p\a) -- (p\b) -- (p\f) -- (p\a) -- cycle;
		 \fill[black!10] (p\a) -- (p\b) -- (p\f) -- (p\a) -- cycle;
                \draw[thick] (p\d) -- (p\e);
            \squa
        \end{scope}}

        \node at (15.8,0) {$\xmapsto{ \, \, s_2 \, \, } \, \, -$};
 
       \foreach \a/\b/\c/\d/\e/\f [count=\shft from 0] in   
       {1/2/3/4/5/6}
        {
        \begin{scope}[xshift=540+100*\shft,yshift=0]
            \squa
            \foreach \i in {1,...,6}
                  \draw[thick] (p\a) -- (p\c) -- (p\f) -- (p\a) -- cycle;
		 \fill[black!10] (p\a) -- (p\c) -- (p\f) -- (p\a) -- cycle;
                \draw[thick] (p\d) -- (p\e);
            \squa
        \end{scope}}

  \end{tikzpicture}
  \end{center}
  as predicted by Corollary~\ref{local-symmetries}.
  \end{example}

 \subsection{Combinatorial crossing resolution}
 Proposition~\ref{quadratic-relation} gives a combinatorial way to calculate the resolution
 $p(\pi)$ for $\pi \in \Pi(n)$ which does not use $F$-fermions.
 The idea is to resolve a fixed pair of crossing blocks $A, B \in \pi$, resulting 
 in a linear combination of partitions which are `less crossing' than $\pi$.
 To describe this procedure, we need notation.
 
Let $S$ be a finite set with a disjoint union decomposition $S = I \sqcup J$.
If $\pi_I$ is a set partition of $I$ and $\pi_J$ is a set partition of $J$, the (disjoint) union
$\pi_I \sqcup \pi_J$ is a set partition of $S$.
Conversely, if $\pi$ is a set partition of $S$ and $I, J$
are unions of blocks of $\pi$ we have the restrictions $\pi \mid_I$ and $\pi \mid_J$ of $\pi$ to $I$ and $J$.

If a finite set $S$ has a total order 
{\em noncrossing partitions} and {\em cyclic intervals} of $S$ are defined in the natural way.
If $\pi$ is a set partition of $S$, two blocks $A, B \in \pi$ are said to {\em cross} if 
$\pi \mid_{A \sqcup B}$ is not noncrossing.
Roughly speaking, the next result states that we can locally resolve the crossing of $A, B$
as in Proposition~\ref{quadratic-relation} while leaving the other blocks of $\pi$ unchanged.

\begin{corollary}
\label{quadratic-extension}
Let $\pi \in \Pi(n)$ and let $A, B$ be blocks of $\pi$ which cross. Write 
$A = A_1 \sqcup A_2 \sqcup \cdots \sqcup A_m$ and
 $B = B_1 \sqcup B_2 \sqcup \cdots \sqcup B_m$
where the $A_i$ and $B_i$ are maximal nonempty cyclic intervals in the ordered set 
$A \sqcup B$ such that the sequence $(A_1, B_1, A_2, B_2, \dots, A_m, B_m)$ is cyclically sequential.
Write $C := [n] - (A \sqcup B)$ for the union of the other blocks of $\pi$. We have 
\begin{equation}
\label{transition-equation}
p(\pi) = \sum_{(S,T)} \epsilon(S,T) \cdot p \left( \{S \, / \, T \} \sqcup \pi \mid_C \right)
\end{equation}
where the sum is over all two-block noncrossing partitions $\{ S \, / \, T \}$ of $A \sqcup B$ 
and the coefficient $\epsilon(S,T)$ is defined as in Proposition~\ref{quadratic-relation}.
\end{corollary}

\begin{proof}
When $[n] = A \sqcup B$ this is precisely Proposition~\ref{quadratic-relation}. 
More generally,  if 
$C = \{ C_1 \, / \, C_2 \, / \, \cdots \, / \, C_r \}$ the desired equation will follow from
the operator identity
\begin{equation}
\label{desired-rho-identity}
\rho_A \circ \rho_B \circ \rho_{C_1} \circ \rho_{C_2} \circ \cdots \circ \rho_{C_r} = 
\sum_{(S,T)} \epsilon(S,T) \cdot \rho_S \circ \rho_T 
\circ \rho_{C_1} \circ \rho_{C_2} \circ \cdots \circ \rho_{C_r} 
\end{equation}
where the conditions on $(S,T)$ are the same as in the statement.
As all sets $A, B, S, T$ appearing in Equation~\eqref{desired-rho-identity}
have size $> 1$, we have $\rho_A = \psi_A, \rho_B = \psi_B, \rho_S = \psi_S,$ and 
$\rho_T = \psi_T$ so that 
Equation~\eqref{desired-rho-identity} is implied by 
\begin{equation}
\label{desired-psi-identity}
\psi_A \circ \psi_B = 
\sum_{(S,T)} \epsilon(S,T) \cdot \psi_S \circ \psi_T
\end{equation}
Equation~\eqref{desired-psi-identity} is proven in the same way as Proposition~\ref{quadratic-relation}.
\end{proof}

By Corollary~\ref{quadratic-extension}, if $\pi \in \Pi(n)$ and $A, B$ are blocks of $\pi$ which cross, we 
may write $p(\pi)$ as a sum of elements of the form $\pm p(\pi')$ where $\pi$ is obtained from $\pi$
by replacing $A, B$ with a new pair $S, T$ of noncrossing, nonsingleton blocks.
If all of the resulting partitions $\pi' \in \Pi(n)$ were noncrossing, this would give the 
resolution $p(\pi) \in \CC[\NC(n)]$, but this need not be the case in general.
However, the $\pi'$ involved are `less crossing' than $\pi$.
We define the {\em tangle} $\tan(\pi)$ of a set partition $\pi = \{ B_1 \, / \, B_2 \, / \, \cdots \, / \, B_k \}$ by
\begin{equation}
\tan(\pi) := | \{ 1 \leq i < j \leq k \,:\, \text{the blocks $B_i$ and $B_j$ cross} \} |
\end{equation}
so that $\pi$ is noncrossing if and only if $\tan(\pi) = 0$.
Certainly any index pair of blocks $S, T$ in the RHS of 
Equation~\eqref{transition-equation} do not cross,
whereas the blocks $A, B$ on the LHS do.
For the remaining blocks, we have the following

\begin{observation}
\label{untangle-observation}
Let $\pi \in \Pi(n)$ be a set partition, let $A, B \in \pi$ be blocks which cross, and let $S,T$
index a term on the
RHS of Equation~\eqref{transition-equation}. If $D$ is a block of $\pi$ other than $A, B$, then
\begin{itemize}
\item if $D$ crosses just one of $A, B$, then $D$ crosses at most one of $S, T$, and
\item if $D$ crosses neither $A$ nor $B$, then $D$ crosses neither $S$ nor $T$.
\end{itemize}
\end{observation}

Given $\pi \in \Pi(n)$,
Corollary~\ref{quadratic-extension} yields a  
`greedy algorithm' to calculate $p(\pi) \in \CC[\NC(n)]$.
\begin{enumerate}
\item If $\pi$ is noncrossing, then $p(\pi) = \pi$. Otherwise, arbitrarily select two blocks $A, B$ of $\pi$ which
cross.
\item Write 
$p(\pi) = \sum_{(S,T)} \epsilon(S,T) \cdot p \left( \{S \, / \, T \} \sqcup \pi \mid_C \right)$
as in Equation~\eqref{transition-equation}.
Go back to Step 1 for each $\{S \, / \, T \} \sqcup \pi \mid_C$ appearing on the RHS.
\end{enumerate}
By Observation~\ref{untangle-observation}, any set partition $\pi' = \{ S \, / \, T \} \sqcup \pi \mid_C$
involved in the RHS of Equation~\eqref{transition-equation} satisfies $\tan(\pi') < \tan(\pi)$ so this 
algorithm terminates.

\subsection{Quadratic ideals $I$ and $J$}
In this subsection we recast our work in the setting
of {\em commutative} rings, ideals, and quotients.
To any nonempty subset $B \subseteq [n]$ we associate a commuting variable $y_B$
and let
$R := \CC \left[ y_B \,:\, \varnothing \neq B \subseteq [n] \, \right]$ be the rank $2^n - 1$ polynomial ring in
these variables.

We introduce two quadratic
ideals $I, J \subset R$ as follows. For any pair $(A,B)$ of nonempty disjoint subsets of $[n]$
which cross, the ideal $I$ has a generator
\begin{equation}
\label{i-generator}
y_A  y_B - \sum_{(S,T)} \epsilon(S,T) \cdot y_S  y_T
\end{equation}
where the pairs $(S,T)$ and $\epsilon(S,T)$ are as in 
Proposition~\ref{quadratic-relation}.

The ideal $I$ has formal similarities with the {\em Pl\"ucker ideals} of Schubert calculus.
These are  quadratic ideals in polynomial rings with variables $\Delta_B$
indexed by nonempty subsets $B \subseteq [n]$.
The variable $\Delta_B$ corresponds to the top-justified minor in an $n \times n$ matrix
$X = (x_{i,j})$ of variables with column set $B$. The analog of 
the generator \eqref{i-generator} is another signed quadratic expression
given by a determinantal identity due to Sylvester in 1851.
See \cite[Sec. 8.1]{Fulton} for a definition of the Pl\"ucker relations.

Although the relations \eqref{i-generator} look somewhat like Pl\"ucker relations, there are two important 
differences.
The sizes $\{ |A| , |B| \}$ and $\{ |S|, |T| \}$ of subsets involved in any Pl\"ucker relation are the same,
but this homogeneity does not usually hold for the expressions \eqref{i-generator}.
Furthermore, the sets pairs $(A,B)$ of a given product $\Delta_A \Delta_B$ appearing in a Pl\"ucker relation
can overlap, but all set pairs $(A,B)$ and $(S,T)$ appearing in \eqref{i-generator} are disjoint.

%One reason the Pl\"ucker relations are important is that they generate the
%ideal corresponding to the (irreducible) type A flag varieties.
%In particular, the Pl\"ucker ideal is prime. 
%The large number of variables in $R$ has impeded the study of the following question.
%To get a positive answer, it may be helpful to reverse the signs $\epsilon(S,T)$
%appearing in \eqref{i-generator}.

%\begin{question}
%\label{I-is-prime}
%Is the ideal $I$  radical? Is $I$ prime?
%\end{question}

The quotient $R/I$ has infinite vector space dimension. 
To get Artinian quotients, we introduce the ideal $J \subset R$ given by
\begin{equation}
J = \langle y_A y_B \,:\, A \cap B \neq \varnothing \rangle
\end{equation}
Since the sets $B$ indexing the variables $y_B$ are nonempty, we have $y_B^2 \in J$ always.
The quotient $R/J$ has a basis consisting of monomials
$y_{B_1} \cdots y_{B_k}$ for which the sets $B_1, \dots, B_k$ are pairwise disjoint.
Quotienting $R$ by the larger ideal $I + J$ yields a more interesting basis of noncrossing  
disjoint subsets.

\begin{proposition}
\label{quadratic-basis}
The quotient ring $R/(I+J)$ has a basis $\BBB$ consisting of monomials 
$y_{B_1} \cdots y_{B_k}$ for which the sets $B_1, \dots, B_k \subseteq [n]$ are pairwise
disjoint and noncrossing.
\end{proposition}

\begin{proof}
Let $V \subset R$ be the vector subspace spanned by all monomials 
$y_{B_1} \cdots y_{B_k}$ for which $B_1, \dots, B_k$ are pairwise disjoint.
The monomials in $\BBB$ and the generators of $I$ lie in $V$.
Since $R = V \oplus J$ as vector spaces it suffices to show that $\BBB$
descends to a basis of $V/(I \cap V)$.
To do this, we introduce direct sum decompositions as follows.

For any subset $U \subseteq [n]$, let $V_U \subseteq V$ be the subspace with basis
given by monomials $y_{B_1} \cdots y_{B_k}$ with $B_1 \sqcup \cdots \sqcup B_k = U$.
We have vector space direct sums
\begin{equation}
V = \bigoplus_{U \subseteq [n]} V_U \quad \quad \text{and} \quad \quad
I \cap V = \bigoplus_{U \subseteq [n]} I \cap V_U
\end{equation}
where the second direct sum is justified as $I$ is spanned by elements of the form
\begin{equation}
\left( y_A  y_B - \sum_{(S,T)} \epsilon(S,T) \cdot y_S  y_T \right) \cdot y_{C_1} \cdots y_{C_r}
\end{equation}
where the union $A \cup B \cup C_1 \cup \cdots \cup C_r$ and the unions
$S \cup T \cup C_1 \cup \cdots \cup C_r$ all equal the same set $U \subseteq [n]$.
This gives rise to an identification
\begin{equation}
V/(I \cap V) = \bigoplus_{U \subseteq [n]} V_U / (I \cap V_U)
\end{equation}
If we set $\BBB_U := \BBB \cap V_U$ we have the disjoint union
\begin{equation}
 \BBB = \bigsqcup_{U \subseteq [n]} \BBB_U
 \end{equation}  
It suffices to prove the following   

{\bf Claim:}
{\em $\BBB_U$ descends to a basis of
$V_U / (I \cap V_U)$ for any $U \subseteq [n]$.}

Fix a subset $U \subseteq [n]$ and consider the 
exterior algebra $\wedge \{ \Theta_U, \Xi_U \}$ over the set $\{ \theta_u, \xi_u \,:\, u \in U \}$
of variables indexed by $U$.
Lemma~\ref{block-operators-commute} states that the block operators satisfy 
$\rho_A \circ \rho_B = \rho_B \circ \rho_A$ for all subsets $A, B \subseteq [n]$.
This endows $\wedge \{ \Theta_U, \Xi_U \}$ with an $R$-module structure via
\begin{equation}
y_B \cdot f := \begin{cases} \rho_B(f) & B \subseteq U \\ 0 & \text{otherwise} \end{cases}
\end{equation}
Restricting from $R$ to $V_U$ gives a bilinear map
\begin{equation}
V_U \times \wedge \{ \Theta_U, \Xi_U \} \longrightarrow \wedge \{ \Theta_U, \Xi_U \}
\end{equation}
Equation~\eqref{desired-rho-identity}, applied over the variables indexed by $U$,
implies that $I \cap V_U$ acts trivially on $\wedge \{ \Theta_U, \Xi_U \}$, so we have an 
induced bilinear map
\begin{equation}
V_U/(V_U \cap I) \times \wedge \{ \Theta_U, \Xi_U \} \longrightarrow \wedge \{ \Theta_U, \Xi_U \}
\end{equation}
Theorem~\ref{upstairs-basis-theorem}, 
again applied over $\wedge \{ \Theta_U, \Xi_U \}$,
implies that 
\begin{equation}
\{ (y_{B_1} \cdots y_{B_k}) \cdot (\theta_1 \cdots \theta_n) \,:\, y_{B_1} \cdots y_{B_k} \in \BBB_U \}
\end{equation}
is linearly independent in $\wedge \{ \Theta_U, \Xi_U \}$, so $\BBB_U$ is linearly independent
in $V_U / (V_U \cap I)$. 
The fact that $\BBB_U$ spans $V_U / (V_U \cap I)$ follows from the relations 
\eqref{desired-rho-identity} and the greedy algorithm following 
Observation~\ref{untangle-observation}.
\end{proof}

Placing the generator $y_B$ corresponding to a subset $B \subseteq [n]$ in bidegree $(|B|,1)$ 
gives 
$R$
the structure of a bigraded ring $R = \bigoplus_{i,j \geq 0} R_{i,j}$.
The ideals $I, J \subseteq R$ are both bihomogeneous, as are the quotients
$R/J$ and $R/(I+J)$.
We close this section by recording their bigraded Hilbert series
\begin{align}
\Hilb(R/J; q, t) &= \sum_{m, k \geq 0}  {n \choose m}  \Stir(m,k) \cdot q^m t^k \\
\Hilb(R/(I+J); q, t) &= \sum_{m, k \geq 0}  {n \choose m}  \Nar(m,k) \cdot q^m t^k 
\end{align}
where $\Stir(m,k) := | \Pi(m,k) |$ is the {\em Stirling number of the second kind}.

\section{Fermionic diagonal coinvariants}
\label{Quotient}

Up until this point, we have studied the fermions $F_{\pi}$ and $f_{\pi}$ as members
of the exterior algebra $\wedge \{ \Theta_n, \Xi_n \}$. In this section we study their 
images in
$FDR_n = \wedge \{ \Theta_n, \Xi_n \} / \langle \wedge \{ \Theta_n, \Xi_n \}^{\symm_n}_+ \rangle$.
Three members of the defining ideal of $FDR_n$ are
\begin{equation}
\theta := \theta_1 + \cdots + \theta_n \quad \quad \xi := \xi_1 + \cdots + \xi_n  \quad \quad
\delta := \theta_1 \xi_1 + \cdots + \theta_n \xi_n
\end{equation}
where the dependence on $n$ of $\theta, \xi, \delta \in \wedge \{ \Theta_n, \Xi_n \}$ is suppressed
and will be clear from context.
In particular, we have
\begin{equation}
f_{\pi} =  \pm \xi \odot F_{\pi}
\end{equation}
for any set partition $\pi \in \Pi(n)$.
Recall that $V(n,k) \subseteq \wedge \{ \Theta_n, \Xi_n \}_{n-k,k-1}$ is the span of the set
$\{ f_{\pi} \,:\, \pi \in \NC(n,k) \}$.
The following lemma states that multiplication by $\theta$ is an injective operation on $V(n,k)$.

\begin{lemma}
\label{theta-injective}
For $1 \leq k \leq n$, the map $\theta \cdot (-): V(n,k) \rightarrow \theta \cdot V(n,k)$ given by multiplication
by $\theta$ is injective.
\end{lemma}

\begin{proof}
Theorem~\ref{upstairs-basis-theorem} states that $\{ f_{\pi} \,:\, \pi \in \NC(n,k) \}$ is a basis for $V(n,k)$
It suffices to show that $\{ \theta \cdot f_{\pi} \,:\, \pi \in \NC(n,k) \}$ is a linearly independent subset of 
$\wedge \{ \Theta_n, \Xi_n \}$. To this end, suppose 
\begin{equation}
\label{hypothetical-relation}
\sum_{\pi \in \NC(n,k)} c_{\pi} \cdot \theta \cdot f_{\pi} = 0
\end{equation}
for some coefficients $c_{\pi} \in \CC$.

Let $U_i$ and $\tau_i: \wedge \{ \Theta_n, \Xi_n \} \twoheadrightarrow U_i$ be as in the proof of
Theorem~\ref{upstairs-basis-theorem}. If $\{i\}$ is a singleton of $\pi \in \NC(n,k)$, a direct computation shows
\begin{equation}
\tau_i( \theta \cdot f_{\pi}) = \pm \xi_i \cdot  \theta^{(i)} \cdot \overline{f}_{\pi^{(i)}}
\end{equation}
where $\overline{f}_{\pi^{(i)}}$ is as in the proof of Theorem~\ref{upstairs-basis-theorem} and 
$\theta^{(i)} := \theta_1 + \cdots + \widehat{\theta_i} + \cdots + \theta_n$.
Furthermore, if $\{ i \}$ is not a singleton of $\pi$, we compute $\tau_i( \theta \cdot f_{\pi}) = 0$.
Applying $\tau_i$ to both sides of Equation~\eqref{hypothetical-relation} gives
\begin{equation}
\label{newish-hypothetical-relation}
\sum_{\substack{\pi \in \NC(n,k) \\ \text{$\{i\}$ a block of $\pi$}}} \pm c_{\pi} \cdot \xi_i \cdot \theta^{(i)} \cdot 
\overline{f}_{\pi^{(i)}} = 0
\end{equation}
which forces (since $\theta^{(i)} \cdot \overline{f}_{\pi^{(i)}}$ do not involve $\xi_i$) the relation
\begin{equation}
\label{new-hypothetical-relation}
\sum_{\substack{\pi \in \NC(n,k) \\ \text{$\{i\}$ a block of $\pi$}}} \pm c_{\pi} \cdot \theta^{(i)} \cdot 
\overline{f}_{\pi^{(i)}} = 0
\end{equation}
By induction on $n$, we conclude that $c_{\pi} = 0$ whenever $\pi$ has a singleton block, so that 
Equation~\eqref{hypothetical-relation} has the form
\begin{equation}
\label{final-hypothetical-relation}
\sum_{\pi \in \NC(n,k,0)} c_{\pi} \cdot \theta \cdot f_{\pi} = 0
\end{equation}
Since the set $\{ f_{\pi} \,:\, \pi \in \NC(n,k,0) \}$ is a basis for a flag-shaped $\symm_n$-irreducible
and the map $\theta \cdot (-)$ is a nonzero $\symm_n$-homomorphism,
the set $\{ \theta \cdot f_{\pi} \,:\, \pi \in \NC(n,k,0) \}$ must be a basis for the same flag-shaped irreducible.
Thus, the coefficients $c_{\pi}$ appearing in Equation~\eqref{final-hypothetical-relation} all vanish and the 
lemma is proven.
\end{proof}

Orellana-Zabrocki \cite{OZ} and Kim-Rhoades \cite{KR} independently proved  
\begin{equation}
\label{generating-set}
\langle \wedge \{ \Theta_n, \Xi_n \}^{\symm_n}_+ \rangle = \langle \theta, \xi, \delta \rangle
\end{equation}
as ideals in $\wedge \{ \Theta_n, \Xi_n \}$, so the defining ideal of $FDR_n$ is generated by three elements.
We exploit this fact in the proof of the following theorem.
Recall that $\overline{FDR_n} = \bigoplus_{k = 1}^{n-1} (FDR_n)_{n-k,k-1}$ is the space of 
extreme bidegrees in $FDR_n$.

\begin{theorem}
\label{quotient-basis-theorem}
The set $\{ f_{\pi} \,:\, \pi \in \NC(n,k) \}$ descends to a basis of $(FDR_n)_{n-k,k-1}$.
Consequently, the set $\{ f_{\pi} \,:\, \pi \in \NC(n) \}$ descends to a basis of 
$\overline{FDR_n}$.
\end{theorem}

\begin{proof}
 We define a subspace $U(n,k) \subseteq \wedge \{ \Theta_n, \Xi_n \}$
by 
\begin{align}
U(n,k) &:= \langle \wedge \{ \Theta_n, \Xi_n \}^{\symm_n}_+ \rangle \cap \wedge \{ \Theta_n, \Xi_n \}_{n-k,k-1} \\
&= \theta \cdot \wedge \{ \Theta_n, \Xi_n \}_{n-k-1,k-1} + 
\xi \cdot \wedge \{ \Theta_n, \Xi_n \}_{n-k,k-2} + 
\delta \cdot \wedge \{ \Theta_n, \Xi_n \}_{n-k-1,k-2} 
\end{align}
where the second line is justified by \eqref{generating-set}.

{\bf Claim:} 
{\em We have $V(n,k) \cap U(n,k) = 0$ as subspaces of $\wedge \{ \Theta_n, \Xi_n \}_{n-k,k-1}$.}

By Lemma~\ref{theta-injective}, it suffices to show 
\begin{equation}
(\theta \cdot V(n,k)) \cap (\theta \cdot U(n,k)) = 0
\end{equation}
Recall that $\langle - , - \rangle$ is the inner product on $\wedge \{ \Theta_n, \Xi_n \}$ for which
the monomial basis $\theta_S \cdot \xi_T$ is orthogonal.
We show that $ (\theta \cdot V(n,k)) \perp (\theta \cdot U(n,k))$ with respect to this inner product.

Fix a set partition $\pi \in \Pi(n,k)$.
We need only show that
\begin{center}
$\langle \theta \cdot f_{\pi}, \theta \cdot \xi \cdot \theta_S \cdot \xi_T \rangle = 0$ and 
$\langle \theta \cdot f_{\pi}, \theta \cdot \delta \cdot \theta_S \cdot \xi_T \rangle = 0$
\end{center}
for any monomial $\theta_S \cdot \xi_T$. The inner product on the left is easier to verify:
by the adjointness property in Proposition~\ref{exterior-basic-properties} we have
\begin{align}
\langle \theta \cdot f_{\pi}, \theta \cdot \xi \cdot \theta_S \cdot \xi_T \rangle &= 
\pm \langle \theta \cdot (\xi \odot F_{\pi}), \theta \cdot \xi \cdot \theta_S \cdot \xi_T \rangle \\
&= \mp \langle \xi \odot( \theta \cdot  F_{\pi}), \theta \cdot \xi \cdot \theta_S \cdot \xi_T \rangle \\
&= \mp \langle  \theta \cdot  F_{\pi}, \xi \cdot \theta \cdot \xi \cdot \theta_S \cdot \xi_T \rangle  \\
&= 0
\end{align}
where the last line used $\xi^2 = 0$.

We turn to the argument that 
$\langle \theta \cdot f_{\pi}, \theta \cdot \delta \cdot \theta_S \cdot \xi_T \rangle = 0$ for any
 $S, T \subseteq [n]$.  We calculate
 \begin{align}
 \langle \theta \cdot f_{\pi}, \theta \cdot \delta \cdot \theta_S \cdot \xi_T \rangle &=
\pm \langle \theta \cdot (\xi \odot F_{\pi}), \theta \cdot \delta \cdot \theta_S \cdot \xi_T \rangle \\
&= \mp  \langle  \xi \odot (\theta \cdot F_{\pi}), \theta \cdot \delta \cdot \theta_S \cdot \xi_T \rangle \\
&= \mp  \langle  \theta \cdot F_{\pi},  \xi \cdot \theta \cdot \delta \cdot \theta_S \cdot \xi_T \rangle \\
&= \mp \langle  F_{\pi},  \xi \cdot (\theta \odot \delta) \cdot \theta \cdot \theta_S \cdot \xi_T \rangle 
 \pm \langle  F_{\pi},  \xi \cdot \delta \cdot \theta \odot (\theta \cdot  \theta_S \cdot \xi_T) \rangle  \\
 &=  \pm \langle  F_{\pi},  \xi \cdot \delta \cdot \theta \odot (\theta \cdot  \theta_S \cdot \xi_T) \rangle 
 \end{align}
 where the last line used $\theta \odot \delta = \xi$ and $\xi^2 = 0$.
The Claim is therefore reduced to showing 
\begin{equation}
\label{target-equation}
\langle  F_{\pi},  \xi \cdot \delta \cdot \theta \odot (\theta \cdot  \theta_S \cdot \xi_T) \rangle = 0
\end{equation}
for any $S, T \subseteq [n]$.  We prove \eqref{target-equation} by verifying the stronger claim that 
\begin{equation}
\label{new-target-equation}
\langle  F_{\pi},  \xi \cdot \delta \cdot  \theta_S \cdot \xi_T \rangle = 0
\end{equation}
for any $S, T \subseteq [n]$.

To see why Equation~\eqref{new-target-equation} holds, suppose that $F_{\pi}$ and
$\xi \cdot \delta \cdot \theta_S \cdot \xi_T$ have any monomials in common. 
From the definition of $F_{\pi}$, the sets $S, T$ must satisfy the following three conditions.
\begin{enumerate}
\item  For exactly one block $B_1$ of $\pi$, we have $B_1 \cap T = \varnothing$ and $|B_1 - S| = 2$.
\item  For exactly one block $B_2$ of $\pi$, we have $B_2 \cap T = \varnothing$ and $|B_2 - S| = 1$.
\item  For every other block $B$ of $\pi$, we have $|B \cap T| = 1, |B - S| = 1$, and 
$B \cap T \subset B \cap S$ whenever $B$ is not a singleton.
\end{enumerate}
If $S, T$ satisfy (1) - (3), the monomials shared between $F_{\pi}$ and
 $\xi \cdot \delta \cdot \theta_S \cdot \xi_T$ have the form $\xi_i \theta_j \xi_j \cdot m$
 where $i \in B_2 \cap S$ (or $i \in B_2$ if $B_2$ is a singleton) and $j \in B_1 - S$ and $m$ is a monomial.
 If $B_1 - S = \{j_1, j_2 \}$, then $\xi_i \theta_{j_1} \xi_{j_1} m$ and
 $\xi_i \theta_{j_2} \xi_{j_2} m$ appear in $F_{\pi}$ with opposite sign.
 We conclude that $\langle  F_{\pi},  \xi \cdot \delta \cdot  \theta_S \cdot \xi_T \rangle = 0$,
 completing the proof of the Claim.
 
 We use the Claim to prove the result.
 By Equation~\eqref{diagonal-component-dimensions} and Theorem~\ref{upstairs-basis-theorem},
 we have 
 \begin{equation}
 \dim \, V(n,k) = \Nar(n,k) = \dim \, (FDR_n)_{n-k,k-1}
 \end{equation}
 and by the definition of $U(n,k)$ we have
 \begin{equation}
 \dim \,  \wedge \{ \Theta_n, \Xi_n \}_{n-k,k-1} = \dim \, (FDR_n)_{n-k,k-1} + \dim U(n,k)
 \end{equation}
 These dimension equalities combine with the Claim to give a direct sum decomposition
 \begin{equation}
   \wedge \{ \Theta_n, \Xi_n \}_{n-k,k-1} = V(n,k) \oplus U(n,k)
 \end{equation}
 so that the composite
 \begin{equation}
 \label{composite-isomorphism}
 V(n,k) \hookrightarrow \wedge   \{ \Theta_n, \Xi_n \}_{n-k,k-1} \twoheadrightarrow (FDR_n)_{n-k,k-1}
 \end{equation}
 is a linear isomorphism.
Theorem~\ref{upstairs-basis-theorem} finishes the proof.
\end{proof}

As a corollary, we obtain our promised identification between the skein action 
and the extreme bidegree components of the fermionic diagonal coinvariants.

\begin{corollary}
\label{quotient-identification}
We have an isomorphism of $\symm_n$-modules
$\CC[\NC(n,k)] \cong (FDR_n)_{n-k,k-1}$ for all $1 \leq k \leq n$.
\end{corollary}

\begin{proof}
Apply the isomorphism \eqref{composite-isomorphism} and
Theorem~\ref{module-isomorphism}.
\end{proof}

Equation~\eqref{schur-kronecker}, 
Theorem~\ref{module-isomorphism},
Corollary~\ref{frobenius-image}, and Corollary~\ref{quotient-identification}
imply
\begin{align}
\sum_{m = 0}^k s_{(k-m,k-m,1^{n-2k+m})} \cdot s_{(1^m)} &= 
\Frob \,  \CC[\NC(n,k)] \\
&= \Frob  \, (FDR_n)_{n-k,k-1} \\
&= s_{(k, 1^{n-k})} * s_{(n-k-1,1^{k-1})} -
s_{(k-1, 1^{n-k+1})} * s_{(n-k-2, 1^{k+2})}
\end{align}
The symmetric function identity 
\begin{equation}
\label{hook-kronecker-difference}
\sum_{m = 0}^k s_{(k-m,k-m,1^{n-2k+m})} \cdot s_{(1^m)} = 
 s_{(k, 1^{n-k})} * s_{(n-k-1,1^{k-1})} -
s_{(k-1, 1^{n-k+1})} * s_{(n-k-2, 1^{k+2})}
\end{equation}
on the extreme ends of this chain of equalities 
relates an application of the dual Pieri rule to a difference of Kronecker products
of hook Schur functions.
It is possible (but tedious) to verify Equation~\eqref{hook-kronecker-difference} directly
using rules for the Schur expansion of $s_{\lambda} * s_{\mu}$ 
where $\lambda, \mu \vdash n$ are hook shapes.
One such rule, due to Rosas \cite{Rosas}, implies that whenever a $\symm_n$-irreducible
$S^{\lambda}$ appears in $\CC[\NC(n,k)]$, we must have $\lambda_3 < 3$.
Rosas's rule also implies that the multiplicities of any irreducible $S^{\lambda}$ 
in $\CC[\NC(n,k)]$ lies in the set $\{0, 1, 2\}$.

\section{Conclusion}
\label{Conclusion}

In Section~\ref{Resolution} we constructed a linear projection $p: \CC[\Pi(n)] \twoheadrightarrow \CC[\NC(n)]$
which resolves crossings in set partitions by means of the rank $2n$ exterior algebra
$\wedge \{ \Theta_n, \Xi_n \}$. It is natural to ask whether this crossing resolution is applicable
more broadly.
We will make the following vague problem more precise after its statement.

\begin{problem}
\label{applications-problem}
Find instances of the crossing resolution $p$, or the skein relations in 
Figure~\ref{fig:skein}, in other mathematical contexts.
\end{problem}

%Consider the ideal $I$ introduced in Section~\ref{Resolution}.
%Depending on the answer to Question~\ref{I-is-prime}, Problem~\ref{applications-problem} could 
%have a geometric interpretation. 
%Namely, our skein relations would generate the ideal of a closed projective variety $V$ if $I$ is radical, and if %$I$ is prime the variety $V$ would be irreducible.
%We could regard $V$ as a `skein analog' of the flag variety.

A classical application of the two-term skein relation is as follows.
Let $X$ be a $2 \times n$ matrix of variables and $\CC[X]$ is the polynomial ring in these
variables. The special linear group $\mathrm{SL}_2$ 
acts on the rows of $X$, and hence on the ring $\CC[X]$, by linear substitutions.
The invariant subring $\CC[X]^{\mathrm{SL}_2}$ is generated
 by the minors 
$\{ \Delta_{ab} \, : \, 1 \leq a < b \leq n \}$ and the syzygy ideal of relations among these generators 
is generated by the Pl\"ucker relations
\begin{equation}
\label{minor-syzygy}
\Delta_{ac} \Delta_{bd} = \Delta_{ab} \Delta_{cd} + \Delta_{ad} \Delta_{bc} \quad \quad
1 \leq a < b < c < d \leq n
\end{equation}
The standard mnemonic for Equation~\eqref{minor-syzygy} is the basic skein relation
\begin{center}

\begin{tikzpicture}[scale=0.5]

    \newcommand{\squa}
        {
        \foreach \i in {1,...,4}
                {
                \coordinate (p\i) at (90*\i + 45:\r);
                \filldraw(p\i) circle (\pointradius pt);
                }
        \node at (0 + 45:1.5*\r) {$a$};
         \node at (-90 + 45:1.5*\r) {$b$};
          \node at (-180 + 45:1.5*\r) {$c$};
           \node at (90 + 45:1.5*\r) {$d$};
        }
            
    \def\r{1}           % pentagon radius
    \def\pointradius{3} % points radius

    \foreach \a/\b/\c/\d [count=\shft from 0] in   
       {1/3/2/4,
        1/2/3/4,
        1/4/2/3}
        {
        \begin{scope}[xshift=120*\shft,yshift=0]
            \squa
            \foreach \i in {1,...,4}
                \draw[thick] (p\a) -- (p\b);
                \draw[thick] (p\c) -- (p\d);
            \squa
        \end{scope}
        }

	\node at (2.0,0) {$=$};
	\node at (6.3,0) {$+$};

    \end{tikzpicture}
    
    \end{center}
    Kung and Rota \cite{KungRota} gave a detailed combinatorial study of this invariant
    ring which has seen representation-theoretic application (e.g. \cite{RhoadesPoly, RT}).

 	Patrias, Pechenik, and Striker \cite{PPS} gave an analogous invariant-theoretic
	interpretation of the three-term and four-term skein relations as follows. 	
	Let $X$ be an $m \times n$ matrix of variables when $m = \ell + 2$ and let 
	$\CC[X]$ be the polynomial ring in these variables.
	Let $P \subseteq \mathrm{GL}_m(\CC)$ be the parabolic subgroup of matrices of 
	block form 
	\begin{equation*}
	\begin{pmatrix} A & 0 \\ B & C \end{pmatrix}
	\end{equation*}
	where $A$ is $2 \times 2$ and $C$ is $\ell \times \ell$ and consider the invariant
	subring $\CC[X]^P$.  This is the homogeneous coordinate ring of the two-step
	flag variety $\mathrm{Fl}(2, \ell)$ in $\CC^n$.
	The authors of 
	\cite{PPR}
	use matrix minors to define elements of $\CC[X]^P$ indexed by set partitions 
	which satisfy the three-term and four-term skein relations.
	It may be useful to consider more general flag varieties in the context of skein theory.

    %In the geometric and invariant theoretic settings above, the Pl\"ucker relations
    %involve products $\Delta_A \Delta_B$ indexed by sets $A, B$ which can overlap.
    %Towards Problem~\ref{applications-problem}, 
    %it may be helpful to define an appropriate notion of `crossing' for potentially
    %overlapping sets $A, B$ and define a crossing resolution in this context.
    %One idea here is to look at the kernel of the representation 
    %$R \rightarrow \mathrm{End} ( \wedge \{ \Theta_n, \Xi_n \} )$ given by 
    %$y_B \mapsto \rho_B$ in the proof of Proposition~\ref{quadratic-basis}.

In this paper we gave a combinatorial interpretation of the bidegree components
$(FDR_n)_{i,j}$ of $FDR_n$ where $i + j = n-1$ is maximal.  It is natural to ask
about bidegrees $(FDR_n)_{i,j}$. 
Future work of the first author will give such an interpretation
related to  {\em noncrossing $(1,2)$-configurations}
(see \cite{Thiel}).
Another possible extension is as follows.

\begin{problem}
\label{type-w}
Extend the skein action from type A to a wider class of reflection groups.
\end{problem}

As mentioned in the introduction, fermionic quotients suggest an avenue for 
 Problem~\ref{type-w}. More precisely, let $W$ be an irreducible complex reflection group 
of rank $n$ acting
on its reflection representation $V \cong \CC^n$.  The action of $W$ on $V$ induces actions of $W$ on 
\begin{itemize}
\item the $n$-dimensional dual space $V^*$,
\item the $2n$-dimensional direct sum $V \oplus V^*$, and 
\item the $2^{2n}$-dimensional exterior algebra $\wedge (V \oplus V^*)$.
\end{itemize}
Kim and Rhoades defined \cite{KR} the {\em $W$-fermionic diagonal coinvariant ring} to be the quotient
\begin{equation}
FDR_W := \wedge (V \oplus V^*)/I
\end{equation}
where $I$ is the (two-sided) ideal in $\wedge (V \oplus V^*)$ generated by the $W$-invariants with vanishing
constant term. By placing $V$ in bidegree $(1,0)$ and $V^*$ in bidegree $(0,1)$, the quotient $FDR_W$
attains the structure of a bigraded $W$-module.

If $\theta_1, \dots, \theta_n$ is a basis of $V$ and $\xi_1, \dots, \xi_n$ is the dual basis of $V^*$, 
Kim and Rhoades proved \cite{KR} that $FDR_W$ may be modeled as 
\begin{equation}
\label{principal-quotient}
FDR_W \cong \wedge \{ \Theta_n, \Xi_n \} / \langle \delta \rangle
\end{equation}
where $\delta = \theta_1 \xi_1 + \cdots + \theta_n \xi_n$.
The bidegree component $(FDR_W)_{i,j}$ is nonzero if and only if $i + j \leq n$, and 
for $i + j \leq n$ we have
$\dim \,  (FDR_W)_{i,j} = {n \choose i} {n \choose j} - {n \choose i-1} {n \choose j-1}$
so that for any $0 \leq k \leq n$ we have
\begin{equation}
\dim \,  (FDR_W)_{n-k,k} = \Nar(n+1,k+1)
\end{equation}
and a plausible solution to Problem~\ref{type-w} could involve the quotient ring \eqref{principal-quotient}
and noncrossing partitions of $[n+1]$.
One thing to note about such a solution is that it would involve an action of $W$ on a space spanned
by classical (type A) noncrossing partitions rather than the $W$-noncrossing partitions studied in 
Coxeter-Catalan theory, giving an action of type $W$ on type A.

The presentation \eqref{principal-quotient} holds somewhat beyond the realm of finite complex reflection 
groups.
Indeed, the identification \eqref{principal-quotient} holds whenever $W$ is a subgroup
of $\GL(V)$ such that the nonzero exterior powers 
$\wedge^0 V, \wedge^1 V, \dots, \wedge^n V$  of $V$ are pairwise nonisomorphic $W$-irreducibles \cite{KR}.
One such subgroup $W$ is $\GL(V)$ itself.
It may be interesting to use \eqref{principal-quotient} to analyze Problem~\ref{type-w} in this broader context.

Returning to the symmetric group, 
various authors \cite{Bergeron, DIW, OZ, PRR} have considered a `multidiagonal' version of the 
fermionic coinvariants defined as follows. Consider an $k \times n$ matrix $\Theta$ 
of fermionic variables $\theta_{i,j}$ where $1 \leq i \leq k$ and $1 \leq j \leq n$. Let 
$\wedge \{ \Theta \}$ be the exterior algebra over these variables, a $\CC$-vector space 
of dimension $2^{n  k}$.
The column permuting action of $\symm_n$ on the matrix $\Theta$ induces an action of 
$\symm_n$ on $\wedge \{ \Theta \}$.  The {\em  fermionic multidiagonal coinvariant
ring} is the quotient
\begin{equation}
FDR(n;k) := \wedge \{ \Theta \} / I
\end{equation}
where $I \subset \wedge \{ \Theta \}$ is the ideal generated by $\symm_n$-invariants with 
vanishing constant term. We have $FDR(n;2) = FDR_n$.   In general $FDR(n;k)$
is a $k$-fold graded $\symm_n$-module. Orellana and Zabrocki \cite{OZ} found generators 
for the ideal $I$, as well as a combinatorial formula for the multigraded Frobenius image 
of the numerator $\wedge \{ \Theta \}$.

\begin{problem}
\label{multigraded-problem}
Give a combinatorial interpretation for the multigraded pieces of $FDR(n;k)$.
\end{problem}

This paper (and future work of the first author) address Problem~\ref{multigraded-problem} when $k = 2$.
For $k > 2$, one complicating feature in Problem~\ref{multigraded-problem} is that the supporting
multidegrees of $FDR(n;k)$ are unknown. Specifically, if $\mathbf{i} = (i_1, \dots, i_k) \in \ZZ_{\geq 0}^k$
it is conjectured \cite{DIW} that 
$FDR(n,k)_{\mathbf{i}} \neq 0$ whenever $i_1 + \cdots + i_k < n$, but when $k > 2$ 
there exist tuples $\mathbf{i}$
with $i_1 + \cdots + i_k \geq n$ and $FDR(n,k)_{\mathbf{i}} \neq 0$.

One clue that Problem~\ref{multigraded-problem} should have an interesting solution is a 
conjecture of 
F. Bergeron \cite{Bergeron}. The column action of $\symm_n$ on $\Theta$ 
commutes with the row action of $\GL_k$, so 
$\mathcal{G}(n;k) := \symm_n \times \GL_k$ acts on 
$\wedge \{ \Theta \}$ and has an induced action on $FDR(n;k)$.
The {\em $\mathcal{G}(n;k)$-character} of this module is 
\begin{equation}
\label{gnk-char}
\mathrm{ch}_{\mathcal{G}(n;k)}  FDR(n;k) := 
\sum_{\mathbf{i}}  \Frob   \, FDR(n;k)_{\mathbf{i}} \otimes (q_1^{i_1} \cdots q_k^{i_k})
\end{equation}
where the sum ranges over all $\mathbf{i} = (i_1, \dots, i_k) \in \ZZ_{\geq 0}^k$ 
and $q_1, \dots, q_k$ are variables.
This character lies in $\Lambda(\xx) \otimes \CC[q_1, \dots, q_k]^{\symm_k}$.

Bergeron proved \cite{Bergeron} that the limit as $k \rightarrow \infty$ of 
the character \eqref{gnk-char}
\begin{equation}
\mathcal{F}_n := \lim_{k \rightarrow \infty} \mathrm{ch}_{\mathcal{G}(n;k)}  FDR(n;k)
\end{equation}
is a well-defined element of $\Lambda(\xx) \otimes \Lambda(\mathbf{q})$
where $\mathbf{q} = (q_1, q_2, \dots )$.
 Bergeron defined \cite{BergeronMulti} another element
 $\mathcal{E}_n \in \Lambda(\xx) \otimes \Lambda(\mathbf{q})$
in the same 
way as $\mathcal{F}_n$, but 
using matrices of {\em commuting} variables and conjectured \cite{Bergeron}
that
\begin{equation}
\mathcal{E}_n = \omega_{\mathbf{q}}  \,  \mathcal{F}_n
\end{equation}
where the  $\omega$-involution acts on $\mathbf{q}$-variables alone.
That is, fermionic multidiagonal coinvariants are expected
to ``contain all the information of" the commuting multidiagonal coinvariants.

\section{Acknowledgements}

B. Rhoades was partially supported by NSF Grants DMS-1500838 and
DMS-1953781.
The authors are grateful to Alessandro Iraci,
 Jongwon Kim, Rebecca Patrias, Oliver Pechenik,
  Steven Sam, and Jessica Striker for helpful conversations.
The authors thank the SageMath team for their work which assisted greatly
in this project.


\begin{thebibliography}{99}
 
 
 \bibitem{BergeronMulti} F. Bergeron.
 Multivariate Diagonal Coinvariant Spaces for Complex Reflection Groups.
 {\em Adv. Math.}, {\bf 239} (2013), 97--108.
  
   \bibitem{Bergeron} F. Bergeron.  
   The bosonic-fermionic diagonal coinvariant modules conjecture.
  Preprint, 2020.
 {\tt arXiv:2005.00924}.
 
 \bibitem{BRT} S. Billey, B. Rhoades, and V. Tewari.
Boolean product polynomials, Schur positivity, and Chern plethysm.
{\em Int. Math. Res. Not. IMRN}, (2019), rnz261,
 {\tt https://doi.org/10.1093/imrn/rnz261}
 
 
  \bibitem{DIW} M. D'Adderio, A. Iraci, and A. Wyngaerd.
Theta operators, refined Delta conjectures, and coinvariants.
{\em Adv. Math.}, {\bf 376}, (2021), 107477.

\bibitem{Fulton} W. Fulton.
{\em Young Tableaux}.
London Mathematical Society Student Texts, no. 35.
Cambridge University Press, 1997.

\bibitem{Haiman} M. Haiman. 
Conjectures on the quotient ring by diagonal invariants.
{\em J. Algebraic Combin.}
 {\bf 3}, (1994), 17--76.
 
 \bibitem{Iraci} A. Iraci.
 {\em Cyclic Sieving for Noncrossing Partitions.}
 Master's Thesis, 
 Universit\`a di Pisa, 2016.
 


\bibitem{KR} J. Kim and B. Rhoades. Lefschetz theory for exterior algebras and fermionic diagonal
coinvariants.
{\em Int. Math. Res. Not. IMRN}, (2020), rnaa203,
{\tt https://doi.org/10.1093/imrn/rnaa203}.

\bibitem{KungRota} J. P. S. Kung and G.-C. Rota. The invariant theory of binary forms.
{\em Bull. Amer. Math. Soc.}, {\bf 19 (1)} (1984), 27--85.


\bibitem{OZ} R. Orellana and M. Zabrocki. A combinatorial model for the decomposition of multivariate
polynomial rings as an $S_n$-module.
{\em Electron. J. Combin.}, {\bf 27 (3)}, (2020), P3.24.


\bibitem{PPS} R. Patrias, O. Pechenik, and J. Striker.
A web basis of invariant polynomials from noncrossing partitions.
In preparation, 2021.

\bibitem{PRR} B. Pawlowski, E. Ramos, and B. Rhoades.
Spanning line configurations and representation stability.
Preprint, 2019. {\tt arXiv:1907.07268}.


\bibitem{Pechenik} O. Pechenik. Cyclic sieving of increasing tableaux and small Schr\"oder paths.
{\em J. Combin. Theory Ser. A}, {\bf 125}, (2014), 357--378.

\bibitem{PPR} K. Petersen, P. Pylyavskyy, and B. Rhoades.
Promotion and cyclic sieving via webs.
{\em J. Algebraic Combin.}, {\bf 30}, (2009), 19--41.



\bibitem{RSW} V. Reiner, D. Stanton, and D. White.
The cyclic sieving phenomenon.
{\em J. Combin. Theory Ser. A}, {\bf 108}, (2004), 17--50.


\bibitem{Rhoades} B. Rhoades. A skein action of the symmetric group on noncrossing partitions.
{\em J. Algebraic Combin.}, {\bf 45 (1)}, (2017), 81--127.

\bibitem{RhoadesPoly} B. Rhoades. The polytabloid basis expands positively into the web basis.
{\em Forum Math. Sigma}, {\bf 7}, (2019), e26.


\bibitem{RW} B. Rhoades and A. T. Wilson. 
Vandermondes in superspace.
{\em Trans. Amer. Math. Soc.}, {\bf 373 (6)}, (2020), 4483--4516.

\bibitem{RW2} B. Rhoades and A. T. Wilson.
Set superpartitions and superspace duality modules.
Preprint, 2021.
{\tt arXiv:2104.05630}.


\bibitem{Rosas} M. Rosas. The Kronecker product of Schur functions indexed by two-row
shapes or hook shapes. 
{\em J. Algebraic Combin.}, {\bf 14 (2)}, (2001), 153--173.


\bibitem{RT} H. Russell and J. Tymoczko. The transition matrix between the Specht and web bases is 
unipotent with additional vanishing entries.
{\em Int. Math. Res. Not. IMRN}, 2019 (5), 1479--1502.


\bibitem{Swanson} J. Swanson. Tanisaki witness relations for harmonic differential forms.
In preparation, 2021.

\bibitem{SW} J. Swanson and N. Wallach. Harmonic differential forms for pseudo-reflection groups I.
Semi-invariants. To appear,
{\em J. Combin. Theory Ser. A}, 2021.
{\tt arXiv:2001.06076}.


\bibitem{Thiel} M. Thiel.
A new cyclic sieving phenomenon for Catalan objects.
{\em Discrete Math.}, 
{\bf 340 (3)}, (2017), 426--429.

\bibitem{ZabrockiDelta} M. Zabrocki. A module for the Delta conjecture.
Preprint, 2019. 
{\tt arXiv:1902.08966}.


\bibitem{ZabrockiFermion} M. Zabrocki. Coinvariants and harmonics.  Blog for
Open Problems in Algebraic Combinatorics 2020.
{\tt https://realopacblog.wordpress.com/2020/01/26/coinvariants-and-harmonics/}.




\end{thebibliography}
\end{document}